\documentclass[10pt]{article}

\usepackage{amsmath,amssymb,amsfonts,mathrsfs}
\usepackage[mathscr]{euscript}

\usepackage{amsthm}
\usepackage{geometry}
\usepackage[backend=biber,style=alphabetic]{biblatex}
\DeclareFieldFormat*{title}{\textit{#1}}
\DeclareFieldFormat*{booktitle}{\textit{#1}}

\DeclareFieldFormat*{journaltitle}{\textup{#1}}
\usepackage{tikz-cd}
\usepackage{microtype}
\usepackage{xcolor}

\definecolor{internalblue}{HTML}{0000FF}
\definecolor{citationgreen}{HTML}{00A000}

\usepackage{hyperref}

\hypersetup{
	colorlinks=true,
	linkcolor=internalblue,
	citecolor=citationgreen,
	urlcolor=internalblue
}

\newtheorem{theorem}{Theorem}[section]
\newtheorem{lemma}[theorem]{Lemma}
\newtheorem{proposition}[theorem]{Proposition}
\newtheorem{corollary}[theorem]{Corollary}

\theoremstyle{definition}
\newtheorem{definition}[theorem]{Definition}
\newtheorem{construction}[theorem]{Construction}
\newtheorem{remark}[theorem]{Remark}
\newtheorem{example}[theorem]{Example}
\newtheorem{notation}[theorem]{Notation}
\newtheorem{observation}[theorem]{Observation}
\newtheorem{question}[theorem]{Question}
\newtheorem{convention}[theorem]{Convention}
\newtheorem{recall}[theorem]{Recall}
\newtheorem{simplification}[theorem]{Simplification}
\newtheorem{restrict}[theorem]{Restriction}

\newcommand{\dR}{\mathrm{dR}}
\newcommand{\CAlg}{\mathsf{CAlg}}
\newcommand{\Alg}{\mathsf{Alg}}
\newcommand{\Ar}{\mathsf{Ar}}
\newcommand{\Aff}{\mathsf{Aff}}

\newcommand{\PStk}{\mathsf{PStk}}
\newcommand{\Stk}{\mathsf{Stk}}
\newcommand{\ArtStk}{\mathsf{ArtStk}}
\newcommand{\dArtStk}{\mathsf{dArtStk}}
\newcommand{\Cat}{\mathsf{Cat_{\infty}}}
\newcommand{\op}{\mathsf{op}}
\newcommand{\an}{\mathsf{an}}
\newcommand{\Fun}{\mathsf{Fun}}
\newcommand{\Mod}{\mathsf{Mod}}
\newcommand{\gr}{\operatorname{gr}}
\newcommand{\Gr}{\mathsf{Gr}}
\newcommand{\Fil}{\mathsf{Fil}}
\newcommand{\CRing}{\mathsf{CRing}}
\newcommand{\Grp}{\mathsf{Grp}}
\newcommand{\Set}{\mathsf{Set}}
\newcommand{\cofib}{\operatorname{cofib}}

\newcommand{\QCoh}{\mathsf{QCoh}}
\newcommand{\syn}{\mathsf{syn}}

\newcommand{\Ani}{\mathsf{Ani}}

\newcommand{\pr}{\operatorname{pr}}
\newcommand{\DGA}{\mathsf{DGA}}

\DeclareMathOperator{\wt}{wt}
\DeclareMathOperator*{\colim}{colim}
\DeclareMathOperator{\Hom}{Hom}

\renewbibmacro*{volume+number+eid}{%
  \printfield{volume}%
  \iffieldundef{number}
    {}
    {\printtext{\mkbibparens{\printfield{number}}}}%
  \setunit{\addcomma\space}%
  \printfield{eid}}

\title{The Hodge-to-de Rham spectral sequence can be arbitrarily non-degenerate
	in positive characteristic}
\author{Kehao Cheng}
\date{\today}

\begin{document}

\maketitle
\begin{abstract}
   There is a long tradition of constructing smooth projective varieties in positive
   characteristic with non-degenerate Hodge-to-de Rham spectral sequences.
   In this paper, we show that for any \(n\geq1\) and any spot $(i, j)$ on the $E_n$-page,
   except for the cases where the differential is forced to be zero for trivial reasons,
   there exists a smooth projective variety whose Hodge-to-de Rham spectral sequence
   has a non-zero differential with source $E_n^{i, j}$. Our construction is inspired by constructions of nilmanifolds in complex geometry and can be regarded as a positive characteristic analogue.
\end{abstract}
\section{Introduction}

\subsection{Complex geometry background}
In complex geometry, the starting point of Hodge theory is the following celebrated result:

\begin{theorem}[Hodge~\cite{Hod51}; Fr\"olicher~\cite{Fro55}]
	Let X be a compact K\"{a}hler manifold. Then the Fr\"{o}licher spectral sequence
\[ E_1^{i,j} = H^j(X,\Omega_X^i) \Longrightarrow H_{\mathrm{dR}}^{i+j}(X,\mathbb C) \]
degenerates on first page.
\end{theorem}

One natural question is whether, given \(n\geq1\) and \(i,j\geq0\),  there exists a compact complex manifold $X$ (necessarily non-K\"{a}hler) whose Fr\"{o}licher spectral sequence has a non-zero differential $d_n: E_n^{i,j}\to E_n^{i+n,j-n+1}$.
It is clear that the answer is negative if $i = j = 0$ or $j<n-1$ for trivial reasons.
On the positive side, for every \(n\geq 1\) and \(i,j\geq 0\) satisfying \(j\geq n-1\) and \((i,j)\neq(0,0)\), Bigalke--Rollenske \cite{Rol08, BR14} constructed a compact complex manifold whose Fr\"{o}licher spectral sequence has a non-zero differential $d_n: E_n^{i,j}\to E_n^{i+n,j-n+1}$.

\subsection{Background and results in positive characteristic algebraic geometry}

Let $k$ be a perfect field of characteristic \(p>0\). It has long been recognized that smooth projective varieties in positive characteristic can have non-degenerate Hodge-to-de Rham spectral sequences:
\begin{example} 
	\begin{enumerate} 
	\item[\textup{(1)}] Mumford constructed a smooth projective surface \(X\) over \(k\) and a global regular \(1\)-form $ \omega\in H^0(X,\Omega^1_X) $ such that $ d\omega\neq 0$ in \cite{Mum61}. In particular, the Hodge-to-de Rham spectral sequence of \(X\) does not degenerate at the \(E_1\)-page.
	\item[\textup{(2)}] Suppose that \(k\) has characteristic \(2\), and let
	\(X\) be a supersingular Enriques surface over \(k\). In \cite[Proposition~II.7.3.8]{Ill79}, Illusie  showed the differential
	\[
	d_1\colon
	H^1(X,\mathcal O)
	\longrightarrow
	H^1(X,\Omega_{X/k}^1)
	\]
	in the Hodge-to-de Rham spectral sequence of \(X\)  is non-zero.
	\item[\textup{(3)}] William Lang studied the Hodge and de Rham cohomology of various classes of smooth projective surfaces in positive characteristic, including quasi-elliptic surfaces, quasi-hyperelliptic surfaces, hyperelliptic surfaces, Enriques surfaces, and Raynaud surfaces; see \cite{Lan79,Lan80,Lan95}. Certain examples among these surfaces provide further instances whose Hodge-to-de Rham spectral sequences do not degenerate at the \(E_1\)-page.
	\end{enumerate} 
\end{example}

On the other hand, Deligne--Illusie proved the following degeneration theorem.

\begin{theorem}[{\cite[Corollaries~2.3 and~2.4]{DI87}}]
	Let \(X\) be a smooth proper scheme over \(k\). If \(X\) admits a \(W_2(k)\)-lifting
	and \(\dim X\leq p\), then the Hodge-to-de Rham spectral sequence of \(X\)
	degenerates at the \(E_1\)-page and the conjugate spectral sequence of \(X\)
	degenerates at the $E_2$-page.
\end{theorem}

Given all the advances in our understanding of the Hodge-to-de Rham spectral sequence
in positive characteristic, one natural question remains open.
\begin{question}[Petrov, \cite{MO259618}]
 For every \(n\geq 1\), does there exist a smooth proper variety \(X\) over \(k\) whose
Hodge-to-de Rham spectral sequence has a non-zero differential on the
\(E_n\)-page?
\end{question}

Inspired by this question, we consider the following refinement.

\begin{question}
	 Given \(n\geq 1\) and
	\(i,j\geq 0\), does there exist a smooth proper variety \(X\) over \(k\)
	whose Hodge-to-de Rham spectral sequence has a non-zero differential $d_n\colon E_n^{i,j}\longrightarrow E_n^{i+n,j-n+1}$?
\end{question}

Let us observe again that the answer is of course negative if \(i=j=0\) or \(j<n-1\). Our main theorem in this paper confirms that, outside these trivial cases, the answer to the above question is yes.

\begin{theorem}[Main result, Corollary~\ref{cor}]\label{Main}
	Given \(n\geq 1\) and \(i,j\geq 0\) satisfying   $ j\geq n-1$ and
	\((i,j)\neq(0,0)\), there exists a smooth projective variety $X$ over $k$ and a class $\beta  \in H^{j}(X,\Omega^i_{X/k}) =E_1^{i,j}$ in the Hodge-to-de Rham spectral sequence of \(X\)
	which satisfies $d_r(\beta)=0 $ for $1\leq r<n$ and $d_{n}(\beta)\neq 0$ in $E^{n+i,j-n+1}_n$.
\end{theorem}

\begin{remark}
	To the best of our knowledge, prior to the present work, the only example among smooth projective varieties in positive characteristic we have found in the literature for which the Hodge-to-de Rham spectral sequence is known to have a non-zero differential beyond the first page is Hirokado's Calabi-Yau threefold $\mathcal H$ in characteristic $3$; see \cite{Hir99}. Indeed, the Hodge-to-de Rham spectral sequence of $\mathcal H$ has the following non-zero differentials
	\[
	d_2^{0,3}\colon H^3(\mathcal H,\mathcal O)\to H^2(\mathcal H,\Omega_{\mathcal H/k}^2),
	\qquad
	d_2^{1,1}\colon H^1(\mathcal H,\Omega_{\mathcal H/k}^1)\to H^0(\mathcal H,\Omega_{\mathcal H/k}^3).
	\]
	See Remark~\ref{CY} for details.
\end{remark}
\subsection{Overview of the strategy}
Let us discuss our proof strategy. A useful approach to such questions is to study the Hodge-to-de Rham spectral sequences of classifying stacks. In \cite{Tot18}, Totaro  systematically studied the Hodge cohomology and de Rham cohomology of classifying stacks, providing a rich source of examples of Hodge-to-de Rham spectral sequences. Subsequently, Antieau--Bhatt--Mathew computed  the Hodge-to-de Rham spectral sequence of \(B\alpha_p\). In particular, they proved the following result.
\begin{theorem}[{\cite[Proposition~4.12]{ABM21}}]
	The Hodge-to-de Rham spectral sequence of 
	\(B\alpha_p\)
	\[
	E_1^{i,j}
	=
	H^j(B\alpha_p, \Lambda^iL_{B\alpha_p/k})
	\Longrightarrow
	H_{\mathrm{dR}}^{i+j}(B\alpha_p/k)
	\]
	does not degenerate at the \(E_1\)-page. More precisely, the 
	differential
	\[
	d_1:
	H^1(B\alpha_p,\mathcal O_{B\alpha_p})
	\longrightarrow
	H^1(B\alpha_p,L_{B\alpha_p/k})
	\]
	is non-zero.
\end{theorem}
To transfer cohomological phenomena from classifying stacks to smooth projective schemes, a natural approach is to approximate classifying stacks by smooth projective schemes in a suitable sense. This idea  already appeared in Bhatt--Morrow--Scholze \cite[\S2.2]{BMS18}. Subsequently, Antieau--Bhatt--Mathew established an approximation theorem   for Hodge cohomology \cite[Theorem 1.2]{ABM21}, which provides a powerful method for constructing smooth projective schemes with non-degenerate Hodge-to-de Rham spectral sequences. For our purposes, we state the following refinement, which  is essentially contained in the argument of \cite[Theorem 1.2 and Proposition 5.11]{ABM21}.
\begin{theorem}[Antieau--Bhatt--Mathew]\label{abm1}
	Suppose \(G\) is a finite group scheme over $k$. For any integer \(d\geq 0\), there exists a smooth
	projective \(k\)-scheme \(X\) of dimension \(d\), together with a map
	\(X\to BG\) such that, writing \(E_r^{i,j}(Z)\) for the \((i,j)\)-term on the \(E_r\)-page of the Hodge-to-de Rham spectral sequence of \(Z\), the induced map between the Hodge-to-de Rham spectral sequences
	\[
	E_r^{i,j}(BG)
	\longrightarrow
	E_r^{i,j}(X)
	\]
	is injective for \(i+j\leq d\), for every $r\geq 1$.
\end{theorem}
It is worth noting that a similar approximation strategy was used by Alexander Petrov to construct a smooth projective variety $X_0$ of dimension $p+1$ over $k$ which admits a  lifting over $W(k)$, while it has a non-zero differential on the $E_{p+1}$-page of the conjugate spectral sequence.

\begin{theorem}[{\cite[Theorem~1.1~and~Corollary~10.7]{Pet23}}]
	There exists a smooth projective scheme \(X\) over \(W(k)\) of relative
	dimension \(p+1\) with special fiber  $	X_0 := X \times_{W(k)} k$, such that 
	the differential
	\[
	d_{p+1}^{0,p}\colon
	H^0(X_0^{(1)},\Omega^p_{X_0^{(1)}/k})
	\longrightarrow
	H^{p+1}(X_0^{(1)},\mathcal O)
	\]
	on the $E_{p+1}$-page of the conjugate spectral sequence of \(X_0\)
	is non-zero.
\end{theorem}

Inspired by these results, our strategy is to first construct a sequence of
classifying stacks with arbitrarily nondegenerate Hodge-to-de Rham spectral sequences. The main technical result of the paper is the following construction.

\begin{theorem}[Theorem~\ref{thm:main1}]
	For $n\geq 2$, there exists a non-commutative finite group scheme $G_n$ over $k$ and a class\[\alpha_{n} \in H^{n-1}(BG_{n},\mathcal{O}) =E_1^{0,n-1}\] in the Hodge-to-de Rham spectral sequence of $BG_n$
	which satisfies $d_1(\alpha_{n}),\cdots,d_{n-1}(\alpha_{n})=0$ and $d_{n}(\alpha_{n})\neq 0$ in $E_n^{n,0}$. 
\end{theorem}

Besides the family \(G_n\), we study  the Frobenius kernel \(H[F]\) of the Heisenberg group scheme \(H\). In the Hodge-to-de Rham spectral sequence of \(B(H[F])\), there  exists a class $u\in E^{1,0}_1$ which satisfies $d_1(u)\neq 0$ in $E^{2,0}_1$; for details, see Proposition~\ref{Heisenberg}.

Applying Theorem~\ref{abm1}  to $BG_{n}$, we obtain the following results.

\begin{theorem}[Theorem~\ref{thm:main 2}]

 For $n\geq 2$, there exists a smooth projective variety $X_n$ of dimension $n$ over $k$ and a class \[\alpha_{n} \in H^{n-1}(X_{n},\mathcal{O}) =E_1^{0,n-1}\] in the Hodge-to-de Rham spectral sequence of $X_n$ which satisfies $d_1(\alpha_{n}),\cdots,d_{n-1}(\alpha_{n})=0$ and $d_{n}(\alpha_{n})\neq 0$ in $E_n^{n,0}$.
\end{theorem}

Moreover, we can apply Theorem~\ref{abm1} to \(B(H[F])\) and \(B\alpha_p\). We obtain a smooth projective surface $S_1$ over $k$ with a global $1$-form $\omega\in H^0(S_1,\Omega^1_{S_1/k})=E_1^{1,0}$ such that $d_1(\omega)\neq 0$ in $H^0(S_1,\Omega^2_{S_1/k})=E^{2,0}_1$. We also obtain a smooth projective surface $S_2$ over $k$ with a class $\alpha\in H^1(S_2,\mathcal{O})=E_1^{0,1}$ such that $d_1(\alpha)\neq 0$ in $H^1(S_2,\Omega^1_{S_2/k})=E^{1,1}_1$; see Proposition~\ref{surface}.

By taking the product of one of $X_n$ $(n\geq 2)$, $S_1$, or $S_2$ with several copies of an arbitrary elliptic  curve, we obtain Theorem~\ref{Main}.

\begin{remark}
	Note that the nilmanifolds have homotopy types also given by $B\Gamma$
	where $\Gamma$ is a lattice subgroup in a unipotent real Lie group.
	We view our \(BG_n\)'s as the ``nilmanifolds'' in positive characteristic.
	Indeed, our $G_n$ is constructed as a closed subgroup scheme inside unipotent group schemes
	$\widetilde{G_n}$ cut out by $p$-powers of various coordinates,
	hence can also be regarded as ``lattice'' subgroups in a unipotent group scheme.
\end{remark}

\subsection{Construction idea and illustration}
Let us explain where the inspiration for our construction comes from: it rests
in understanding the theory of nilmanifolds in complex geometry.
Recall that a nilmanifold is a compact quotient $X=\Gamma\backslash G$ where \(G\) is a simply connected real nilpotent Lie group and \(\Gamma\subset G \) is a lattice subgroup. A left-invariant integrable complex structure on \(G\) descends to a complex structure on \(X\).
The group multiplication is designed so that certain \(\overline{\partial}\)-closed \((0,1)\)-forms exhibit non-trivial \(\partial\)-\(\overline{\partial}\) zig-zags. Let us illustrate what we mean by the following example.
\begin{example}[cf.~{\cite[Lemma~3.2]{CFU02}}]
	\label{Complex geometry example of H'}
Let \(H'=\mathbb C^3\). We embed \(H'\) into
\(\operatorname{GL}_3(\mathbb C)\) via
\[
(x,y,z)\longmapsto
\begin{pmatrix}
	1 & \overline{x} & y\\
	0 & 1 & z\\
	0 & 0 & 1
\end{pmatrix}.
\]
 The image is closed under matrix multiplication and inversion. Hence  \(H'\) inherits a group structure from \(\operatorname{GL}_3(\mathbb C)\).
 
Let \(\Gamma=\mathbb Z[i]^3\subset H'\) be a lattice subgroup, and set
\(X=\Gamma\backslash H'\). Consider the \((0,1)\)-form
\(\eta=d\overline{y}-x\,d\overline{z}\). One checks that \(\eta\) is
\(\Gamma\)-invariant and satisfies \(\overline{\partial}\eta=0\).
Therefore, the \((0,1)\)-form \(\eta\) descends to a \(\overline{\partial}\)-closed
\((0,1)\)-form on \(X\), and hence determines a class
\(
[\eta]\in H_{\overline{\partial}}^{0,1}(X)
\simeq H^1(X,\mathcal O)
\). A direct computation gives
\(\partial\eta=-dx\wedge d\overline{z}\). Thus
\(
d_1[\eta]=[-dx\wedge d\overline{z}]
\in H_{\overline{\partial}}^{1,1}(X)\simeq H^1(X,\Omega_X^1)
\).
We claim that this class is non-zero.  Indeed, one can choose a \(\overline{\partial}\)-closed $(2,2)$-form $\beta$ such that $
\int_X dx\wedge d\overline{z}\wedge\beta\neq 0
$. Therefore, Stokes' theorem implies that there is no 
\((1,0)\)-form \(\omega\) on \(X\) satisfying
\(\overline{\partial}\omega=-dx\wedge d\overline{z}\). In particular, the differential $d_1\colon H^1(X,\mathcal O)\to H^1(X,\Omega_X^1) $ on the \(E_1\)-page of the Fr\"{o}licher spectral sequence of \(X\) is non-zero.
\end{example}

Below, we shall explain a construction which we regard as a characteristic \(p\) analogue of the above example. First, let us explain the following heuristic: on a compact Kähler manifold \(X\), a class in \(H^1(X,\mathcal O)\) can be represented by a harmonic \( (0,1) \)-form. 
Complex conjugation turns it into a harmonic \((1,0)\)-form, 
which therefore represents a class in \(H^0(X,\Omega_X^1)\).

A very similar phenomenon happens in the Hodge theory of classifying stacks in characteristic $p$. To explain this,
suppose that $G$ is a closed subgroup scheme of a smooth group scheme $H$.
Then a regular function $f \in \mathcal{O}(H)$ gives rise to a class in $ H^1(BG,\mathcal O) \simeq \operatorname{Hom}_{\Grp}(G,\mathbb G_a)$
if $f|_G \in  \mathcal{O}(G)$ defines a group scheme homomorphism $G \to \mathbb{G}_a$. We denote this class by $\alpha_f$. On the other hand, if a regular function $g \in \mathcal{O}(H)$
vanishes on $G$, then it gives rise to an element in $\mathcal H^{-1}(\ell_G)$ under suitable conditions, where
$\ell_G$ is the co-Lie complex of $G$; see \cite[\S3.1]{Ill72a}, \cite{Ill72b}, \cite[\S4]{ABM21} and \cite[Example~A.3.8]{KP24} for the co-Lie complex formalism. Via the identification $L_{BG/k} \simeq \ell_G[-1]$
, the element $g$ gives rise to a global $1$-form on $BG$ under suitable conditions. We denote this global $1$-form by $\eta_g$. Suppose now that \(G\) is a connected finite group scheme and let
\(f\in\mathcal O(H)\) be a regular function whose restriction
\(f|_G\in\mathcal O(G)\) defines a group scheme homomorphism
\(G\to\mathbb G_a\). Raising $f|_G \in \mathcal{O}(G)$ to a suitable $p$-power 
 gives zero in $\mathcal{O}(G)$. Hence, the corresponding $p$-power of $f$ satisfies the condition above and gives rise to a global $1$-form on $BG$ under suitable conditions. Since raising elements to a \(p\)-power plays a role analogous to complex conjugation in characteristic $p$ geometry, the above should be viewed as our analogue of the complex geometry picture.

This vague heuristic is best illustrated through the following pair of examples:

\begin{example}[Proposition~\ref{conjugate}]
	\label{char p geometry example of H'}
	Note that $H'$ in the complex geometry example is obtained by introducing complex conjugation into the group law of the complex Heisenberg group. Let us try to
	mimic this construction in positive characteristic. It is natural to replace complex conjugation by a power of Frobenius, the question is which power? It turns out that the answer is \emph{any} positive power:
	
For a positive integer $r$, let $H^r=\mathbb{A}^3$  as a scheme. We equip $H^r$ with an affine group scheme structure over $k$ as follows. For a \(k\)-algebra \(R\) and  $P=(x,y,z), P'=(x',y',z')\in H^r(R)$, we define $P\cdot P'=\bigl(x+x',\,y+y'+x^{p^r}z',\,z+z'\bigr) $. For positive integers $a,b,c$ satisfying either \(b\geq c\) or \(b\geq a-r\), let $G^r_{a,b,c} $ be the closed subscheme of $ H^r$ cut out by the ideal $(x^{p^a},y^{p^b},z^{p^c})$. One can verify that it is a closed subgroup scheme of $H^r$. 

Suppose now that $b-1\geq a-r\geq c$. The regular function $f=y^{p^c}\in \mathcal{O}(H^r)$ satisfies $\pr_1^*(f) + \pr_2^*(f) - m^*(f) =- x^{p^{c+r}} \otimes z^{p^c}\in \mathcal{O}(H^r)\otimes \mathcal{O}(H^r)$. Since the regular function $y^{p^c}$ does not vanish identically on \(G^r_{a,b,c}\) and the regular function \(z^{p^c}\) vanishes on \(G^r_{a,b,c}\), the regular function $y^{p^c}$ gives rise to a non-zero class $\alpha_{y^{p^c}}\in H^1(BG^r_{a,b,c},\mathcal O)$ and the regular function $z^{p^c}$ gives rise to a class $\eta_{z^{p^c}}\in H^0(BG^r_{a,b,c},L_{BG^r_{a,b,c}/k})$. Indeed, one can show that the class \(\eta_{z^{p^c}}\) is non-zero in $H^0(BG^r_{a,b,c},L_{BG^r_{a,b,c}/k})$. When $a-r>c$, the regular function $x^{p^{c+r}}$ does not vanish identically on $ G^r_{a,b,c}$ and hence it gives rise to a non-zero class $\alpha_{x^{p^{c+r}}}\in H^1(BG^r_{a,b,c},\mathcal O)$. Then the above equality gives that $d_1(\alpha_{y^{p^c}})=-\alpha_{x^{p^{c+r}}}\otimes \eta_{z^{p^c}}\neq 0$ in $E_{1}^{1,1}=H^1(BG^r_{a,b,c},L_{BG^r_{a,b,c}/k})$. When $a-r=c$, the regular function $x^{p^{c+r}}$ vanishes on $ G^r_{a,b,c}$ and hence it gives rise to a class $\eta_{x^{p^{c+r}}}\in H^0(BG^r_{a,b,c},L_{BG^r_{a,b,c}/k})$, which is in fact non-zero. Then the above equality gives that  $ d_1(\alpha_{y^{p^c}})= 0$ and $d_2(\alpha_{y^{p^c}})=-\eta_{x^{p^{c+r}}}\wedge \eta_{z^{p^c}}\neq 0$ in $E^{2,0}_2$. 

\end{example}

\begin{example}[Proposition~\ref{Heisenberg2}]
Let $H$  be the Heisenberg group scheme. Let $H[F^{4/3}]$ be the closed subscheme of $H$ cut out by the ideal $(x^p,y^{p^2},z^p)$. One can verify that it is a closed subgroup scheme of $H$. The regular function $y^p\in \mathcal{O}(H)$ satisfies $\pr_1^*(y^p) + \pr_2^*(y^p) - m^*(y^p) =-x^p\otimes z^p\in \mathcal{O}(H)\otimes \mathcal{O}(H)$. Since the regular functions  $x^p$, $z^p\in \mathcal{O}(H)$ vanish on $H[F^{4/3}]$ and the regular function $y^p$ does not vanish identically on $H[F^{4/3}]$, the regular function $y^p$ gives rise to a non-zero class $\alpha_{y^p}\in H^1(BH[F^{4/3}],\mathcal O)$ and $x^p$, $z^p$ give rise to two classes $\eta_{x^p}$, $\eta_{z^p}\in H^0(BH[F^{4/3}],L_{BH[F^{4/3}]/k})$. One can show that both \(\eta_{x^p}\) and \(\eta_{z^p}\) are non-zero in $ H^0(BH[F^{4/3}],L_{BH[F^{4/3}]/k})$. The above equality gives that $d_1(\alpha_{y^p})=0$, and $d_2(\alpha_{y^p})=-\eta_{x^p}\wedge\eta_{z^p}\neq 0$ in $E_2^{2,0}$.

\end{example}

\section*{Notation and conventions}\label{sec:notation}

\begin{enumerate}
    \item We freely use the language of $\infty$-categories as developed in \cite{Lur09}.
    \item For a commutative ring \(R\), we denote by
    \(\mathsf D(R)\) the derived \(\infty\)-category of
    \(R\)-modules. All tensor products are implicitly derived.
    \item A prestack is an accessible functor $\CRing \rightarrow \Ani$  from the category of commutative rings to the \(\infty\)-category
    of anima (i.e. spaces or $\infty$-groupoids). A stack  is a prestack satisfying fppf descent. We write \(\PStk\) for the \(\infty\)-category of prestacks and \(\Stk\) for the \(\infty\)-category of stacks. We follow the definition of Artin stacks in \cite[Appendix~A.1]{KP24} and write
    $\ArtStk$ for the full subcategory of Artin stacks. For a finitely presented affine flat group scheme $G$ over a commutative ring $R$, we denote by $BG$ the classifying stack of $G$-torsors for the fppf topology, which is in fact an Artin stack by \cite[Tag~06PL]{Sta26}.
    \item For a prestack $X$, we denote by $\mathsf{D}(X)$ the derived \(\infty\)-category of quasi-coherent sheaves  on $X$, which is defined as the limit $ \displaystyle \lim_{\mathrm{Spec}(R)\to X}\mathsf{D}(R)$. The category \(\mathsf D(X)\) carries a natural \(t\)-structure, and we
    denote its heart by \(\QCoh(X)\), the abelian category of quasi-coherent sheaves
    on \(X\). For an object $M\in \mathsf{D}(X)$, we denote by $\mathcal{H}^i(M)\in \QCoh(X)$  the cohomology sheaf of $M$ in degree $i$. When $X=\operatorname{Spec}(A)$ is affine, for $M\in \mathsf{D}\bigl(\operatorname{Spec}(A)\bigr)=\mathsf{D}(A)$, we  write $H^i(M)\in \Mod(A)$  instead of $\mathcal{H}^i(M)$. For a morphism of prestacks $f\colon X\to Y$, we write $f^*\colon \mathsf{D}(Y)\to \mathsf{D}(X)$ for the derived pullback functor and $f_* \colon \mathsf{D}(X)\to \mathsf{D}(Y)$  for the derived pushforward functor. When $Y=\operatorname{Spec}(R)$ is affine, we always write  $R\Gamma(X,-)$ for the functor \(f_*\colon\mathsf D(X)\to\mathsf D(R)\), and, for \(M\in\mathsf D(X)\), we  write $H^i(X,M):=H^i\bigl(R\Gamma(X,M)\bigr)$ for  the $i$-th cohomology of the derived global sections of $M$.
    \item For $M\in \mathsf{D}(X)$, we write $\mathrm{Sym}^n(M), \Lambda^n(M), \Gamma^n(M) \in \mathsf{D}(X)$ for the derived symmetric power, derived exterior power, and derived divided power as in \cite[Section~A.2.2]{KP24}; see also
    \cite[Definition~2.3]{Pet23} and, in the affine case,
    \cite[Section~25.2.2]{Lur18}.
    \item For  a finitely presented affine flat group scheme \(G\) over a commutative ring \(R\), we will freely identify
    \(\QCoh(BG)\) with the category \(\mathsf{Rep}(G)\) of \(G\)-representations over \(R\). Let $\varphi\colon G\to H$ be a homomorphism of  finitely presented affine flat group schemes over $R$. It induces a flat morphism of classifying stacks $B\varphi\colon BG\to BH$. Hence, the pullback functor $(B\varphi)^*\colon \mathsf{D}(BH) \to \mathsf{D}(BG)$ is \(t\)-exact. Under the identifications
    $
    \QCoh(BH)\simeq \mathsf{Rep}(H)$,
    $
    \QCoh(BG)\simeq \mathsf{Rep}(G),
    $
    the induced functor on the hearts
    $
    (B\varphi)^*\colon \QCoh(BH)\longrightarrow \QCoh(BG)
    $
    identifies with the usual restriction functor
  $
    \varphi^*\colon
    \mathsf{Rep}(H)\longrightarrow \mathsf{Rep}(G)
   $.
   \item For a morphism of Artin stacks $X\to Y$, we denote by \(L_{X/Y}\in\mathsf D(X)\) its cotangent complex. 
   \item For a ring-ideal pair $(A,I)$, we denote by $(D_A(I),J,\gamma)$ the divided power envelope of $(A,I)$ relative to $(\mathbb Z, 0, \gamma_{\mathrm{tri}})$ as in \cite[Tag~07H9]{Sta26}. We write $\mathrm{Fil}^n_{\mathrm{PD}}D_A(I):=J^{[n]}$, with $J^{[n]}$ defined as in \cite[Tag~07HQ]{Sta26}.
   \item A variety over a field $k$ is a geometrically integral separated scheme of finite type over $k$.
    \item Throughout the main text, we assume that $k$ is a perfect field of characteristic $p> 0$. Unless otherwise specified, all constructions are taken over the base field $k$, which will  always be omitted from the notation. 
\end{enumerate}

\section*{Preliminaries on totalizations}

We first recall the totalization of a cosimplicial cochain complex. Let $A^{\bullet}:\Delta\to \mathsf{Ch}(R) $ be a cosimplicial cochain complex. We consider unnormalized cochain complex in $\mathsf{Ch}(R)$ in the cosimplicial direction  as in \cite[Tag~019H]{Sta26}.
\[
0 \longrightarrow A^0
\xrightarrow{\delta}
A^1
\xrightarrow{\delta}
A^2
\xrightarrow{\delta}
A^3
\longrightarrow \cdots.
\]
This is an object in $\mathsf{Ch}\bigl(\mathsf{Ch}(R)\bigr)$, or equivalently a bicomplex. Write $ A^{i,j}:=(A^i)^j$,
where \(i\) denotes the cosimplicial degree and \(j\) the cochain degree. The resulting bicomplex has the form

\[
\begin{tikzcd}[column sep=large, row sep=large]
	\vdots &
	\vdots &
	\vdots &
	\\
	A^{0,2} \arrow[u,"d"] \arrow[r,"\delta"] &
	A^{1,2} \arrow[u,"d"] \arrow[r,"\delta"] &
	A^{2,2} \arrow[u,"d"] \arrow[r,"\delta"] &
	\cdots \\
	A^{0,1} \arrow[u,"d"] \arrow[r,"\delta"] &
	A^{1,1} \arrow[u,"d"] \arrow[r,"\delta"] &
	A^{2,1} \arrow[u,"d"] \arrow[r,"\delta"] &
	\cdots \\
	A^{0,0} \arrow[u,"d"] \arrow[r,"\delta"] &
	A^{1,0} \arrow[u,"d"] \arrow[r,"\delta"] &
	A^{2,0} \arrow[u,"d"] \arrow[r,"\delta"] &
	\cdots \\
	\vdots \arrow[u,"d"] &
	\vdots \arrow[u,"d"] &
	\vdots \arrow[u,"d"] &
\end{tikzcd}
\]
where
$
\delta\colon A^{i,j}\longrightarrow A^{i+1,j}
$
is the differential in the cosimplicial direction, and
$
d\colon A^{i,j}\longrightarrow A^{i,j+1}
$
is the original cochain differential of \(A^i\). We define  \(\operatorname{Tot}(A^\bullet)\) to be the product totalization of the bicomplex \((A^{\bullet,\bullet},\delta,d)\) above. More precisely, the degree-\(n\) term of the totalization is $\operatorname{Tot}(A^\bullet)^n
=
\prod_{i+j=n} A^{i,j}$, the total differential is given on \(A^{i,j}\) by  $d_{\operatorname{Tot}}=\delta+(-1)^i d$.

Suppose now that \(A^\bullet\) is a cosimplicial differential graded \(R\)-algebra. We explain the multiplicative structure on \(\operatorname{Tot}(A^\bullet)\).  The unit of \(A^0\) induces a map \(R\to \operatorname{Tot}(A^\bullet)^0\). To define the multiplication, it is enough to specify maps
\[
A^{k,r}\otimes_R A^{l,s}\longrightarrow A^{k+l,r+s}
\]
for all \(k,l,r,s\geq 0\).
These maps are given by the Alexander--Whitney product. Let
$
\sigma_f\colon [k]\hookrightarrow [k+l],
\sigma_b\colon [l]\hookrightarrow [k+l]
$
be the front face and back face maps, defined by
$
\sigma_f(i)=i,
\sigma_b(j)=k+j
$.
Thus \(\sigma_f\) identifies \([k]\) with the first \(k+1\) vertices of \([k+l]\), while \(\sigma_b\) identifies \([l]\) with the last \(l+1\) vertices. For \(x\in A^{k,r}\) and \(y\in A^{l,s}\), we define
\[
x\cup y
:=
 (-1)^{rl}(\sigma_f)_*(x)\cdot (\sigma_b)_*(y)
\in A^{k+l,r+s}.
\]
Here \((\sigma_f)_*\) and \((\sigma_b)_*\) are the maps induced by the cosimplicial structure, and \(\cdot\) denotes the multiplication in the differential graded algebra \(A^{k+l}\).
The Alexander--Whitney product equips \(\operatorname{Tot}(A^\bullet)\) with the structure of a differential graded \(R\)-algebra.

Suppose moreover that \(A^\bullet\) is equipped with a multiplicative decreasing $\mathbb N$-indexed filtration, so that \(\operatorname{Fil}^*A^\bullet\) a cosimplicial diagram of differential graded $R$-algebras
equipped with decreasing $\mathbb N$-indexed filtrations. Then \(\operatorname{Tot}(A^\bullet)\) inherits a filtration \(\operatorname{Fil}^*\operatorname{Tot}(A^\bullet)\) in the evident sense. The Alexander--Whitney product is compatible with this filtration, and hence \(\operatorname{Fil}^*\operatorname{Tot}(A^\bullet)\) is a filtered differential graded \(R\)-algebra.

  We record the following standard facts, which will be used later.
\begin{proposition} \label{Tot} Let $A^{\bullet}:\Delta\to \mathsf{Ch}(R) $ be a cosimplicial cochain complex. We have an equivalence in $\mathsf{D}(R)$
\[\lim_{[i]\in\Delta }A^{i}\simeq \mathrm{Tot}(A^{\bullet}).\]

Moreover if $A^{\bullet}:\Delta\to \DGA_R $ be a cosimplicial decreasing differential graded \(R\)-algebra, then we have an equivalence in $\mathsf{D}(R)$ \[\lim_{[i]\in\Delta }A^{i}\simeq \mathrm{Tot}(A^{\bullet}).\]  is an equivalence in $\Alg(\mathsf{D}(R))$. 

Finally, if $\mathrm{Fil}^*A^{\bullet}:\mathsf{\Delta}\to \DGA_R^{\mathsf{fil}} $ is a cosimplicial diagram of differential graded $R$-algebras
equipped with decreasing $\mathbb N$-indexed filtrations, then the above equivalence is compatible with the filtrations. More precisely, we have an  equivalence \[\lim_{[i]\in\Delta }\mathrm{Fil}^*A^{i}\simeq \mathrm{Fil}^* \mathrm{Tot}(A^{\bullet}).\]
in $\Alg\bigl(\Fil(\mathsf{D}(R))\bigr)$.

\end{proposition}
\section*{Acknowledgements}
I would like to thank my advisors Ruochuan Liu and Shizhang Li for their guidance and support. I am particularly grateful to Shizhang Li  for introducing me to this question and for many helpful discussions.
	
\section*{AI Disclosure}
The author used large language models during the preparation of this manuscript for finding references, proofreading, improving the exposition, and checking mathematical arguments for possible errors and internal inconsistencies. All mathematical ideas and proofs in this paper were developed independently of large language models. The author takes full responsibility for all mathematical content.

\section{Hodge-to-de Rham spectral sequences of classifying stacks}

In this section, we  explain our method for computing  Hodge-to-de Rham spectral sequence of a classifying stack based on divided power de Rham complexes.
\subsection{Derived de Rham cohomology of stacks}

In this subsection, we fix a commutative ring \(R\). We recall the definition of derived de Rham cohomology for syntomic stacks over \(R\), following the construction of Antieau--Bhatt--Mathew in \cite[Construction~2.7]{ABM21}.

Recall that an \(R\)-algebra \(A\) is \emph{syntomic} if it is flat and finitely presented over \(R\), and the cotangent complex \(L_{A/R}\) has \(\operatorname{Tor}\)-amplitude contained in \([-1,0]\). We denote by \(\mathsf{Syn}_R\) the category of syntomic \(R\)-algebras.  Furthermore, the notion of a syntomic morphism can be extended to Artin stacks; see, for example, \cite[Construction~A.1.12]{KP24}.

For \(A\in\mathsf{Syn}_R\), let
\(\mathrm{Fil}_{\mathrm{Hod}}^*\dR_{A/R}\in
\CAlg(\Fil(\mathsf D(R)))\)
denote the Hodge-filtered derived de Rham cohomology introduced in  \cite[Definition~2.1]{Bha12}; see also Construction~\ref{de Rham}. 

\begin{construction}[Derived de Rham cohomology of syntomic stacks]
	\label{dR-syntomic-stack}
	Let \(X\) be a syntomic Artin stack over \(R\). We define its Hodge-filtered derived de Rham cohomology by right Kan extension
	\[
	\mathrm{Fil}_{\mathrm{Hod}}^*\dR_{X/R}
	:=
	\lim_{\substack{\operatorname{Spec}(A)\to X\\ A/R\ \mathrm{syntomic}}}
	\mathrm{Fil}_{\mathrm{Hod}}^*\dR_{A/R}
	\in
	\CAlg\bigl(\Fil(\mathsf D(R))\bigr).
	\]	We denote the underlying object by \(\dR_{X/R}\), and write $H^i_{\dR}(X/R):=H^i\left(\dR_{X/R}\right)\in \Mod_R$.
\end{construction}

\begin{remark}

	When \(R\) has characteristic \(p>0\), the assignment $A\longmapsto \mathrm{Fil}_{\mathrm{Hod}}^*\dR_{A/R}$ satisfies syntomic descent; see \cite[Theorem~2.5]{ABM21} and \cite[Example~5.12]{BMS19}.
	Hence, derived de Rham cohomology of syntomic stacks can be computed by descent along any syntomic atlas. Let \(U\to X\) be a syntomic cover , and let \(U_\bullet\to X\) be its \v{C}ech nerve, so that \(U_n=U\times_X\cdots\times_XU\), with \(n+1\) factors. Syntomic descent gives a natural equivalence
	
	\[
	\mathrm{Fil}_{\mathrm{Hod}}^*\dR_{X/R}
	\simeq
	\lim_{[n]\in\Delta}
	\mathrm{Fil}_{\mathrm{Hod}}^*\dR_{U_n/R}.
	\]
	
\end{remark}

\begin{remark} Suppose that \(R\) has characteristic \(p>0\). It follows from \cite[Corollary~3.10]{Bha12} that there is a natural equivalence \(\mathrm{Fil}_{\mathrm{Hod}}^*\dR_{A/R}\simeq \mathrm{Fil}_{\mathrm{Hod}}^*\Omega^{\bullet}_{A/R}\) for every smooth $R$-algebra $A$, where $\mathrm{Fil}_{\mathrm{Hod}}^*\Omega^{\bullet}_{A/R}$ denotes the classical Hodge-filtered de Rham cohomology. Hence, for every smooth $R$-scheme $X$, Zariski descent shows that \(\mathrm{Fil}_{\mathrm{Hod}}^*\dR_{X/R}\) recovers the classical Hodge-filtered de Rham cohomology of \(X/R\).

\end{remark}
\begin{example}\label{syn}
	Let $G$ be a flat finitely presented affine group scheme over $R$. Note that $G$ is automatically syntomic over $R$. Hence, the canonical base point $*=\operatorname{Spec}(R)\to BG$ is a syntomic cover. Applying syntomic descent, we  obtain a natural equivalence
	
	\[\mathrm{Fil}_{\mathrm{Hod}}^*\dR_{BG/R}
	\simeq
	\lim_{[n]\in\Delta}
	\mathrm{Fil}_{\mathrm{Hod}}^*\dR_{G^n/R}.\]
\end{example}

We next record the description of the graded pieces of derived de Rham cohomology.
\begin{lemma}[{cf.~\cite[Remark~2.8]{ABM21}; Proposition~\ref{cotangent}}]
	The graded pieces of the Hodge filtration are given by
	
	$$
	\operatorname{gr}_{\mathrm{Hod}}^i\dR_{X/R}
	\simeq
	R\Gamma\left(X,\Lambda^iL_{X/R}\right)[-i].
	$$
\end{lemma}

Thus the Hodge filtration gives the Hodge-to-de Rham spectral sequence

\[
E_1^{i,j}
=
H^j\left(X,\Lambda^iL_{X/R}\right)
\Longrightarrow
H_{\mathrm{dR}}^{i+j}(X/R).
\]

\subsection{Derived de Rham cohomology of classifying stacks via the divided power de Rham complex}
We consider the category $\mathsf{Pair}_k^{\mathsf{sm,reg}}$
whose objects are ring-ideal pairs $(A, I)$, where $A$ is a smooth $k$-algebra and $I \subset A$ is an ideal generated by a regular sequence; see Remark~\ref{Pair1} and Definition~\ref{Pair2} for details. And for such a pair $(A, I)$, we write $(D_A(I),J,\gamma)$ for the divided power envelope of $(A, I)$ respect to $(\mathbb{Z},0,\gamma_{tri})$ as in \cite[Tag~07H9]{Sta26}. We then consider the filtered divided power de Rham complex functor,
\begin{equation}
	\begin{aligned}
		\mathrm{Fil}^*_{\mathrm{PD}}\Omega^{\mathrm{pd},\bullet}_{D_-(-)}:\mathsf{Pair}_k^{\mathsf{sm,reg}} &\longrightarrow \CAlg\bigl(\Fil(\mathsf{D}(k))\bigr), \\
		(A, I) &\longmapsto \mathrm{Fil}^*_{\mathrm{PD}}\Omega^{\mathrm{pd},\bullet}_{D_A(I)/k}
	\end{aligned}
\end{equation}
where the divided power de Rham complex $\Omega^{\mathrm{pd},\bullet}_{D_A(I)/k}$ is defined as:
\begin{equation}
	\Omega^{\mathrm{pd},\bullet}_{D_A(I)/k}: = \left( D_A(I) \xrightarrow{d} D_A(I) \otimes_A \Omega^1_{A/k} \xrightarrow{d} D_A(I) \otimes_A \Omega^2_{A/k} \xrightarrow{d} \dots \right)
\end{equation}
And this complex is  equipped with a decreasing multiplicative filtration defined by:
\begin{equation}
	\mathrm{Fil}_{\mathrm{PD}}^{ r} \Omega^{\mathrm{pd},\bullet}_{D_A(I)/k}: = \left( \mathrm{Fil}_{\mathrm{PD}}^{ r} D_A(I) \rightarrow \mathrm{Fil}_{\mathrm{PD}}^{ r-1} D_A(I) \otimes_A \Omega^1_{A/k} \rightarrow \mathrm{Fil}_{\mathrm{PD}}^{ r-2} D_A(I) \otimes_A \Omega^2_{A/k} \rightarrow \dots \right).
\end{equation}
where the decreasing multiplicative filtration $\mathrm{Fil}_{\mathrm{PD}}^{*} D_A(I)$ on $ D_A(I)$  is defined by
$\mathrm{Fil}_{\mathrm{PD}}^{ n} D_A(I)= J^{[n]}$.

By  Theorem~\ref{P1} and Theorem~\ref{P2}, we obtain the following result; compare \cite[Proposition~3.27]{Bha12}, \cite[Theorem~7.16]{Mat22} and \cite[Lemma~4.32 and Proposition~4.35]{Mao24}.
\begin{theorem}\label{PD1}
	There is a natural identification of functors $\mathsf{Pair}_k^{\mathsf{sm,reg}} \to \CAlg \bigl( \Fil(\mathsf{D}(k)) \bigr)$
	\begin{equation}
		\mathrm{Fil}_{\mathrm{PD}}^{*} \Omega^{\mathrm{pd},\bullet}_{D_A(I)/k} \simeq \mathrm{Fil}_{\mathrm{Hod}}^{*} \dR_{(A/I)/k},
	\end{equation}
	for $(A,I) \in \mathsf{Pair}_k^{\mathsf{sm,reg}}$.
\end{theorem}

\begin{construction}\label{PD2}
	We consider an embedding of affine group schemes $G \hookrightarrow \widetilde{G}$, 
	where $\widetilde{G}$ is a smooth affine group scheme over $k$ and the ideal $I = \operatorname{Ker}(\mathcal{O}(\widetilde{G}) \to \mathcal{O}(G))$ is generated by a regular sequence. Then $I_i = \operatorname{Ker}(\mathcal{O}(\widetilde{G}^i) \to \mathcal{O}(G^i))$
	is also generated by a regular sequence for any $i\geq 0$.  We  write $A= \mathcal{O}(\widetilde{G})$ for notational simplicity. In this case, we apply Theorem~\ref{PD1} to the diagram $\{(A^{\otimes i}, I_i)\}_{i\in \Delta}$ and obtain
	\begin{equation}
		\mathrm{Fil}_{\mathrm{Hod}}^{*}\dR_{BG/k}\simeq\lim_{[i]\in \Delta}\mathrm{Fil}_{\mathrm{Hod}}^{*}\dR_{G^i/k}=\lim_{[i]\in \Delta} \mathrm{Fil}_{\mathrm{Hod}}^{*}\dR_{\mathcal{O}(G^i)/k}\simeq\lim_{[i]\in \Delta}\mathrm{Fil}_{\mathrm{PD}}^{*}\Omega^{\mathrm{pd},\bullet}_{D_{I_i}(A^{\otimes i})/k} ,
	\end{equation}
	where the first equivalence comes from syntomic descent, as in Example~\ref{syn}. Using Proposition~\ref{Tot} we can identify $\displaystyle \lim_{[i]\in \Delta}\mathrm{Fil}_{\mathrm{PD}}^{*}\Omega^{\mathrm{pd},\bullet}_{D_{I_i}(A^{\otimes i})/k}$ with the totalization of the filtered bicomplex $\mathrm{Fil}^*B^{\bullet,\bullet}$ described as follows.

The underlying bicomplex \(B^{\bullet,\bullet}\) is given by
	\[
	\begin{tikzcd}[
		row sep=3.1em,
		column sep=2.2em,
		cells={nodes={inner sep=1.5pt}}
		]
		\vdots
		&
		\vdots
		&
		\vdots
		&
		\vdots
		&
		\\
		0
		\ar[u]
		\ar[r]
		&
		D_A(I)\otimes_A\Omega^2_{A/k}
		\ar[u,"d"]
		\ar[r,"\delta"]
		&
		\begin{gathered}
			D_{A^{\otimes 2}}(I_2)\otimes_{A^{\otimes 2}}
			\Omega^2_{A^{\otimes 2}/k}
		\end{gathered}
		\ar[u,"d"]
		\ar[r,"\delta"]
		&
		\begin{gathered}
			D_{A^{\otimes 3}}(I_3)\otimes_{A^{\otimes 3}}
			\Omega^2_{A^{\otimes 3}/k}
		\end{gathered}
		\ar[u,"d"]
		\ar[r,"\delta"]
		&
		\cdots
		\\
		0
		\ar[u]
		\ar[r]
		&
		D_A(I)\otimes_A\Omega^1_{A/k}
		\ar[u,"d"]
		\ar[r,"\delta"]
		&
		\begin{gathered}
			D_{A^{\otimes 2}}(I_2)\otimes_{A^{\otimes 2}}
			\Omega^1_{A^{\otimes 2}/k}
		\end{gathered}
		\ar[u,"d"]
		\ar[r,"\delta"]
		&
		\begin{gathered}
			D_{A^{\otimes 3}}(I_3)\otimes_{A^{\otimes 3}}
			\Omega^1_{A^{\otimes 3}/k}
		\end{gathered}
		\ar[u,"d"]
		\ar[r,"\delta"]
		&
		\cdots
		\\
		k
		\ar[u,"d"]
		\ar[r,"\delta"]
		&
		D_A(I)
		\ar[u,"d"]
		\ar[r,"\delta"]
		&
		D_{A^{\otimes 2}}(I_2)
		\ar[u,"d"]
		\ar[r,"\delta"]
		&
		D_{A^{\otimes 3}}(I_3)
		\ar[u,"d"]
		\ar[r,"\delta"]
		&
		\cdots .
	\end{tikzcd}
	\tag{\(\ast\)}
	\label{bicomplex}
	\]
	Here the horizontal differential \(\delta\) is induced by the
	cosimplicial structure, while the vertical differential \(d\) is the
	de Rham differential. The decreasing filtration on \(B^{\bullet,\bullet}\) is defined bidegreewise by
$\mathrm{Fil}^m B^{i,j}
	:=
	\mathrm{Fil}_{\mathrm{PD}}^{m-j}
	D_{A^{\otimes i}}(I_i)
	\otimes_{A^{\otimes i}}
	\Omega^j_{A^{\otimes i}/k}
$, where we use the conventions \(A^{\otimes 0}=k\) and \(I_0=0\).
	We denote the totalization of $\mathrm{Fil}^*B^{\bullet,\bullet}$ by \(\mathrm{Fil}^*C^\bullet\), which is a filtered differential graded $k$-algebra. Consequently, we can compute the differentials in the Hodge-to-de Rham spectral sequence of $BG$ explicitly from the filtered cochain complex \(\mathrm{Fil}^*C^\bullet\); see \cite[Tag~012K]{Sta26}.
\end{construction}	
	
	\begin{remark}\label{product}
		We would like to describe the product $a\cdot b$ in the differential graded $k$-algebra $C^{\bullet}$ for elements $a\in D_{A^{\otimes n}}(I_n)\subset C^n$ and $b\in D_{A^{\otimes m}}(I_m)\subset C^m$. Indeed, there is a canonical equivalences $D_{A^{\otimes r}}(I_r)\simeq D_{A}(I)^{\otimes r}$ for every $r\geq 0$. Suppose that
		\(a=a_1\otimes\cdots\otimes a_n\) and
		\(b=b_1\otimes\cdots\otimes b_m\),
		where \(a_i,b_j\in D_A(I)\). Then the product is given by \[
		a\cdot b
		=
		a_1\otimes\cdots\otimes a_n
		\otimes
		b_1\otimes\cdots\otimes b_m \in D_{A}(I)^{\otimes (n+m)}\simeq D_{A^{\otimes(n+m)}}(I_{n+m}).
		\]
	\end{remark}

	\begin{notation}
	For an element $a\in \mathrm{Fil}^rC^n$ whose image in $\mathrm{Fil}^rC^n/\mathrm{Fil}^sC^n$ is a cycle in the cochain complex $\mathrm{Fil}^rC^{\bullet}/\mathrm{Fil}^sC^{\bullet}$, we denote the class determined by \(a\) in $H^n(\mathrm{Fil}^rC^{\bullet}/\mathrm{Fil}^sC^{\bullet})$ by $[a]$. When there is a risk of ambiguity, we indicate the ambient cohomology group by writing $[a]\in H^n(\mathrm{Fil}^rC^{\bullet}/\mathrm{Fil}^sC^{\bullet})$.
\end{notation}

\subsection{Low degree Hodge cohomology of classifying stacks} 
In this subsection, we study $H^1(BG,\mathcal O)$,
$H^0(BG,L_{BG/k})$, and $H^1(BG,L_{BG/k})$ using the comparisons
in Theorem~\ref{PD1} and Construction~\ref{PD2}. This will help us to analyze the Hodge-to-de Rham spectral sequence of $BG$. From now on, we assume that $G$ is a finite type affine group scheme  over $k$, and write $i:*=\mathrm{Spec}(k)\to G$ for the unit section of $G$ and $e:*=\mathrm{Spec}(k)\to BG$ for the canonical base point of $BG$. On a first reading, the reader may skip the lemmas and propositions, while keeping track of the surrounding discussion, remarks, and notation.

We first treat $H^1(BG,\mathcal{O})$. By descent along the canonical base point map \(\operatorname{Spec}(k)\to BG\), we obtain equivalences $\displaystyle 
R\Gamma(BG,\mathcal O)
\simeq
\lim_{[n]\in\Delta}\mathcal O(G^n)
\simeq
\mathrm{Tot}\!\left(\mathcal O(G)^{\otimes \bullet}\right)
$.
Here $\mathrm{Tot}\!\left(\mathcal O(G)^{\otimes \bullet}\right)$  is the standard cobar complex computing the rational cohomology of $G$ with
coefficients in the trivial $G$-representation $k$: 
\[
k
\xrightarrow{\delta=0}
\mathcal O(G)
\xrightarrow{\delta=\mathrm{pr}_1^*+\mathrm{pr}_2^*-m^*}
\mathcal O(G)^{\otimes 2}
\xrightarrow{\delta}
\mathcal O(G)^{\otimes 3}
\longrightarrow
\cdots .
\]

\begin{notation}
	We know that \(H^1(BG,\mathcal O)\simeq \ker\!\left(\operatorname{pr}_1^*+\operatorname{pr}_2^*-m^*\colon \mathcal O(G)\to \mathcal O(G)\otimes_k\mathcal O(G)\right) \simeq \operatorname{Hom}_{\Grp}(G,\mathbb G_a)\) for an affine group scheme \(G\) over \(k\).Thus, an element $f\in\mathcal O(G)$ satisfying
	$\operatorname{pr}_1^*(f)+\operatorname{pr}_2^*(f)-m^*(f)=0$,
	i.e. a primitive element, determines a class in $H^1(BG,\mathcal O)$. We denote this class by $\alpha_f$.
\end{notation}

\begin{remark}\label{alpha}
	Suppose now that we are given a regular closed embedding $G\hookrightarrow \widetilde{G}$ as in Construction~\ref{PD2} and retain the notation introduced there. Taking the $\gr^0$ of the equivalence  $\mathrm{Fil}_{\mathrm{Hod}}^{*}\dR_{BG/k}\simeq \mathrm{Fil}^*C^{\bullet}$ in Construction~\ref{PD2}, we obtain $
	R\Gamma(BG,\mathcal O)\simeq \gr^0C^\bullet
	$. By construction, \(\gr^0C^\bullet\) is the totalization of the cosimplicial complex whose degree \(r\) term is
	$
	D_{A^{\otimes r}}(I_r)/\mathrm{Fil}_{\mathrm{PD}}^1D_{A^{\otimes r}}(I_r)
	$.
	For every \(r\ge 0\), there are canonical identifications
	$
	D_{A^{\otimes r}}(I_r)/\mathrm{Fil}_{\mathrm{PD}}^1D_{A^{\otimes r}}(I_r)\simeq A^{\otimes r}/I_r\simeq\mathcal O(G^r)
	$. Hence this cosimplicial complex is canonically identified with the standard cobar complex above.
\end{remark}

We next treat $H^0(BG,L_{BG/k})$. In order to understand Hodge cohomology of $BG$, we first recall some basic facts about the co-Lie complex \(\ell_G\in\mathsf D(BG)\); see \cite[Chapitre~VII, \S 3.1]{Ill72a} and \cite[\S 2.5.1 and \S 3.1]{Ill72b} for the co-Lie complex formalism. It is a $G$-equivariant refinement of $i^*L_{G/k}$,  in the sense that there is a canonical equivalence $e^*\ell_G\simeq i^*L_{G/k}$. Moreover, we have $L_{BG/k}\simeq \ell_G[-1]$, and hence  $e^*L_{BG/k}\simeq i^*L_{G/k}[-1]$; see, for example, \cite[Example~A.3.8]{KP24}.

For completeness, we explain these identifications. We first record the following auxiliary lemma, which packages the transitivity triangle and base change for cotangent complexes. We now denote by \(\dArtStk\) the \(\infty\)-category of derived Artin stacks, where the derived setting will only be used to form derived fiber products.
\begin{lemma}\label{pullback}
	Let \(S\) be a derived Artin stack, and let \(\dArtStk_{S,*}\) denote the
	\(\infty\)-category of pointed derived Artin stacks over \(S\).
	Then the functor
	\[
	F : \dArtStk_{S,*}^{\op}
	\longrightarrow \mathsf{D}(S),
	\qquad
	(X,e: S\to X) \longmapsto e^*L_{X/S},
	\] commutes with finite colimits. 
\end{lemma}
\begin{proof}
	The final object in $\dArtStk_{S,*}$ is $(S,\mathrm{Id})$, and we have $F\bigl((S,\mathrm{Id})\bigr)= L_{S/S}=0$.
	It remains to show that \(F\) carries pullback squares in $\dArtStk_{S,*}$ to pushout squares in $\mathsf D(S)$. For a pullback diagram in $\dArtStk_{S,*}$
	\begin{equation}
		\begin{tikzcd}[column sep=large, row sep=large]
			\bigl(X,e_1\bigr)
			\ar[r,"f"]
			\ar[d,"g"]
			\ar[dr, phantom, very near start, "\lrcorner"]
			&
			\bigl(Y,e_2\bigr)
			\ar[d,"g_1"]
			\\
			\bigl(Z,e_3\bigr)
			\ar[r,"f_1"]
			&
			\bigl(W,e_4\bigr),
		\end{tikzcd}
	\end{equation}
	we need to show 
	
	\begin{equation}
		\begin{tikzcd}[column sep=large, row sep=large]
			e_4^*L_{W/S} \ar[r] \ar[d]
			&
			e_2^*L_{Y/S} \ar[d]
			\\
			e_3^*L_{Z/S} \ar[r]
			&
			e_1^*L_{X/S}
		\end{tikzcd}
	\end{equation}
	is a pushout diagram. By the transitivity triangles for $Y\to W\to S$ and
	$X\to Z\to S$, we have
	\[
	\begin{aligned}
		\cofib(e_4^*L_{W/S}\to e_2^*L_{Y/S})
		&\simeq
		\cofib(e_2^*g_1^*L_{W/S}\to e_2^*L_{Y/S})
		\simeq e_2^*L_{Y/W},\\
		\cofib(e_3^*L_{Z/S}\to e_1^*L_{X/S})
		&\simeq
		\cofib(e_1^*g^*L_{Z/S}\to e_1^*L_{X/S})
		\simeq e_1^*L_{X/Z}.
	\end{aligned}
	\]
	Since the original square is a pullback, base change for the
	cotangent complex shows that the natural map
	$f^*L_{Y/W}\to L_{X/Z}$ is an equivalence.
	Hence the induced cofiber map
	$e_2^*L_{Y/W}\to e_1^*L_{X/Z}$ is an equivalence.
	Since $\mathsf D(S)$ is stable, the desired square is both
	a pullback and a pushout.
	Therefore, $F$ preserves finite colimits, completing the proof.
\end{proof}	

\begin{example}Applying Lemma~\ref{pullback} to the following pullback diagram 
	\[
	\begin{tikzcd}[column sep=large, row sep=large]
		(G,i)
		\ar[r]
		\ar[d]
		\ar[dr, phantom, very near start, "\lrcorner"]
		&
		*
		\ar[d]
		\\
		*
		\ar[r]
		&
		(BG,e),	
	\end{tikzcd}
	\]
	we obtain $e^*L_{BG}	\simeq i^*L_G[-1]$. Moreover, the conjugation action defines a \(G\)-action on \(G\) by group automorphisms with unit section $i_{BG}:BG\to [G/^{\mathrm{ad}}G]$. Thus, we obtain a group object $[G/^{\mathrm{ad}}G]$ over $BG$. Consider the following pullback diagram 
	\[
	\begin{tikzcd}[column sep=large, row sep=large]
		([G/^{\mathrm{ad}}G],i_{BG})
		\ar[r]
		\ar[d]
		\ar[dr, phantom, very near start, "\lrcorner"]
		&
		(BG,\mathrm{Id})
		\ar[d,"\Delta"]
		\\
		(BG,\mathrm{Id})
		\ar[r,"\Delta"]
		&
		(BG\times BG,\Delta),
	\end{tikzcd}
	\]
	then we obtain $L_{BG}\simeq \Delta^*L_{BG\times BG/BG}
	\simeq i_{BG}^*L_{[G/^{\mathrm{ad}}G]/BG}[-1]$. One can check this recovers the equivalence in \cite[Example~A.3.8]{KP24}. In fact, one can define the co-Lie complex as
	$\ell_G:=i_{BG}^*L_{[G/^{\mathrm{ad}}G]/BG}\in\mathsf D(BG)$,
	so the equivalence above gives $L_{BG/k}\simeq\ell_G[-1]$.
\end{example}

Since \(G\) is a finite type group scheme over \(k\), it is syntomic over \(k\) by \cite[Lemma~2.1.3]{CZ24}. Hence \(L_{BG/k}\simeq \ell_G[-1]\) is a perfect complex of $\mathrm{Tor}$-amplitude \([0,1]\). In particular, its only possibly non-zero cohomology sheaves are $\mathcal{H}^0(L_{BG/k})$ and $\mathcal{H}^1(L_{BG/k})$.

\begin{notation}
	Following Illusie's notation \cite[\S~4.1]{Ill72b}, we define $n_G:=\mathcal{H}^0(L_{BG/k})=\mathcal{H}^{-1}(\ell_G)\in \QCoh(BG)$, and $\omega_G:=\mathcal{H}^1(L_{BG/k})=\mathcal{H}^{0}(\ell_G)\in \QCoh(BG)$, and use the same notation for their underlying \(k\)-vector spaces when no confusion can arise.
\end{notation}
Accordingly, there is a canonical fiber sequence $ n_G\to L_{BG/k} \to  \omega_G[-1]$ in $\mathsf{D}(BG)$. Taking cohomology, we obtain an exact sequence 
\[0\to H^0(BG,n_G)\to H^0(BG,L_{BG/k}) \to H^0(BG,\omega_G[-1])  =0.\] Hence, we see $ H^0(BG,L_{BG/k})\simeq H^0(BG,n_G)=(n_G)^G $, the $G$-invariant subspace of $n_G$. More generally, one can show that \(\mathcal H^0(\Lambda^rL_{BG/k})\simeq\Lambda^rn_G\), and hence $H^0(BG,\Lambda^r L_{BG/k})\simeq H^0(BG,\Lambda^r n_G)=(\Lambda^r n_G)^G$.

Suppose now that we are given a regular closed embedding $f: G\hookrightarrow \widetilde{G}$ as in Construction~\ref{PD2} and retain the notation introduced there. Write $j:\operatorname{Spec}(k)\to\widetilde G$ for the unit section of $\widetilde G$ and $\widetilde{e}:\operatorname{Spec}(k)\to B\widetilde G$ for the canonical base point of $B\widetilde G$. For each $n\geq 0$, let
$f_n:G^n\hookrightarrow\widetilde G^n$
denote the induced embedding, and let
$f_{-1}:BG\to B\widetilde G$
denote the induced morphism of classifying stacks. The embedding $f: G\hookrightarrow \widetilde{G}$ gives another description of $H^0(BG,L_{BG/k})$ as follows. Taking the first graded piece of the equivalence  $\mathrm{Fil}_{\mathrm{Hod}}^{*}\dR_{BG/k}\simeq \mathrm{Fil}^*C^{\bullet}$ in Construction~\ref{PD2}, we obtain an equivalence $R\Gamma(BG,L_{BG/k})[-1]\simeq \gr^1(C^{\bullet})$. Consequently, we have $H^0(BG,L_{BG/k})\simeq H^1(\gr^1\bigl(C^{\bullet})\bigr)$. Here $\gr^1(C^{\bullet})$ is the totalization of the following bicomplex
\[
\begin{tikzcd}[row sep=2.8em, column sep=3.5em]
	\vdots
	&
	\vdots
	&
	\vdots
	&
	\\
	0 \arrow[u] \arrow[r, "\delta"]
	&
	0 \arrow[u] \arrow[r, "\delta"]
	&
	0 \arrow[u] \arrow[r, "\delta"]
	&
	\cdots
	\\
	0 \arrow[u] \arrow[r, "\delta"]
	&
	\Omega^1_{A/k}\otimes_A A/I
	\arrow[r, "\delta"]
	\arrow[u, "d"]
	&
	\Omega^1_{A^{\otimes 2}/k}
	\otimes_{A^{\otimes 2}} A^{\otimes 2}/I_2
	\arrow[r, "\delta"]
	\arrow[u, "d"]
	&
	\cdots
	\\
	0 \arrow[u] \arrow[r, "\delta"]
	&
	I/I^2
	\arrow[r, "\delta"]
	\arrow[u, "d"]
	&
	I_2/I_2^2
	\arrow[r, "\delta"]
	\arrow[u, "d"]
	&
	\cdots  .
\end{tikzcd}
\]
Hence, we have
\[
H^1(\gr^1)\simeq H^1(0
\rightarrow
I/I^2
\xrightarrow{\,(\delta,-d)\,}
I_2/I_2^2\oplus \Omega^1_{A/k} \otimes_A (A/I)
\longrightarrow \cdots)=\ker(I/I^2
\xrightarrow{\,(\delta,-d)\,}
I_2/I_2^2\oplus \Omega^1_{A/k} \otimes_A (A/I)).
\]

We next show that these two descriptions of \(H^0(BG,L_{BG/k})\) are compatible in a suitable sense.
\begin{convention}
	For a map $d:M\to N$, we write $[M\xrightarrow{d}N]_{[a,a+1]}$
	for the two-term complex with $M$ in cohomological degree $a$
	and $N$ in cohomological degree $a+1$.
\end{convention}

We begin with the following commutative diagram, which involves
some subtle sign issues that we explain in the proof.
\begin{lemma}\label{cp}There is a commutative diagram  in $\mathsf{D}(k)$
\[
\begin{tikzcd}[column sep=large,row sep=large]
	R\Gamma(BG,L_{BG/k})[-1]
	\ar[r,"\simeq"]
	\ar[d,"\simeq"']
	&
	\gr^1(C^{\bullet})
	=
	\operatorname{Tot}\!\left(
	\left[
	I_\bullet/I_\bullet^2
	\xrightarrow{d}
	\Omega^1_{A^{\otimes\bullet}/k}\otimes_{A^{\otimes\bullet}} A^{\otimes\bullet}/I_\bullet
	\right]_{[0,1]}
	\right)
	\ar[d,"\simeq"]
	\\
	\displaystyle
	\left(
	\lim_{[n]\in\Delta} R\Gamma(G^n,L_{G^n/k})
	\right)[-1]
	\ar[r,"\simeq"]
	&
	\operatorname{Tot}\!\left(
	\left[
	I_\bullet/I_\bullet^2
	\xrightarrow{d}
	\Omega^1_{A^{\otimes\bullet}/k}\otimes_{A^{\otimes\bullet}} A^{\otimes\bullet}/I_\bullet
	\right]_{[-1,0]}
	\right)[-1].
\end{tikzcd}
\]
Here the upper horizontal equivalence is induced by taking the first graded piece of the equivalence $\mathrm{Fil}_{\mathrm{Hod}}^{*}\dR_{BG/k}\simeq \mathrm{Fil}^*C^{\bullet}$ in Construction~\ref{PD2}, the left vertical equivalence is induced by syntomic descent, and the lower horizontal equivalence comes from the standard presentation of cotangent complexes in the lci case. The right vertical equivalence acts by \((-1)^{n+1}\) on \(I_n/I_n^2\) and by \((-1)^n\) on \(\Omega^1_{A^{\otimes n}/k}\otimes_{A^{\otimes n}}A^{\otimes n}/I_n\).	
\end{lemma}

\begin{proof}
	The upper horizontal equivalence in the statement can be expressed
	as the composite from the upper-left corner to the lower-right corner
	in the following commutative diagram:
	
	\[
	\begin{tikzcd}[column sep=large,row sep=large]
		R\Gamma(BG,L_{BG/k})[-1]
		=
		\gr^1_{\mathrm{Hod}}\dR_{BG/k}
		\ar[d,"\simeq"']
		\ar[dr,"\simeq"]
		&
		\\
		\displaystyle
		\lim_{[n]\in\Delta}
		\bigl(R\Gamma(G^n,L_{G^n/k})[-1]\bigr)
		=
		\lim_{[n]\in\Delta}
		\gr^1_{\mathrm{Hod}}\dR_{G^n/k}
		\ar[r,"\simeq"]
		\ar[d,"\simeq"']
		&
		\displaystyle
		\lim_{[n]\in\Delta}
		\gr^1_{\mathrm{PD}}
		\Omega^{\mathrm{pd},\bullet}_{D_{A^{\otimes n}}(I_n)/k}
		\ar[d,"\simeq"]
		\\
		\displaystyle
		\operatorname{Tot}\!\left(
		\left[
		I_\bullet/I_\bullet^2
		\xrightarrow{d}
		\Omega^1_{A^{\otimes\bullet}/k}
		\otimes_{A^{\otimes\bullet}}
		A^{\otimes\bullet}/I_\bullet
		\right]_{[-1,0]}[-1]
		\right)
		\ar[r,"\simeq"]
		&
		\displaystyle
		\operatorname{Tot}\!\left(
		\left[
		I_\bullet/I_\bullet^2
		\xrightarrow{d}
		\Omega^1_{A^{\otimes\bullet}/k}
		\otimes_{A^{\otimes\bullet}}
		A^{\otimes\bullet}/I_\bullet
		\right]_{[0,1]}
		\right)
		=\gr^1C^\bullet.
	\end{tikzcd}
	\]

The upper left vertical equivalence is induced by syntomic descent,
and the middle horizontal equivalence is induced by the comparison in Theorem~\ref{PD1}.
The lower vertical equivalences are given by the standard two-term
presentations of the cotangent complexes and the first graded
pieces of divided power de Rham complexes, respectively.
By Remark~\ref{sign1}, the lower horizontal equivalence acts by
$-\mathrm{id}$ on $I_n/I_n^2$ and by $\mathrm{id}$ on
$\Omega^1_{A^{\otimes n}/k}
\otimes_{A^{\otimes n}}A^{\otimes n}/I_n$.
On the other hand, by Lemma~\ref{sign2}, we obtain a commutative diagram
\[
\begin{tikzcd}[column sep=large,row sep=large]
	\displaystyle
	\left(
	\lim_{[n]\in\Delta}
	R\Gamma(G^n,L_{G^n/k})
	\right)[-1]
	\ar[r,"\simeq"] \ar[d,"\simeq"']
	&
	\displaystyle
	\lim_{[n]\in\Delta}
	\bigl(R\Gamma(G^n,L_{G^n/k})[-1]\bigr)
	\ar[d,"\simeq"]
	\\
	\operatorname{Tot}\!\left(
	\left[
	I_\bullet/I_\bullet^2
	\xrightarrow{d}
	\Omega^1_{A^{\otimes\bullet}/k}
	\otimes_{A^{\otimes\bullet}}
	A^{\otimes\bullet}/I_\bullet
	\right]_{[-1,0]}
	\right)[-1]
	\ar[r,"\simeq"]
	&
	\operatorname{Tot}\!\left(
	\left[
	I_\bullet/I_\bullet^2
	\xrightarrow{d}
	\Omega^1_{A^{\otimes\bullet}/k}
	\otimes_{A^{\otimes\bullet}}
	A^{\otimes\bullet}/I_\bullet
	\right]_{[-1,0]}[-1]
	\right).
\end{tikzcd}
\]
The upper horizontal arrow is the canonical equivalence exchanging
the shift with the limit, and the vertical equivalences come from
the standard two-term presentations of the cotangent complexes.
By Lemma~\ref{sign2}, the lower horizontal equivalence acts by
$(-1)^n$ on both \(I_n/I_n^2\) and \(\Omega^1_{A^{\otimes n}/k}\otimes_{A^{\otimes n}}A^{\otimes n}/I_n\). Composing the lower horizontal equivalences in the second and first
diagrams gives an equivalence acting by $(-1)^{n+1}$ on $I_n/I_n^2$
and by $(-1)^n$ on
$\Omega^1_{A^{\otimes n}/k}
\otimes_{A^{\otimes n}}A^{\otimes n}/I_n$.
Its inverse gives the right vertical equivalence in the statement,
with the same signs. Therefore, combining the two commutative diagrams gives the desired
commutative diagram, completing the proof.
	
\end{proof}

\begin{proposition}\label{com}
	There is a commutative diagram of $k$-vector spaces
	\[
	\begin{tikzcd}[column sep=large,row sep=large]
		(n_G)^G
		\ar[r,"\simeq"]
		\ar[d,hook]
		&
		H^0(BG,L_{BG/k})
		\ar[r,"\simeq"]
		&
		\ker\!\left(
		I/I^2
		\xrightarrow{\,(\delta,-d)\,}
		I_2/I_2^2\oplus
		\Omega^1_{A/k}\otimes_A(A/I)
		\right)
		\ar[d,"i^*"]
		\\
		n_G
		\ar[rr,"\simeq"]
		&&
		\ker\!\left(
		i^*(I/I^2)\to j^*\Omega^1_{A/k}
		\right).
	\end{tikzcd}
	\]
	The second upper horizontal isomorphism is obtained by taking $H^1$
	of the equivalence
	$R\Gamma(BG,L_{BG/k})[-1]\simeq\gr^1C^\bullet$.
	Viewing $n_G$ as a $k$-vector space, the lower horizontal isomorphism
	is induced by the identifications
	$n_G\simeq H^0(e^*L_{BG/k})
	\simeq H^{-1}(i^*L_{G/k})
	\simeq\ker(i^*(I/I^2)\to j^*\Omega^1_{A/k})$,
	where the last isomorphism comes from  the standard
	two-term presentation
	$L_{G/k}\simeq
	[I/I^2\xrightarrow{d}\Omega^1_{A/k}\otimes_A A/I]_{[-1,0]}$.
\end{proposition}

\begin{proof}
	Let $i_n:\operatorname{Spec}(k)\to G^n$ denote the unit section.
	Pulling back along the base point $e$ of $BG$ and the unit sections
	$i_n$ of $G^n$ gives the following commutative diagram in $\mathsf D(k)$:
	\[
	\begin{tikzcd}[row sep=large,column sep=large]
		R\Gamma(BG,L_{BG/k})
		\ar[r,"\simeq"]
		\ar[d,"e^*"']
		&
		\displaystyle
		\lim_{[n]\in\Delta}
		R\Gamma(G^n,L_{G^n/k})
		\ar[r,"\simeq"]
		\ar[d,"\displaystyle\lim_{[n]\in\Delta}i_n^*"]
		&
		\displaystyle
		\operatorname{Tot}\!\left(
		\left[
		I_\bullet/I_\bullet^2
		\xrightarrow{d}
		\Omega^1_{A^{\otimes\bullet}/k}
		\otimes_{A^{\otimes\bullet}}
		A^{\otimes\bullet}/I_\bullet
		\right]_{[-1,0]}
		\right)
		\ar[d,"\operatorname{Tot}(i_\bullet^*)"]
		\\
		e^*L_{BG/k}
		\ar[r,"\simeq"]
		&
		\displaystyle
		\lim_{[n]\in\Delta}
		i_n^*L_{G^n/k}
		\ar[r,"\simeq"]
		&
		\displaystyle
		\operatorname{Tot}\!\left(
		\left[
		i_\bullet^*(I_\bullet/I_\bullet^2)
		\xrightarrow{d}
		i_\bullet^*\!\left(
		\Omega^1_{A^{\otimes\bullet}/k}
		\otimes_{A^{\otimes\bullet}}
		A^{\otimes\bullet}/I_\bullet
		\right)
		\right]_{[-1,0]}
		\right).
	\end{tikzcd}
	\]
	The first upper horizontal arrow is an equivalence by flat descent
	for cotangent complexes.
	The first lower horizontal arrow is induced by the augmented
	cosimplicial diagram
	$e^*L_{BG/k}\to(i_n^*L_{G^n/k})_{[n]\in\Delta}$.
	By Lemma~\ref{pullback}, this is the co\v{C}ech nerve of
	$e^*L_{BG/k}\to0$, and hence is a limit diagram by the dual of
	\cite[Lemma~1.2.4.19]{Lur17}.
	The remaining horizontal equivalences are induced by the standard
	two-term presentations for the cotangent complexes.
	
Let
$C:=[i^*(I/I^2)\xrightarrow{d}j^*\Omega^1_{A/k}]_{[-1,0]}$.
The cosimplicial system of cochain complexes
$\bigl(
[i_n^*(I_n/I_n^2)\xrightarrow{d}
i_n^*(\Omega^1_{A^{\otimes n}/k}
\otimes_{A^{\otimes n}}A^{\otimes n}/I_n)]_{[-1,0]}
\bigr)_{[n]\in\Delta}$
is naturally isomorphic to $(C^{\oplus n})_{[n]\in\Delta}$
with its standard cobar structure maps.
Normalizing in the cosimplicial direction gives a bicomplex
with $C$ in cosimplicial degree $1$ and zero in all other
cosimplicial degrees.
Its total complex is $C[-1]$, and the inclusion into the
unnormalized bicomplex induces a natural quasi-isomorphism
$C[-1]\to\operatorname{Tot}(C^{\oplus\bullet})$.
Thus we obtain equivalences
$e^*L_{BG/k}\simeq\operatorname{Tot}(C^{\oplus\bullet})\simeq C[-1]$
in $\mathsf D(k)$, recovering the canonical equivalence
$e^*L_{BG/k}\simeq i^*L_{G/k}[-1]$.
Consequently, taking $H^0$ of the lower row of the diagram above
gives the lower horizontal isomorphism in the statement.
	
Taking $H^0$ and applying Lemma~\ref{cp}, we obtain the following
commutative diagram:
\[
\begin{tikzcd}[row sep=large,column sep=large]
	H^0(BG,L_{BG/k})
	\ar[r,"\simeq"]
	\ar[d,"e^*"']
	&
	H^0\!\left(
	\operatorname{Tot}\!\left(
	\left[
	I_\bullet/I_\bullet^2
	\xrightarrow{d}
	\Omega^1_{A^{\otimes\bullet}/k}
	\otimes_{A^{\otimes\bullet}}
	A^{\otimes\bullet}/I_\bullet
	\right]_{[-1,0]}
	\right)
	\right)
	\ar[r,"\simeq"]
	&
	H^1(\gr^1C^\bullet)
	\ar[d,"i^*"]
	\\
	n_G
	\ar[rr,"\simeq"]
	&&
	\ker\!\left(
	i^*(I/I^2)\to j^*\Omega^1_{A/k}
	\right).
\end{tikzcd}
\]
The second upper horizontal isomorphism is induced by
Lemma~\ref{cp} and acts by $(-1)^{1+1}=1$ on $I/I^2$
in cosimplicial degree $1$.
Under the identification
$H^1(\gr^1C^\bullet)
=\ker(I/I^2\xrightarrow{(\delta,-d)}
I_2/I_2^2\oplus\Omega^1_{A/k}\otimes_A A/I)$,
the right vertical map is therefore restriction along $i$.
The composite of the upper horizontal isomorphisms is obtained
by taking $H^1$ of
$R\Gamma(BG,L_{BG/k})[-1]\simeq\gr^1C^\bullet$.
Identifying $H^0(BG,L_{BG/k})$ with $(n_G)^G$ gives
the diagram in the statement.
\end{proof}
\begin{notation}
	Let \(y\in I\), and write \(\bar y\in I/I^2\) for its image. Suppose that $i^*(\bar y) \in \ker\bigl(i^*(I/I^2)\to j^*\Omega^1_{A/k} \bigr)$.
	Under the canonical identification \(n_G=\ker(i^*(I/I^2)\to j^*\Omega^1_{A/k})\), we denote by \(\eta_y\in n_G\) the element corresponding to \(i^*\bar y\).
\end{notation}
\begin{remark}\label{eta}
	Suppose that $\eta_y $ is $G$-invariant, i.e. that it lies in $(n_G)^G$.  By Proposition~\ref{com}, there is a unique element \([\widetilde y]\in \ker\left(I/I^2\xrightarrow{(\delta,-d)}I_2/I_2^2\oplus
	\bigl(\Omega^1_{A/k}\otimes_A A/I\bigr)\right)=H^1(\gr^1C^{\bullet})\) whose restriction along the  map $i^*:\ker\!\left(
	I/I^2
	\xrightarrow{\,(\delta,-d)\,}
	I_2/I_2^2\oplus
	\bigl(\Omega^1_{A/k}\otimes_A A/I\bigr)
	\right)\to \ker\bigl(i^*(I/I^2)\to j^*\Omega^1_{A/k} \bigr)$ is $\eta_y$. Under the two descriptions of $H^0(BG,L_{BG/k})$ above,
	the elements $\eta_y$ and its lift represent the same cohomology class.
\end{remark}

We conclude our discussion of $H^0(BG,L_{BG/k})$ with a lemma giving a sufficient condition for the \(G\)-action on \(n_G\) to be trivial.

\begin{lemma}\label{normal}
	Let $G$ be a finite type group scheme over $k$ admitting a closed embedding into a smooth  affine group scheme $\widetilde{G}$, such that  $G$ is a normal subgroup scheme of $\widetilde{G}$. Then $n_G$ is a $G$-representation with trivial $G$-action.
\end{lemma}
\begin{proof}
Let $K$ be the quotient group scheme $\widetilde{G}/G$, which is a smooth affine  group scheme over $k$ by \cite[Exp.~VI$_B$, Thm.~11.17]{SGA3}.  Consider a fiber sequence in $\Stk_{k}$
	\[
	\begin{tikzcd}[column sep=large, row sep=large]
		BG \arrow[r,"h"] \arrow[d,"\pi"] \arrow[dr, phantom, "\lrcorner", very near start]
		&
		B\widetilde{G} \arrow[d,"g"]
		\\
		* \arrow[r,"e_1"]
		&
		BK .
	\end{tikzcd}
	\]
	Thus, we obtain\[L_{BG}\simeq h^*L_{B\widetilde{G}/BK}\simeq h^*\cofib(g^*L_{BK}\to L_{B\widetilde{G}})\simeq \cofib(h^*g^*L_{BK}\to h^*L_{B\widetilde{G}})\simeq \cofib(\pi^*e_1^*L_{BK}\to h^*L_{B\widetilde{G}} ).\] 
	Taking cohomology, we obtain an exact sequence in $\QCoh(BG)=\mathsf{Rep}(G)$
	
	\[0\to n_G\to \pi^*e_1^*\omega_K\to h^*\omega_{\widetilde{G}}.\]
	Since $\pi^*e_1^*\omega_K$  is a trivial $G$-representation, it follows that $n_G$ is a trivial $G$-representation.

\end{proof}

\begin{remark}
	In general, the \(G\)-representation \(n_G\) need not be trivial, even for finite group schemes. For example, consider the semidirect product $G=\alpha_p^2\rtimes C_2$, where $C_2$ acts on $\alpha_p^2$ by permuting the two factors.
\end{remark}

We next consider $H^1(BG,L_{BG/k})$.
We continue to work with the fixed embedding
$G\hookrightarrow\widetilde G$ and retain the notation introduced above. We now impose the additional condition that the canonical map $d\colon I/I^2\to\Omega^1_{A/k}\otimes_A A/I$ is zero.
This condition will be assumed throughout the remainder of this subsection. In this case, we obtain a canonical splitting of $L_{BG/k}$ as follows.
\begin{lemma}\label{split} Under the above assumption, we have canonical equivalences $L_{BG/k}\simeq n_G\oplus \omega_G[-1]$ and $f_{-1}^*L_{B\widetilde{G}/k}\simeq \omega_G[-1]$.
	
\end{lemma}	
\begin{proof}
	Note that there is a fiber sequence $ f_{-1}^*L_{B\widetilde{G}/k}\to L_{BG/k} \to L_{BG/B\widetilde{G}}$. Since $\widetilde{G}$ is smooth over $k$, the complex $f_{-1}^*L_{B\widetilde{G}/k}$ is concentrated in cohomological degree \(1\). Thus, it suffices to show $f_{-1}^*L_{B\widetilde{G}/k}\to L_{BG/k}$ induces an equivalence on $\mathcal{H}^1$. Indeed, this identifies \(f_{-1}^*L_{B\widetilde{G}/k}\) with \(\omega_G[-1]\) via the composite \(f_{-1}^*L_{B\widetilde{G}/k}\to L_{BG/k}\to \omega_G[-1]\), which gives a canonical splitting \(L_{BG/k}\simeq n_G\oplus\omega_G[-1]\). Recall that the base point
	$ e: \operatorname{Spec}(k)\longrightarrow BG
	$
	is a fppf cover. Hence \(e^*\) is \(t\)-exact and conservative. We are thus reduced to proving that $H^1(\widetilde{e}^*L_{B\widetilde{G}/k})\to H^1(e^*L_{BG/k})$ is an isomorphism. Equivalently, we need to show that $H^0(j^*L_{\widetilde{G}/k}) \to H^0(i^*L_{G/k}) $ is an isomorphism. Concretely, this is the map on \(H^0\) induced by
	$
	j^*\Omega^1_{\widetilde G/k}
	\longrightarrow
	\bigl[
	i^*(I/I^2)
	\xrightarrow{0}
	j^*\Omega^1_{\widetilde G/k}
	\bigr]_{[-1,0]}
	$, which is an isomorphism by the assumption.
\end{proof}	

It follows that $H^1(BG,L_{BG/k})\simeq H^1(BG,n_G\oplus \omega_G[-1])\simeq H^1(BG,n_G)\oplus H^0(BG,\omega_G)=H^1(BG,n_G)\oplus (\omega_G)^G$. We next establish a compatibility statement for $H^1(BG,L_{BG/k})$ analogous to that in Proposition~\ref{com}. We begin with the following technical lemma.

\begin{lemma}\label{fun}The natural map
		\[ R\Gamma\!\left(BG,f_{-1}^*L_{B\widetilde{G}/k}\right)\longrightarrow \lim_{[n]\in \Delta}R\Gamma\!\left(G^n,f_n^*L_{\widetilde G^n/k}\right)\]
		induced by functoriality of the cotangent complex is an equivalence.
\end{lemma}
\begin{proof} By the transitivity triangle for cotangent complexes, we obtain a map of fiber sequences
	\[
	\begin{tikzcd}[column sep=large, row sep=large]
		R\Gamma\!\left(BG,f_{-1}^*L_{B\widetilde G/k}\right)
		\ar[r] \ar[d]
		&
		R\Gamma(BG,L_{BG/k})
		\ar[r] \ar[d]
		&
		R\Gamma(BG,L_{BG/B\widetilde G})
		\ar[d]
		\\
		\displaystyle
		\lim_{[n]\in\Delta}
		R\Gamma\!\left(G^n,f_n^*L_{\widetilde G^n/k}\right)
		\ar[r]
		&
		\displaystyle
		\lim_{[n]\in\Delta}
		R\Gamma\!\left(G^n,L_{G^n/k}\right)
		\ar[r]
		&
		\displaystyle
		\lim_{[n]\in\Delta}
		R\Gamma\!\left(G^n,L_{G^n/\widetilde G^n}\right).
	\end{tikzcd}
	\]
	The middle vertical map is an equivalence by flat descent for the cotangent complex in \cite[Proposition 1.1.5]{KP22}. It therefore remains to show that $\displaystyle
	R\Gamma(BG,L_{BG/B\widetilde G})
	\to
	\lim_{[n]\in\Delta}
	R\Gamma\!\left(G^n,L_{G^n/\widetilde G^n}\right)$ is an equivalence.
	
	Let $a_n:G^n\to BG$ denote the augmentation, and set
	$b_n:=f_{-1}\circ a_n$.
	It is enough to show that the natural map
	\[
	f_{-1*}L_{BG/B\widetilde G}
	\longrightarrow
	\lim_{[n]\in\Delta}
	b_{n*}L_{G^n/\widetilde G^n}
	\]
	is an equivalence in $\mathsf D(B\widetilde G)$. Since the base point $\widetilde e\colon \mathrm{Spec}(k)\to B\widetilde G$ is a smooth cover, we see that $\widetilde e^*$ is conservative. Hence, we just need to show that $\displaystyle\widetilde e^*
	f_{-1*}L_{BG/B\widetilde G}
	\to
	\widetilde e^*\bigl(\lim_{[n]\in\Delta}
	(b_n)_*L_{G^n/\widetilde G^n}\bigr)$ is an equivalence. Since $f_n\colon G^n\to \widetilde G^n$ is a regular closed embedding, the cotangent complexes $L_{G^n/\widetilde G^n}$ have
	$\mathrm{Tor}$-amplitude contained in $[-1,0]$. Thus, by \cite[Corollary~3.1.13]{KP24} we have $\displaystyle \widetilde e^*\bigl(\lim_{[n]\in\Delta}
	(b_n)_*L_{G^n/\widetilde G^n}\bigr)\simeq  \lim_{[n]\in\Delta}
	\widetilde e^*(b_n)_*L_{G^n/\widetilde G^n}$. Hence, we need to show that $\displaystyle \widetilde e^*
	f_{-1*}L_{BG/B\widetilde G}
	\to
	\lim_{[n]\in\Delta}\widetilde e^*
	(b_n)_*L_{G^n/\widetilde G^n}$ is an equivalence. 
	
	By flat base change and base change for relative cotangent complexes,
	we have
	$\widetilde e^*f_{-1*}L_{BG/B\widetilde G}
	\simeq R\Gamma(\widetilde G/G,L_{(\widetilde G/G)/k})$
	and
	$\widetilde e^*(b_n)_*L_{G^n/\widetilde G^n}
	\simeq R\Gamma(G^n\times\widetilde G,
	L_{(G^n\times\widetilde G)/\widetilde G^{n+1}})$
	for each $n\geq 0$.
	Thus it remains to show that the natural map
	$R\Gamma(\widetilde G/G,L_{(\widetilde G/G)/k})
	\to \lim_{[n]\in\Delta}
	R\Gamma(G^n\times\widetilde G,
	L_{(G^n\times\widetilde G)/\widetilde G^{n+1}})$
	is an equivalence. Here the natural map is induced by the augmented Čech nerve of
	the morphism
	$(\widetilde G\xrightarrow{\mathrm{id}}\widetilde G)
	\to(\widetilde G/G\to\operatorname{Spec}(k))$
	in the arrow category of $k$-stacks
	\[
	\begin{tikzcd}[column sep=large, row sep=large]
		\cdots \ar[r]
		&
		G^2\times\widetilde G
		\ar[r,shift left=1.2ex]
		\ar[r]
		\ar[r,shift right=1.2ex]
		\ar[d,"q_2"]
		&
		G\times\widetilde G
		\ar[r,shift left=.6ex]
		\ar[r,shift right=.6ex]
		\ar[d,"q_1"]
		&
		\widetilde G
		\ar[r]
		\ar[d,"\mathrm{id}"]
		&
		\widetilde G/G
		\ar[d]
		\\
		\cdots \ar[r]
		&
		\widetilde G^3
		\ar[r,shift left=1.2ex]
		\ar[r]
		\ar[r,shift right=1.2ex]
		&
		\widetilde G^2
		\ar[r,shift left=.6ex]
		\ar[r,shift right=.6ex]
		&
		\widetilde G
		\ar[r]
		&
		\operatorname{Spec}(k).
	\end{tikzcd}
	\]
	
	Here $q_n:G^n\times\widetilde G\to\widetilde G^{n+1}$
	denotes the vertical map in simplicial degree $n$. There are two standard presentations of the augmented Čech nerve
	of $\widetilde G\to\operatorname{Spec}(k)$, both with
	$\widetilde G^{n+1}$ in simplicial degree $n$.
	One is associated with the left multiplication action of
	$\widetilde G$ on itself.
	The other has face and degeneracy maps given by
	\[
	d_m(g_0,\ldots,g_n)=(g_0,\ldots,\widehat{g_m},\ldots,g_n),
	\qquad
	s_m(g_0,\ldots,g_n)=(g_0,\ldots,g_m,g_m,\ldots,g_n),
	\]
	where $0\leq m\leq n$, with $n\geq 1$ for the face maps. We use the second presentation, in which all maps in the
	augmented simplicial object are group homomorphisms. For each \(n\geq 0\), the transitivity triangle gives a fiber sequence in $\mathsf{D}(G^n\times\widetilde G)$
\[
q_n^*L_{\widetilde G^{n+1}/k}
\longrightarrow
L_{(G^n\times\widetilde G)/k}
\longrightarrow
L_{(G^n\times\widetilde G)/\widetilde G^{n+1}}.
\]
	Applying derived global sections and taking cosimplicial limits, we obtain the following fiber sequence in $\mathsf{D}(k)$
\[
\lim_{[n]\in\Delta}
R\Gamma\!\left(
G^n\times\widetilde G,q_n^*L_{\widetilde G^{n+1}/k}
\right)
\longrightarrow
\lim_{[n]\in\Delta}
R\Gamma\!\left(
G^n\times\widetilde G,L_{(G^n\times\widetilde G)/k}
\right)
\longrightarrow
\lim_{[n]\in\Delta}
R\Gamma\!\left(
G^n\times\widetilde G,
L_{(G^n\times\widetilde G)/\widetilde G^{n+1}}
\right).
\]
	Flat descent for cotangent complexes \cite[Proposition 1.1.5]{KP22} gives a natural equivalence
	$\displaystyle R\Gamma(\widetilde G/G,L_{(\widetilde G/G)/k})\simeq\lim_{[n]\in \Delta}
	R\Gamma\!\left(
	G^n\times\widetilde G,L_{(G^n\times\widetilde G)/k}
	\right)$.
	Thus it suffices to show that
	$\displaystyle \lim_{[n]\in\Delta}
	R\Gamma(G^n\times\widetilde G,q_n^*L_{\widetilde G^{n+1}/k})
	\simeq 0$. Write $\pi_{n+1}:G^n\times\widetilde G\to\operatorname{Spec}(k)$
	and $\widetilde\pi_{n+1}:\widetilde G^{n+1}\to\operatorname{Spec}(k)$
	for the structure morphisms, and
	$j_{n+1}:\operatorname{Spec}(k)\to\widetilde G^{n+1}$
	for the unit section.
	By \cite[Chapitre~VII, (4.1.2.1)]{Ill72a} (see also \cite[\S~3.1]{Ill72b}), we have natural equivalences $L_{\widetilde G^{n+1}/k}\simeq \widetilde{\pi}_{n+1}^*j_{n+1}^*L_{\widetilde G^{n+1}/k}\simeq \widetilde{\pi}_{n+1}^*\omega_{\widetilde G^{n+1}}$ for every $n\geq -1$, where $\omega_{\widetilde G^{n+1}}$ is viewed as a $k$-vector space. The naturality follows from the fact that all morphisms appearing in the augmented simplicial diagram $\{\widetilde G^{n+1}\}_{[n]\in(\Delta^+)^{\mathsf{op}}}$ are group homomorphisms. Hence, we obtain natural equivalences 
\[
\begin{aligned}
	\lim_{[n]\in\Delta}
	R\Gamma(G^n\times\widetilde G,q_n^*L_{\widetilde G^{n+1}/k})
	&\simeq \lim_{[n]\in\Delta}
	R\Gamma(G^n\times\widetilde G,
	q_n^*\widetilde{\pi}_{n+1}^*\omega_{\widetilde G^{n+1}})\\
	&\simeq \lim_{[n]\in\Delta}
	R\Gamma(G^n\times\widetilde G,
	\pi_{n+1}^*\omega_{\widetilde G^{n+1}})\\
	&\simeq \lim_{[n]\in\Delta}
	\mathcal O(G^n\times\widetilde G)\otimes
	\omega_{\widetilde G^{n+1}}\\
	&\simeq \lim_{([n],[m])\in\Delta\times\Delta}
	\mathcal O(G^n\times\widetilde G)\otimes
	\omega_{\widetilde G^{m+1}}\\
	&\simeq \lim_{[n]\in\Delta}
	\left(\lim_{[m]\in\Delta}
	\mathcal O(G^n\times\widetilde G)\otimes
	\omega_{\widetilde G^{m+1}}\right).
\end{aligned}
\]
Here the third equivalence follows from the projection formula, and the fourth equivalence follows from 
the fact that the diagonal map $\Delta\to \Delta\times \Delta$ is left cofinal by \cite[Tag~02QP]{Lur26}.
Note that \(\{\omega_{\widetilde G^{n+1}}\}_{[n]\in\Delta^{+}}\simeq \{\omega_{\widetilde G}^{\oplus (n+1)}\}_{[n]\in\Delta^{+}}\) is a split augmented cosimplicial $k$-vector space; see \cite[Tag~04S9]{Lur26} for the notion of split augmented cosimplicial objects. By functoriality, for any $M\in\mathsf{D}(k)$, we have that \(\{\omega_{\widetilde G^{n+1}}\otimes M\}_{[n]\in\Delta^{+}}\) is also a split augmented cosimplicial object in $\mathsf{D}(k)$. Consequently, by \cite[Tag~04SH]{Lur26}, we conclude that
	$\displaystyle
	\lim_{[m]\in\Delta}
	\omega_{\widetilde G^{m+1}}
	\otimes\mathcal O(G^n\times\widetilde G)
	\simeq \omega_{\widetilde G^0}\otimes \mathcal O(G^n\times\widetilde G)
	\simeq 0$.
	Therefore, the limit
	$\displaystyle \lim_{[n]\in\Delta}
	R\Gamma(G^n\times\widetilde G,q_n^*L_{\widetilde G^{n+1}/k})$ vanishes, as desired. This completes the proof of the lemma.

\end{proof}

\begin{proposition}\label{com2}
	There is a commutative diagram of $k$-vector spaces
	\[
	\begin{tikzcd}[row sep=large,column sep=large]
		H^1(BG,L_{BG/k})
		\ar[r,"\simeq"]
		&
		H^2(\gr^1C^\bullet)
		\\
		(\omega_G)^G
		\simeq H^1(BG,f_{-1}^*L_{B\widetilde G/k})
		\ar[r,"\simeq"]
		\ar[u,hook]
		\ar[d,hook]
		&
		\ker\!\left(
		\delta:
		\Omega^1_{A/k}\otimes_A A/I
		\to
		\Omega^1_{A^{\otimes 2}/k}
		\otimes_{A^{\otimes 2}}
		A^{\otimes 2}/I_2
		\right)
		\ar[u,hook]
		\ar[d,"i^*"]
		\\
		\omega_G
		\ar[r,"\simeq"]
		&
		j^*\Omega^1_{A/k}.
	\end{tikzcd}
	\]
	The upper horizontal isomorphism is induced by the equivalence
	$R\Gamma(BG,L_{BG/k})[-1]\simeq\gr^1C^\bullet$.
	The middle horizontal isomorphism is induced by the equivalence in the Lemma~\ref{fun}.
	Viewing $\omega_G$ as a $k$-vector space, the lower horizontal
	isomorphism is induced by the identifications
	$\omega_G
	\simeq H^1(e^*f_{-1}^*L_{B\widetilde G/k})
	\simeq H^1(\widetilde e^*L_{B\widetilde G/k})
	\simeq H^0(j^*L_{\widetilde G/k})
	\simeq j^*\Omega^1_{A/k}$,
	where the first isomorphism comes from Lemma~\ref{split}.
	The left vertical maps are the natural inclusions.
	The upper right vertical map is the negative of the natural inclusion,
	and the lower right vertical map is restriction along $i$.
\end{proposition}

\begin{proof}
	The natural map
	$f_{-1}^*L_{B\widetilde G/k}\to L_{BG/k}$
	and its analogues in each cosimplicial degree induce the following
	commutative diagram in $\mathsf D(k)$:
	\[
	\begin{tikzcd}[column sep=large,row sep=large]
		R\Gamma(BG,f_{-1}^*L_{B\widetilde G/k})
		\ar[r,"\simeq"]
		\ar[d]
		&
		\displaystyle
		\lim_{[n]\in\Delta}
		R\Gamma(G^n,f_n^*L_{\widetilde G^n/k})
		\ar[r,"\simeq"]
		\ar[d]
		&
		\displaystyle
		\operatorname{Tot}\!\left(
		\Omega^1_{A^{\otimes\bullet}/k}
		\otimes_{A^{\otimes\bullet}}
		A^{\otimes\bullet}/I_\bullet
		\right)
		\ar[d]
		\\
		R\Gamma(BG,L_{BG/k})
		\ar[r,"\simeq"]
		&
		\displaystyle
		\lim_{[n]\in\Delta}
		R\Gamma(G^n,L_{G^n/k})
		\ar[r,"\simeq"]
		&
		\displaystyle
		\operatorname{Tot}\!\left(
		\left[
		I_\bullet/I_\bullet^2
		\xrightarrow{d}
		\Omega^1_{A^{\otimes\bullet}/k}
		\otimes_{A^{\otimes\bullet}}
		A^{\otimes\bullet}/I_\bullet
		\right]_{[-1,0]}
		\right).
	\end{tikzcd}
	\]
	The first upper horizontal equivalence is given by Lemma~\ref{fun},
	and the first lower horizontal equivalence is induced by syntomic
	descent.
	The remaining horizontal equivalences come from the smoothness of
	$\widetilde G$ and the standard two-term presentations for the cotangent
	complexes of $G^n$.
	
	Taking $H^1$ and applying Lemma~\ref{cp}, we obtain the following
	commutative diagram:
	\[
	\begin{tikzcd}[row sep=large,column sep=large]
		(\omega_G)^G
		\simeq H^1(BG,f_{-1}^*L_{B\widetilde G/k})
		\ar[r,"\simeq"]
		\ar[d,hook]
		&
		\ker\!\left(
		\delta:
		\Omega^1_{A/k}\otimes_A A/I
		\to
		\Omega^1_{A^{\otimes 2}/k}
		\otimes_{A^{\otimes 2}}
		A^{\otimes 2}/I_2
		\right)
		\ar[d,hook,"{\xi\mapsto[\xi]}"]
		\\
		H^1(BG,L_{BG/k})
		\ar[r,"\simeq"]
		\ar[dr,"\simeq"']
		&
		H^1\!\left(
		\operatorname{Tot}\!\left(
		\left[
		I_\bullet/I_\bullet^2
		\xrightarrow{d}
		\Omega^1_{A^{\otimes\bullet}/k}
		\otimes_{A^{\otimes\bullet}}
		A^{\otimes\bullet}/I_\bullet
		\right]_{[-1,0]}
		\right)
		\right)
		\ar[d,"\simeq"]
		\\
		&
		H^2(\gr^1C^\bullet).
	\end{tikzcd}
	\]
	The lower right vertical isomorphism is induced by Lemma~\ref{cp}
	and sends the class of a form cocycle $\xi$ in cosimplicial degree $1$
	to $-[\xi]$.
	The diagonal isomorphism is the one in the statement.
	Thus the composite of the right vertical maps is the negative
	of the natural inclusion. This therefore gives the upper square in the statement.
	
	To obtain the lower square, let
	$j_n:\operatorname{Spec}(k)\to\widetilde G^n$ denote the unit section.
	Pulling back along the base point $e$ of $BG$ and the unit sections
	$i_n$ of $G^n$ gives the following commutative diagram in $\mathsf D(k)$:
	\[
	\begin{tikzcd}[row sep=large,column sep=large]
		R\Gamma(BG,f_{-1}^*L_{B\widetilde G/k})
		\ar[r,"\simeq"]
		\ar[d,"e^*"']
		&
		\displaystyle
		\lim_{[n]\in\Delta}
		R\Gamma(G^n,f_n^*L_{\widetilde G^n/k})
		\ar[r,"\simeq"]
		\ar[d,"\displaystyle\lim_{[n]\in\Delta}i_n^*"]
		&
		\displaystyle
		\operatorname{Tot}\!\left(
		\Omega^1_{A^{\otimes\bullet}/k}
		\otimes_{A^{\otimes\bullet}}
		A^{\otimes\bullet}/I_\bullet
		\right)
		\ar[d,"\operatorname{Tot}(i_\bullet^*)"]
		\\
		e^*f_{-1}^*L_{B\widetilde G/k}
		\ar[r,"\simeq"]
		&
		\displaystyle
		\lim_{[n]\in\Delta}
		j_n^*L_{\widetilde G^n/k}
		\ar[r,"\simeq"]
		&
		\displaystyle
		\operatorname{Tot}\!\left(
		j_\bullet^*\Omega^1_{A^{\otimes\bullet}/k}
		\right).
	\end{tikzcd}
	\]
	Here we use $f_n\circ i_n=j_n$ to identify the pullbacks in the
	second row.
	Taking $H^1$ gives the lower square by the same argument as in
	Proposition~\ref{com}.
	Combining the two squares proves the proposition.
\end{proof}
\begin{notation}
	For an element $x\in A$, the differential $dx\in\Omega^1_{A/k}$ determines an element $j^*(dx)\in j^*\Omega^1_{A/k}\simeq \omega_G$. We     denote this element by $\omega_x $ . 
\end{notation}
\begin{remark}\label{omega}
	Suppose $\omega_x $ is $G$-invariant, i.e. that it lies in $(\omega_G)^G$. Then by Proposition~\ref{com2}, $\omega_x$ admits a unique lift  to \[
	\ker\!\left(
	\delta:
	\Omega^1_{A/k}\otimes_A A/I
	\to
	\Omega^1_{A^{\otimes 2}/k}
	\otimes_{A^{\otimes 2}}
	A^{\otimes 2}/I_2
	\right)\subset H^2(\gr^1C^{\bullet})\] along the restriction map $-i^*:\Omega^1_{A/k}\otimes_A A/I\to j^*\Omega^1_{A/k}$. Under the two descriptions of $H^1(BG,L_{BG/k})$ above,
	the elements $\omega_x$ and its lift represent the same cohomology class.
\end{remark}

\subsection{$B\alpha_{p^n}$}
In \cite[Propositions~4.10 and~4.12]{ABM21}, Antieau--Bhatt--Mathew computed the Hodge-to-de Rham spectral sequence of \(B\alpha_{p}\). We now apply our tools to compute the Hodge-to-de Rham spectral sequence of $B\alpha_p$. In fact, we will treat the case of $B\alpha_{p^n}$ by considering the closed embedding $\alpha_{p^n}\hookrightarrow \mathbb{G}_a$. Let $t$ be the standard coordinate on
$\mathbb G_a=\mathrm{Spec}(k[t])$, so that
$\alpha_{p^n}=\mathrm{Spec}(k[t]/(t^{p^n}))$. We first describe the structure sheaf cohomology of  $B\alpha_{p^n}$.

\begin{proposition}[{\cite[Proposition~I.4.27]{Jan03}}]\label{compute}
	We have $H^1(B\alpha_{p^n},\mathcal{O})\simeq \Hom_{\Grp}(\alpha_{p^n},\mathbb{G}_a)\simeq \displaystyle \bigoplus_{0\leq i\leq n-1}k\alpha_{t^{p^i}}$.
	Moreover, if $\operatorname{char}(k)=2$, there is a canonical isomorphism
	\[H^*(B\alpha_{p^n},\mathcal{O})\simeq\operatorname{Sym}^*(H^1(B\alpha_{p^n},\mathcal{O}))=\operatorname{Sym}^*(k\alpha_{t} \oplus k\alpha_{t^p}\oplus \cdots\oplus k\alpha_{t^{p^{n-1}}}).\]
	If \(\operatorname{char}(k)\neq2\), there is a canonical isomorphism
	\[
	\begin{aligned}
		H^*(B\alpha_{p^n},\mathcal{O}) 
		&\simeq \Lambda^*(H^1(B\alpha_{p^n},\mathcal{O})) \otimes \operatorname{Sym}^* \left( \operatorname{Span}_k \left\{ \bar{\beta} H^1(B\alpha_{p^n},\mathcal{O}) \right\} \right) \\
		&= \Lambda^* \left(  k\alpha_{t} \oplus k\alpha_{t^p}\oplus \cdots\oplus k\alpha_{t^{p^{n-1}}}  \right) \otimes \operatorname{Sym}^*(k\bar{\beta}(\alpha_{t}) \oplus k\bar{\beta}(\alpha_{t^p})\oplus \cdots\oplus k\bar{\beta}(\alpha_{t^{p^{n-1}}})).
	\end{aligned}
	\]
\end{proposition}
We briefly recall the map \(\bar\beta:H^1(B\alpha_{p^n},\mathcal O)\to H^2(B\alpha_{p^n},\mathcal O)\) appearing above. Consider the map $\beta:\mathcal{O}(\alpha_{p^n})\to \mathcal{O}(\alpha_{p^n})\otimes \mathcal{O}(\alpha_{p^n})$, $\displaystyle \beta(f)=\sum_{1 \leq i\leq p-1}\frac{1}{p}\binom{p}{i}f^i\otimes f^{p-i}$, where $\displaystyle \frac{1}{p}\binom{p}{i}\in \mathbb Z$. A direct calculation shows that $\beta$  induces a map $\bar{\beta}:H^1(B\alpha_{p^n},\mathcal{O})\to H^2(B\alpha_{p^n},\mathcal{O})$.

\begin{lemma}[{cf.~\cite[Proposition 4.10]{ABM21}}]We have an isomorphism of bigraded commutative algebras
	\[H^*(B\alpha_{p^n},\Lambda^*L_{B\alpha_{p^n}})\simeq H^*(B\alpha_{p^n},\mathcal{O})\otimes \Lambda^*(k\eta_{t^{p^n}})\otimes \mathrm{Sym}^*(k\omega_t),\]
	where $\eta_{t^{p^n}}\in H^0(B\alpha_{p^n},L_{B\alpha_{p^n}})$ and $\omega_t \in H^1(B\alpha_{p^n},L_{B\alpha_{p^n}})$.
\end{lemma}
\begin{proof}
	To compute Hodge cohomology, we first calculate the cotangent complex $L_{B\alpha_{p^n}/k}$. Since \(\alpha_{p^n}\) is commutative, the conjugation action of $\alpha_{p^n}$ on itself is trivial. It follows from \cite[Chapitre~VII, Proposition~4.1.1]{Ill72a} or \cite[Proposition~4.4]{Ill72b}
	that $L_{B\alpha_{p^n}/k}
	\simeq \pi^*e^*L_{B\alpha_{p^n}/k}
	\simeq \pi^*i^*L_{\alpha_{p^n}/k}[-1]$,
	where $\pi:B\alpha_{p^n}\to\mathrm{Spec}(k)$
	is the structure morphism. Using the closed embedding $\alpha_{p^n}\hookrightarrow \mathbb{G}_a$, we  see that $L_{\alpha_{p^n}/k}$ can be computed by 2-term complex $[(t^{p^n})/(t^{2p^{n}})\xrightarrow{0} k[t]/(t^{p^n})dt]$. Hence, we obtain $L_{B\alpha_{p^n}/k}=n_{\alpha_{p^n}}\oplus \omega_{\alpha_{p^n}}[-1]=k\eta_{t^{p^n}}\oplus k\omega_t[-1]$. Thus, we have 
	\[
	\begin{aligned}
		R\Gamma(B\alpha_{p^n},\Lambda^*L_{B\alpha_{p^n}})[-*]\simeq  R\Gamma(B\alpha_{p^n},\Lambda^*(k\eta_{t^{p^n}}\oplus k\omega_t[-1]))[-*]&\simeq R\Gamma(B\alpha_{p^n},\mathcal{O})\otimes\Lambda^*(k\eta_{t^{p^n}}\oplus k\omega_t[-1])[-*]\\&\simeq R\Gamma(B\alpha_{p^n},\mathcal{O})\otimes \Lambda^*k\eta_{t^{p^n}}[-*]\otimes \mathrm{Sym}^*(k\omega_t)[-2*].
	\end{aligned}
	\]
	The second equivalence follows from projection formula, while the last equivalence follows from Illusie's décalage isomorphism; see \cite[Chapitre~I, \S4.3.2]{Ill71} or \cite[Proposition~A.2.4]{KP24}. Taking cohomology, we obtain the desired isomorphism.
\end{proof}
\begin{proposition}[{cf.~\cite[Proposition 4.12]{ABM21}}]\label{alphap} On the $E_1$-page of the Hodge-to-de Rham spectral sequence of $B\alpha_{p^n}$, we have $d_1(\eta_{t^{p^n}})=0$,$d_1(\omega_t)=0$, $d_1(\alpha_{t^{p^i}})= 0$ for $1\leq i\leq n-1$, and $d_1(\alpha_t)=\omega_t\neq 0$. If \(\operatorname{char}(k)=p\neq 2\), we also have $d_1(\bar{\beta}(\alpha_{t^{p^i}}))=0$ for $0\leq i\leq n-1$.
\end{proposition}

\begin{proof}
	We prove this proposition only in the case $\operatorname{char}(k)\neq 2$; the case $\operatorname{char}(k)= 2$ is similar. Consider the closed embedding $\alpha_{p^n}\hookrightarrow \mathbb{G}_a$. We apply Theorem~\ref{PD1} and Construction~\ref{PD2} to this closed embedding and retain the notation introduced there. As in Proposition~\ref{compute}, we also define  $ \beta:\mathcal{O}(\mathbb G_a)\to \mathcal{O}(\mathbb G_a)\otimes \mathcal{O}(\mathbb G_a)$ by $\displaystyle \beta(f)=\sum_{1 \leq i\leq p-1}\frac{1}{p}\binom{p}{i}f^i\otimes f^{p-i}$.
	
	By Remark~\ref{alpha}, we see that $\alpha_{t^{p^i}}\in H^1(B\alpha_{p^n},\mathcal{O})$ corresponds to $[t^{p^i}]\in H^1(\gr^0C^{\bullet})$, and $\bar{\beta}(\alpha_{t^{p^i}})\in H^2(B\alpha_{p^n},\mathcal{O})$ corresponds to $[\beta(t^{p^i})]\in H^2(\gr^0C^{\bullet})$. We first compute $d_1(\alpha_{t})$. Since  \(\delta(t)=0\), we have $d_{\mathrm{Tot}}(t)=-dt$ in the complex $C^{\bullet}$. The class $d_1(\alpha_t)$ can be computed by
	the boundary operator $\partial:H^1(\gr^0C^{\bullet})\to H^2(\gr^1C^{\bullet})$
	induced by the short exact sequence of cochain complexes
	\[0\to \gr^1C^{\bullet}\to \mathrm{Fil}^0C^{\bullet}/\mathrm{Fil}^2C^{\bullet}\to \gr^0C^{\bullet}\to 0.\] Hence, we compute $d_1(\alpha_t)=[d_{\mathrm{Tot}}(t)]=-[dt]$ in $H^2(\gr^1C^{\bullet})=E^{1,1}_1$. By Remark~\ref{omega}, the class \(-[dt]\in H^2(\gr^1)\) corresponds to \(\omega_t\), and hence \(d_1(\alpha_t)=\omega_t\). In particular, we have $d_1(\omega_t)=0$.
	
	We next compute \(d_1(\alpha_{t^{p^i}})\) and  \(d_1(\bar{\beta}(\alpha_{t^{p^i}}))\) for \(1\leq i\leq n-1\). In the bicomplex $B^{\bullet,\bullet}$, 
	we have $\delta(t^{p^i})=0$ and $\delta(\beta(t^{p^i}))=0$  for $1\leq i\leq n-1$, $d(t^{p^i})=0$ and $d(\beta(t^{p^i}))=0$ for $1\leq i\leq n-1$.  Hence $t^{p^i}\in D_A(I)=C^1$ satisfies $d_{\mathrm{Tot}}(t^{p^i})=0$ for $1 \leq i\leq n-1$, and $\beta(t^{p^i})\in D_{A^{\otimes 2}}(I_2)\subset C^2$ satisfies $d_{\mathrm{Tot}}(\beta(t^{p^i}))=0$ for $1 \leq i\leq n-1$. Thus, the element $t^{p^i}$ determines a class in $H^1(\mathrm{Fil}^0C^{\bullet})$ with image $\alpha_{t^{p^i}}$ in $H^1(\gr^0C^{\bullet})=E^{0,1}_1$ for $1\leq i\leq n-1$; and  $\beta(t^{p^i})$ defines a class in $H^2(\mathrm{Fil}^0C^{\bullet})$ with image $\bar{\beta}(\alpha_{t^{p^i}})$ in $H^2(\gr^0C^{\bullet})=E^{0,2}_1$. It follows that $\alpha_{t^{p^i}}$ and $\bar{\beta}(\alpha_{t^{p^i}})$ are permanent cycles for $1 \leq i\leq n-1$. 
	
	We next consider $\bar{\beta}(\alpha_{t})$.
	Note that $\delta(\beta(t))=0$ and  $d(\beta(t))=(t_1+t_2)^{p-1}d(t_1+ t_2)-t_1^{p-1}dt_1-t_2^{p-1}dt_2=-\delta(t^{p-1}dt) $ in the bicomplex $B^{\bullet,\bullet}$, where we write $t\otimes1$ and $1\otimes t$ as $t_1$ and $t_2$, respectively.
	Since \(\Omega^2_{\mathbb G_a/k}=0\), we  have $d(t^{p-1}dt)=0$. Thus, the element $a=(\beta(t), t^{p-1}dt)\in D_{A^{\otimes 2}}(I_2)\oplus D_{A}(I)\otimes_A \Omega^1_{A/k}=C^2 $ satisfies $d_{\mathrm{Tot}}(a)=0$. It follows that $a$ determines a class in $H^2(C^{\bullet})=H^2(\mathrm{Fil}^0C^{\bullet})$ with image $ \bar{\beta}(\alpha_{t})$ in $H^2(\gr^0)=E^{0,2}_1$. Hence, we see that $ \bar{\beta}(\alpha_{t})$ is a permanent cycle.
	
	Finally, we consider $\eta_{t^{p^n}}$. By Remark~\ref{eta}, the class $\eta_{t^{p^n}}$ corresponds to $[t^{p^n}]\in H^1(\gr^1 C^{\bullet})$. Since $\delta(t^{p^n})=0$ and  $d(t^{p^n})=0$ in the bicomplex $B^{\bullet,\bullet}$,  the element $t^{p^n}$ determines a class in $H^1(\mathrm{Fil}^1C^{\bullet} )$ with image $\eta_{t^{p^n}}$ in $H^1(\gr^1)=E^{1,0}_1$. Hence, we see that $\eta_{t^{p^n}}$ is a permanent cycle.
	
\end{proof}

\section{Examples from Heisenberg}
Let us review an example coming from complex geometry. Consider the complex Heisenberg group
\[
H(\mathbb{C}):=
\left\{
\begin{pmatrix}
	1 & x & y\\
	0 & 1 & z\\
	0 & 0 & 1
\end{pmatrix}
\;\middle|\;
x,y,z\in\mathbb C
\right\}
\]
and the discrete cocompact subgroup
\[
\Gamma:=H(\mathbb{C})\cap \operatorname{GL}(3,\mathbb Z[i])\subset H(\mathbb{C}).
\]
Then $M:=\Gamma\backslash H(\mathbb C)$
is a complex  manifold known as the Iwasawa manifold. It admits a non-closed holomorphic $1$-form. In particular, the Fr\"olicher spectral sequence of Iwasawa manifold is non-degenerate at $E_1$-page; see, for example, \cite[p.~444]{GH78}. This example from complex geometry suggests us to consider variants of Heisenberg group scheme for our purposes.

We first  recall the definition of Heisenberg group scheme. 
\begin{definition}
Let \(H:\CRing_k\to\Grp\) be the functor defined by
\[
H(R)=
\left\{
\begin{pmatrix}
	1&x&y\\
	0&1&z\\
	0&0&1
\end{pmatrix}
\;\middle|\;
x,y,z\in R\right\}.
\]
This defines a finite type affine group scheme over $k$.	
\end{definition}

 We next consider the following variants of $H$. These examples and computations illustrate the initial idea behind our general construction.

\subsection[{$H[F]$}]{$H[F]$}

Let \(H[F]\) be the Frobenius kernel of $H$.  More explicitly, let \(H[F]:\CRing_k\to\Grp\) be the subgroup functor of \( H \) defined by
\[
H[F](R)=
\left\{
\begin{pmatrix}
1&x&y\\
0&1&z\\
0&0&1
\end{pmatrix}
\;\middle|\;
x,y,z\in R,\quad
x^p=y^p=z^p=0
\right\}.
\]
Hence, it defines a closed subgroup scheme $H[F]$ of $H$ with definition ideal $I=(x^p,y^p,z^p)$.
    Informally,  \[H[F]=
\begin{pmatrix}
1 & \alpha_p & \alpha_p \\
0 & 1 & \alpha_p \\
0 & 0 & 1 
\end{pmatrix} .
\]

We apply Theorem~\ref{PD1} and Construction~\ref{PD2} to the closed embedding $H[F]\hookrightarrow H$, and retain the notation
introduced in Construction~\ref{PD2}. In this case, we have
$I/I^2=(A/I)x^p\oplus(A/I)y^p\oplus(A/I)z^p$
and
$n_{H[F]}\simeq
k\eta_{x^p}\oplus k\eta_{y^p}\oplus k\eta_{z^p}$.
Since $H[F]$ is normal in $H$, the $H[F]$-action on $n_{H[F]}$
is trivial by Lemma~\ref{normal}. Hence, we have $H^0(B(H[F]),\Lambda^rL_{B(H[F])})\simeq\Lambda^r (n_{H[F]})$ for every $r\geq 0$.

\begin{proposition}\label{Heisenberg}
    In the Hodge-to-de Rham spectral sequence of $B(H[F])$, the class $\eta_{y^p}\in H^0(B(H[F]),L_{B(H[F])})=E^{1,0}_1$ satisfies $d_1(\eta_{y^p})=-\eta_{x^p}\wedge \eta_{z^p}\neq 0$ in $H^0(B(H[F]),\Lambda^2L_{B(H[F])/k})=E^{2,0}_1$.
\end{proposition}
\begin{proof}We apply Theorem~\ref{PD1} and Construction~\ref{PD2} to this closed embedding  $H[F]\hookrightarrow H$, and retain the notation introduced in Construction~\ref{PD2}. By Remark~\ref{eta}, the class  $\eta_{y^p} \in n_{H[F]}=H^0(B(H[F]),L_{B(H[F])/k})$ corresponds to $[y^p]\in H^1(\gr^1)$. We now compute the two differentials of $y^p$  in the bicomplex $B^{\bullet,\bullet}$: $d(y^p)=0$, while 
$\delta(y^p)=y^p\otimes 1+1\otimes y^p-(y\otimes 1+1\otimes y+x\otimes z)^p=-x^p\otimes z^p\in \mathrm{Fil}^2_{\mathrm{PD}}D_{A^{\otimes 2}}(I_2)$. Note that $d_1([y^p])$ is given by the image of $[y^p]$ under the boundary operator $\partial :H^1(\gr^1C^{\bullet})=E^{1,0}_1\to H^2(\gr^2C^{\bullet})=E^{2,0}_1$  associated with the short exact sequence of cochain complexes
\[  0\to \gr^2C^{\bullet}\to \mathrm{Fil}^1C^{\bullet}/\mathrm{Fil}^3C^{\bullet}\to \gr^1C^{\bullet}\to 0 .        \] Thus, we have $d_1([y^p] )= [d_{\mathrm{Tot}}(y^p)]= -[x^p\otimes z^p]$ in $H^2(\gr^2C^{\bullet})=E^{2,0}_1$. By Remark~\ref{product} we have $[x^p\otimes z^p]=[x^p]\cdot[z^p]$ in $H^2(\gr^2C^{\bullet})$, which corresponds to $\eta_{x^p}\wedge \eta_{z^p}$. Hence, we have $d_1(\eta_{y^p})=-\eta_{x^p}\wedge \eta_{z^p}$. Finally, since \(\eta_{x^p},\eta_{y^p},\eta_{z^p}\) form a basis of \(n_{H[F]}\), we have \(\eta_{x^p}\wedge\eta_{z^p}\neq0\) in $\Lambda^2(n_{H[F]})=H^0(B(H[F]),\Lambda^2L_{B(H[F])/k})$. Hence $d_1(\eta_{y^p})\neq 0$ in $E_1^{2,0}$.
\end{proof}

\subsection[{$H[F^{4/3}]$}]{$H[F^{4/3}]$}
Let \(H[F^{4/3}]:\CRing_k\to\Grp\) be a subgroup functor of \( H \) defined by
\[
H[F^{4/3}](R)=
\left\{
\begin{pmatrix}
1&x&y\\
0&1&z\\
0&0&1
\end{pmatrix}
\;\middle|\;
x,y,z\in R,\quad
x^p=y^{p^2}=z^p=0
\right\}.
\]
Hence, it defines a closed  subgroup scheme $H[F^{4/3}]$ of $H$ with definition ideal $I=(x^p,y^{p^2},z^p)$.
    Informally, \[H[F^{4/3}]=
\begin{pmatrix}
1 & \alpha_p & \alpha_{p^2} \\
0 & 1 & \alpha_p \\
0 & 0 & 1 
\end{pmatrix}.
\]
We apply Theorem~\ref{PD1} and Construction~\ref{PD2} to the closed embedding $H[F^{4/3}]\hookrightarrow H$, and retain the notation
introduced in Construction~\ref{PD2}. In this case, we have
\(I/I^2=(A/I)x^p\oplus(A/I)y^{p^2}\oplus(A/I)z^p\)
and \(n_G\simeq k\eta_{x^p}\oplus k\eta_{y^{p^2}}
\oplus k\eta_{z^p}\). Since $H[F^{4/3}]$ is normal in $H$, the $H[F^{4/3}]$-action on $n_{H[F^{4/3}]}$
is trivial by Lemma~\ref{normal}. Hence, we have $H^0(B(H[F^{4/3}]),\Lambda^rL_{B(H[F^{4/3}])})\simeq\Lambda^r (n_{H[F^{4/3}]})$ for every $r\geq 0$. Moreover, since $ x,y^p,z\in \operatorname{Ker}(\mathrm{pr}_1^*+\mathrm{pr}_2^*-m^*:\mathcal{O}(H[F^{4/3}])\to \mathcal{O}(H[F^{4/3}])\otimes \mathcal{O}(H[F^{4/3}]))$, they determine classes $\alpha_x,\alpha_{y^p},\alpha_z \in H^1(B(H[F^{4/3}]),\mathcal{O})$.

\begin{proposition}\label{Heisenberg2}
 In the Hodge-to-de Rham spectral sequence of $B(H[F^{4/3}])$, the class $\alpha_{y^p} \in H^1(B(H[F^{4/3}]),\mathcal{O})=E_1^{0,1}$ satisfies $d_1(\alpha_{y^p})=0$, and $d_2(\alpha_{y^p})=- \eta_{x^p} \wedge \eta_{z^p} \neq 0$ in $E_{2}^{2,0}$.
\end{proposition}

\begin{proof} We apply Theorem~\ref{PD1} and Construction~\ref{PD2} to the closed embedding $H[F^{4/3}]\hookrightarrow H$, and retain the notation introduced in Construction~\ref{PD2}. By Remark~\ref{alpha}, the class $\alpha_{y^p}\in H^1(B(H[F^{4/3}]),\mathcal{O})$ corresponds to $[y^p]\in H^1(\gr^0C^{\bullet})$. In the bicomplex \(B^{\bullet,\bullet}\), we have $d(y^p)=0$, while
$\delta(y^p)=y^p\otimes 1+1\otimes y^p-(y\otimes 1+1\otimes y+x\otimes z)^p=-x^p\otimes z^p\in \mathrm{Fil}^2_{\mathrm{PD}}D_{A^{\otimes 2}}(I_2)$. Thus the image of $y^p$ in $\mathrm{Fil}^0C^{\bullet}/\mathrm{Fil}^2C^{\bullet}$ defines a class  $v\in H^1(\mathrm{Fil}^0C^{\bullet}/\mathrm{Fil}^2C^{\bullet})$, which lifts the class $[y^p]\in H^1(\gr^0C^{\bullet})=E_1^{0,1}$. Therefore, we see that $d_1(\alpha_{y^p})=0$, and
$d_2(\alpha_{y^p})\in E_2^{2,0}$ is represented by
the class $\partial(v)\in H^2(\gr^2C^\bullet)=E_1^{2,0}$,
where $\partial$ is the boundary operator associated with
the short exact sequence of cochain complexes
\[
0\to\gr^2C^\bullet
\to\mathrm{Fil}^0C^\bullet/\mathrm{Fil}^3C^\bullet
\to\mathrm{Fil}^0C^\bullet/\mathrm{Fil}^2C^\bullet
\to0.
\] 
By the above computation, we can see that $\partial(v)=[d_{\mathrm{Tot}}(y^p)]=[-x^p\otimes z^p ]$, which corresponds to $-\eta_{x^p}\wedge \eta_{z^p} $ by a similar argument in  Proposition~\ref{Heisenberg}.

It remains to show that $\eta_{x^p}\wedge \eta_{z^p}\neq 0$ in $E^{2,0}_2$. Equivalently, we need to show $\eta_{x^p}\wedge \eta_{z^p}\in E^{2,0}_1$ is not in the image of the differential $d_1: E^{1,0}_1 \to E^{2,0}_1$. We know that
$E_1^{1,0}=H^0(B(H[F^{4/3}]),L_{B(H[F^{4/3}])/k})
=n_{H[F^{4/3}]}=
k\eta_{x^p}\oplus k\eta_{y^{p^2}}\oplus k\eta_{z^p}$
and
$E_1^{2,0}=H^0(B(H[F^{4/3}]),\Lambda^2L_{B(H[F^{4/3}])/k})
=\Lambda^2(n_{H[F^{4/3}]})$. Hence, we have $\eta_{x^p}\wedge\eta_{z^p}\neq 0$
in $E_1^{2,0}$. We next show that $d_1: E^{1,0}_1 \to E^{2,0}_1$ is a zero map.
  This can be verified by a computation similar to the one above. Alternatively, we can argue by functoriality. Consider three homomorphisms $g_1,g_2, g_3:H[F^{4/3}]\to \alpha_p$ sending $(x,y,z)$ to $x$, $y^p$ and $z$, respectively. Then we have $g_1^*(\eta_{t^p})=\eta_{x^p}$, $g_2^*(\eta_{t^p})=\eta_{y^{p^2}}$, and $g_3^*(\eta_{t^p})=\eta_{z^p}$. Since $d_1(\eta_{t^p})=0$ for $B\alpha_p$ by Proposition~\ref{alphap},
  functoriality gives
  $d_1(\eta_{x^p})=d_1(\eta_{y^{p^2}})
  =d_1(\eta_{z^p})=0$. This completes the proof.
\end{proof}

\subsection{$G^r_{a,b,c}$} We now consider variants of $H$ obtained by introducing \(p\)-powers into group multiplication. For a positive integer $r$, let $H^r=\mathbb{A}^3$  as a scheme. We can equip $H^r$ with an affine group scheme structure over $k$ as follows. For a \(k\)-algebra \(R\) and  $P=(x,y,z),P'=(x',y',z')\in H^r(R)$, we define $P\cdot P'=\bigl(x+x',\,y+y'+x^{p^r}z',\,z+z'\bigr) $. For integers $ a,b,c\geq 1$, let $G^r_{a,b,c} $ be the closed subscheme of $ H^r$ cut out by the ideal $(x^{p^a},y^{p^b},z^{p^c})$. One can verify that this closed subscheme is a closed subgroup scheme of $H^r$ provided that either $b\geq a-r$ or $b\geq c$. We consider the closed embedding  $G^r_{a,b,c}\hookrightarrow H^r$. We apply Theorem~\ref{PD1} and Construction~\ref{PD2} to this closed embedding, and retain the notation
introduced in Construction~\ref{PD2}. In this case, we have
\(I/I^2=(A/I)x^{p^a}\oplus(A/I)y^{p^b}\oplus(A/I)z^{p^c}\)
and \(n_{G^r_{a,b,c}}\simeq k\eta_{x^{p^a}}\oplus k\eta_{y^{p^b}}
\oplus k\eta_{z^{p^c}}\). Moreover, since the $G^r_{a,b,c}$-action on $n_{G^r_{a,b,c}}$ is trivial by Lemma~\ref{normal}.

We first recall a basic fact that will be used below.
\begin{lemma} \label{zero}
	Let $(B,J,\gamma)$ be a divided power ring such that $B$ is an $\mathbb F_p$-algebra. Then the \(p\)-th power of every element of $J$ is zero.

\end{lemma}
\begin{proof}
	 Let $b$ be an element of $J$. Since $B$ is an $\mathbb F_p$-algebra, we have $b^p=p!\gamma_p(b)=0$.
\end{proof}
 Suppose now $ b-1\geq a-r\geq c$. We apply Theorem~\ref{PD1} and Construction~\ref{PD2} to the closed embedding  $G^r_{a,b,c}\hookrightarrow H^r$, and retain the notation
introduced in Construction~\ref{PD2}. In this case, we have
\(I/I^2=(A/I)x^{p^a}\oplus(A/I)y^{p^b}\oplus(A/I)z^{p^c}\)
and \(n_{G^r_{a,b,c}}\simeq k\eta_{x^{p^a}}\oplus k\eta_{y^{p^b}}
\oplus k\eta_{z^{p^c}}\). Since $G^r_{a,b,c}$ is normal in $H^r$, the $G^r_{a,b,c}$-action on $n_{G^r_{a,b,c}}$ is trivial by Lemma~\ref{normal}. Moreover, since $
y^{p^c},y^{p^{c+1}},\cdots,y^{p^{b-1}}
 \in \operatorname{Ker}(\mathrm{pr}_1^*+\mathrm{pr}_2^*-m^*:\mathcal{O}(G^r_{a,b,c})\to \mathcal{O}(G^r_{a,b,c})\otimes \mathcal{O}(G^r_{a,b,c})$, they determine classes $\alpha_{y^{p^c}},\alpha_{y^{p^{c+1}}},\ldots,
 \alpha_{y^{p^{b-1}}}\in H^1(BG^r_{a,b,c},\mathcal O)$. We next study the differentials of these classes in the Hodge-to-de Rham spectral sequence of $BG^r_{a,b,c}$.

\begin{proposition}\label{conjugate}Consider the Hodge-to-de Rham spectral sequence of $BG^r_{a,b,c}$. For $c\leq s\leq b-1$, the classes
	$\alpha_{y^{p^s}}\in H^1(BG^r_{a,b,c},\mathcal O)$
	satisfy the following properties:
	 \begin{enumerate}
	 	\item[(1)] If $a-r>c$, then
	 	$d_1(\alpha_{y^{p^c}})\neq0$.
	 	
	 	\item[(2)] If $a-r=c$, then
	 	$d_1(\alpha_{y^{p^c}})=0$ and
	 	$d_2(\alpha_{y^{p^c}})\neq0$.
	 	
	 	\item[(3)] If $c+1\leq s\leq b-1$, then
	 	$d_1(\alpha_{y^{p^s}})=0$ and
	 	$d_2(\alpha_{y^{p^s}})=0$.
	 \end{enumerate}
	\end{proposition}

\begin{proof} Applying Theorem~\ref{PD1} and Construction~\ref{PD2} to the closed embedding  $G^r_{a,b,c}\hookrightarrow H^r$, and retain the notation introduced in Construction~\ref{PD2}. For $c\leq s\leq b-1$, by Remark~\ref{alpha}, the class $\alpha_{y^{p^s}}\in H^1(BG^r_{a,b,c},\mathcal O)$ corresponds to $[y^{p^s}]\in H^1(\gr^1C^{\bullet})$. 
  Since  $z^{p^s}=(z^{p^{s-1}})^p$ and $z^{p^{s-1}}\in \mathrm{Fil}_{\mathrm{PD}}^1D_{A^{\otimes 2}}(I_2)$, Lemma~\ref{zero} gives $z^{p^s}=0$ in $D_{A^{\otimes 2}}(I_2)$. Hence, when $c+1\leq s\leq b-1$, we have $\delta(y^{p^s})=-x^{p^{r+s}} \otimes z^{p^s}=0 \in D_{A^{\otimes 2}}(I_2)$. It follows that $y^{p^s}$ determines a class in $H^1(C^{\bullet})$ which lifts $[y^{p^s}]\in H^1(\gr^1C^{\bullet})$. Thus, the class $\alpha_{y^{p^s}}\in H^1(BG^r_{a,b,c},\mathcal O)$ is a permanent cycle.
	
When $a-r>c$, we have $\delta(y^{p^c})=-x^{p^{r+c}} \otimes z^{p^c} \in \mathrm{Fil}^1_{\mathrm{PD}}D_{A^{\otimes 2}}(I_2)$. In this case, the class $[x^{p^{r+c}}]\in H^1(\gr^0C^{\bullet})$ corresponds to a non-zero class $\alpha_{x^{p^{r+c}}}\in H^1(BG^r_{a,b,c},\mathcal{O})$ and $[z^{p^c}]\in H^1(\gr^1C^{\bullet})$ corresponds to  a non-zero class $\eta_{z^{p^c}}\in H^0(BG^r_{a,b,c} ,L_{BG^r_{a,b,c} /k})$. Hence, by the same argument as in Proposition~\ref{Heisenberg}, we obtain $d_1(\alpha_{y^{p^c}})=-\alpha_{x^{p^{r+c}}}\otimes \eta_{z^{p^{c}}}\neq 0$ in $H^1(BG^r_{a,b,c} ,\mathcal O)\otimes n_G=H^1(BG^r_{a,b,c} ,n_G)\subset H^1(BG^r_{a,b,c} ,L_{BG^r_{a,b,c} /k})$. When $a-r=c$, we have $\delta(y^{p^c})=-x^{p^{a}} \otimes z^{p^c}\in \mathrm{Fil}^2_{\mathrm{PD}} D_{A^{\otimes 2}}(I_2)$. In this case, the class $[x^{p^{a}}]\in H^1(\gr^1C^{\bullet})$ corresponds to a non-zero class $\eta_{x^{p^{a}}}\in H^0(BG^r_{a,b,c},L_{BG^r_{a,b,c} /k})$ and $[z^{p^c}]\in H^1(\gr^1C^{\bullet})$ corresponds to  a non-zero class $\eta_{z^{p^c}}\in H^0(BG^r_{a,b,c} ,L_{BG^r_{a,b,c} /k})$. Hence, by the same argument as in Proposition~\ref{Heisenberg2}, we have $d_1(\alpha_{y^{p^c}})=0$ and $d_2(\alpha_{y^{p^c}})=-\eta_{x^{p^{a}}}\wedge \eta_{z^{p^c}}\neq 0$ in $E^{2,0}_2$.

\end{proof}

\section{Classifying stacks with highly non-degenerate Hodge-to-de Rham spectral sequences}

In this section, we construct a non-commutative finite group scheme $G_n$ such that the Hodge-to-de Rham spectral sequence of $BG_n$ has a non-zero differential on the $E_n$-page.
\subsection{The classifying stacks $BG_n$}

\begin{construction}\label{example}

For \(n\geq 2\), consider the following group scheme \(\widetilde G_n\). As a scheme, we set
\[\widetilde{G}_{n}=\mathbb A^{4n-5}=\operatorname{Spec}k[x_1,\cdots,x_{n-2},y_1,\cdots,y_{n-1},z_1,\cdots,z_{n-1},w_1,\cdots,w_{n-1}].\]
We regard \(\widetilde G_n\) as a set-valued functor via its functor of points
\[
\widetilde G_n\colon \CRing_k\longrightarrow\Set,
\qquad
R\longmapsto\widetilde G_n(R)=R^{4n-5}.
\]
We equip \(\widetilde G_n\) with a group scheme structure by promoting the above functor to a group-valued functor $ \CRing_k \longrightarrow \Grp$.
The group law on \(\widetilde G_n(R)\) is defined as follows. Given \(P=(x_i,y_j,z_i,w_k)\) and \(P'=(x_i',y_j',z_i',w_k')\) in
\(\widetilde G_n(R)\), write
\(P\cdot P'=(x_i'',y_j'',z_i'',w_k'')\).
The coordinates are defined by
\(x_i''=x_i+x_i'\),
\(y_j''=y_j+y_j'\),
\(z_i''=z_i+z_i'\), and
\(w_1''=w_1+w_1'+y_1^p(z_1')^p\).
For \(2\leq k\leq n-1\), we set
\[
w_k''
=
w_k+w_k'
+z_{k-1}^p y_k'
+x_{k-1}(z_k')^p.
\]
Hence, we obtain a smooth affine group scheme $\widetilde G_n$ over $k$. We consider a closed subscheme 
\[
G_n
=V\bigl(
x_1^p,\ldots,x_{n-2}^p,
y_1^{p^2},\ldots,y_{n-1}^{p^2},
z_1^{p^2},\ldots,z_{n-1}^{p^2},
w_1^{p^2},\ldots,w_{n-1}^{p^2}
\bigr).
\]
It is straightforward to check that $G_n$ is a closed normal subgroup scheme of \(\widetilde G_n\).
\end{construction}
\begin{remark}
The group scheme $G_2$ can be viewed as a variant of $H[F^{4/3}]$. In fact, there is a group scheme homomorphism $G_2\to H[F^{4/3}]$ sending $(y_1,z_1,w_1)\in G_2(R)$ to $(x,y,z)=(y^p, w,z^p)\in H[F^{4/3}](R)$. 
\end{remark}

 We can verify $x_i, \; y_j, \; y_j^p, \; z_i, \; z_i^p, \; w_k^p \in \operatorname{Ker}(\mathrm{pr}_1^*+\mathrm{pr}_2^*-m^*:\mathcal{O}(G_{n+1})\to \mathcal{O}(G_{n+1})\otimes \mathcal{O}(G_{n+1}))$. It follows that $\alpha_{x_i},\alpha_{y_j},\alpha_{y_j^p}, \alpha_{z_j},\alpha_{z_j^p},\alpha_{w_j^p} \in H^1(BG_{n+1},\mathcal{O})$. Throughout this section, unless otherwise specified,
 $E_r^{i,j}$ denotes the $(i,j)$-term on the $r$-th page
 of the Hodge--de Rham spectral sequence of $BG_{n+1}$.

\begin{theorem}\label{thm:main1}
    For $n\geq 2$, the class
     \[\alpha_{n} = \alpha_{w_1^p}\alpha_{y_2^p}\alpha_{y_3^p}  \dots  \alpha_{y_{n-1}^p} \in H^{n-1}(BG_{n},\mathcal{O}) =E_1^{0,n-1}\] in the Hodge-to-de Rham spectral sequence of $BG_{n}$
    satisfies 
    \begin{enumerate}
    \item[(1)]$d_1(\alpha_{n})=0,\cdots,d_{n-1}(\alpha_{n})=0$
    \item[(2)]$d_{n}(\alpha_{n})\neq 0$ in $E_{n}^{n,0}$.
    \end{enumerate}
\end{theorem}

\begin{proof}[Proof of Theorem~\ref{thm:main1}(1)]

For ease of notation, we work with \(G_{n+1}\) for $n\geq 1$. We apply Theorem~\ref{PD1} and Construction~\ref{PD2} to the closed embedding $G_{n+1}\hookrightarrow \widetilde{G}_{n+1}$ and retain the notation introduced there.

By Remark~\ref{alpha}, the class $\alpha_{n+1}=\alpha_{w_1^p}\alpha_{y_2^p}\alpha_{y_3^p}  \dots  \alpha_{y_{n}^p}$ corresponds to $[w_1^p\otimes y_2^p\otimes\cdots\otimes y_{n}^p]\in H^n(\gr^0 C^{\bullet})$. We set $s_1=w_1^p\otimes y_2^p\otimes\cdots\otimes y_{n}^p$ and, for \(2\leq l\leq n\), $s_l=y_1^{p^2}\otimes x_1^p\otimes x_2^p\otimes\cdots\otimes
x_{l-2}^p\otimes w_l^p\otimes y_{l+1}^p\otimes\cdots\otimes y_n^p$.
Observe $\delta(w_1^p)=w_1^p\otimes 1+1\otimes w_1^p-(w_1\otimes 1+1\otimes w_1+y_1^p\otimes z_1^p)^p=-y_1^{p^2}\otimes z_1^{p^2}$ in the bicomplex $B^{\bullet,\bullet}$, while $\delta(y^p_i)=0$. By Leibniz rule of $\delta$, we have $\delta(s_1)=\delta(w_1^p\otimes y_2^p\otimes\cdots\otimes y_{n}^p)=\delta(w_1^p)\otimes y_2^p\otimes\cdots\otimes y_{n}^p=-y_1^{p^2}\otimes z_1^{p^2}\otimes y_2^p\otimes\cdots\otimes y_{n}^p$ in the bicomplex $B^{\bullet,\bullet}$. Similarly, for $2\leq l\leq n$, we have
\[
\begin{aligned}
	&\delta(s_l)=\delta\bigl(
	y_1^{p^2}\otimes x_1^p\otimes x_2^p\otimes\cdots\otimes
	x_{l-2}^p\otimes w_l^p\otimes y_{l+1}^p\otimes\cdots\otimes y_n^p
	\bigr) \\
	&\qquad =
	(-1)^l\Bigl(
	y_1^{p^2}\otimes x_1^p\otimes x_2^p\otimes\cdots\otimes
	x_{l-2}^p\otimes z_{l-1}^{p^2}\otimes y_l^p\otimes
	y_{l+1}^p\otimes\cdots\otimes y_n^p \\
	&\qquad\qquad\quad
	+
	y_1^{p^2}\otimes x_1^p\otimes x_2^p\otimes\cdots\otimes
	x_{l-2}^p\otimes x_{l-1}^p\otimes z_l^{p^2}\otimes
	y_{l+1}^p\otimes\cdots\otimes y_n^p
	\Bigr).
\end{aligned}
\]
Note that $x_i^p, y_j^{p^2},z_j^{p^2},w_j^{p^2}\in \mathrm{Fil}^1_{\mathrm{PD}}D_A(I)$. Consequently, we have $\delta(s_1+s_2+\cdots+s_n)=(-1)^l(y_1^{p^2}\otimes x_1^p\otimes x_2^p\otimes\cdots\otimes x_{n-1}^p\otimes z_n^{p^2})\in  \mathrm{Fil}_{\mathrm{PD}}^{n+1}D_{A^{\otimes n+1}}(I_{n+1})$. On the other hand, since each \(s_i\) is a \(p\)-th power, we have \(d(s_i)=0\) in the bicomplex $B^{\bullet,\bullet}$.

Therefore, we can define
\[\widetilde{\alpha}_{n+1}=[s_1+s_2+\cdots+s_n] \in H^n(\mathrm{Fil}^0C^{\bullet}/\mathrm{Fil}^{n+1}C^{\bullet}),\] whose image in $H^n(\gr^0)$ is $[s_1]=\alpha_{n+1}$. Hence, the class $\alpha_{n+1}$ satisfies $d_1(\alpha_{n+1})=0,\cdots,d_{n}(\alpha_{n+1})=0$.

\end{proof}
To describe the class $d_{n+1}(\alpha_{n+1})$
in $E_{n+1}^{n+1,0}$, we first identify
$n_{G_{n+1}}$ explicitly.
As an $A/I$-module, we have
\[
I/I^2
=
\bigoplus_{i=1}^{n-1}(A/I)x_i^p
\oplus
\bigoplus_{i=1}^{n}(A/I)y_i^{p^2}
\oplus
\bigoplus_{i=1}^{n}(A/I)z_i^{p^2}
\oplus
\bigoplus_{i=1}^{n}(A/I)w_i^{p^2},
\]
where the generators denote their classes modulo $I^2$. Since $d\colon I/I^2\to \Omega^1_{A/k}\otimes_A A/I$ is zero, the underlying $k$-vector space of $n_{G_{n+1}}$ is
\[
n_{G_{n+1}}
=
\bigoplus_{i=1}^{n-1}k\eta_{x_i^p}
\oplus
\bigoplus_{i=1}^{n}k\eta_{y_i^{p^2}}
\oplus
\bigoplus_{i=1}^{n}k\eta_{z_i^{p^2}}
\oplus
\bigoplus_{i=1}^{n}k\eta_{w_i^{p^2}}.
\]
Since $G_{n+1}$ is a normal subgroup scheme of $\widetilde{G}_{n+1}$, the $G_{n+1}$-action on $n_{G_{n+1}}$ is trivial by Lemma~\ref{normal}. Consequently, we have $E_{1}^{r,0}=H^0(BG_{n+1},\Lambda^{r}L_{BG_{n+1}/k})\simeq \Lambda^r(n_G)$ for every $r\geq 0$. In the following lemma, we give an explicit description of the class $d_{n+1}(\alpha_{n+1})$ in terms of $\Lambda^{n+1}(n_G)$.
\begin{lemma}\label{form}
In the Hodge-to-de Rham spectral sequence of $BG_{n+1}$, the class 
$d_{n+1}(\alpha_{n+1})\in E_{n+1}^{n+1,0}$ is represented by  $(-1)^{n}\eta_{y_1^{p^2}}\wedge\eta_{x_1^{p}}\wedge\cdots\wedge\eta_{x_{n-1}^{p}}\wedge\eta_{z_n^{p^2}}\in \Lambda^{n+1}(n_G)\simeq E_{1}^{n+1,0}$.
\end{lemma}

\begin{proof}
The class $d_{n+1}(\alpha_{n+1}) $ is represented by the image of $\widetilde{\alpha}_{n+1}$ under the boundary operator
$\partial:H^n(\mathrm{Fil}^0/\mathrm{Fil}^{n+1})\to H^{n+1}(\mathrm{Fil}^{0}/\mathrm{Fil}^{n+2})=H^{n+1}(\gr^{n+1}) $ associated with the short exact sequence of cochain complexes
\[ 0\to \gr^{n+1}C^{\bullet}\to \mathrm{Fil}^0C^{\bullet}/\mathrm{Fil}^{n+2}C^{\bullet} \to \mathrm{Fil}^{0}C^{\bullet}/\mathrm{Fil}^{n+1}C^{\bullet}\to 0.\]
By the computation in Theorem~\ref{thm:main1}, we have
$\partial(\widetilde{\alpha}_{n+1}^{(n+1)})=[d_{\mathrm{Tot}}(s_1+\cdots+s_n)]=[\delta(s_1+\cdots+s_n)]=(-1)^n [y_1^{p^2}\otimes x_1^p\otimes \cdots \otimes x_{n-1}^p\otimes z_{n}^{p^2}]\in H^{n+1}(\gr^{n+1})$. By Remark~\ref{product}, the element $y_1^{p^2}\otimes x_1^p\otimes \cdots \otimes x_{n-1}^p\otimes z_{n}^{p^2}$ is the product of $y_1^{p^2}, x_1^p, \cdots, x_{n-1}^p, z_{n}^{p^2}$ in the differential graded $k$-algebra $C^{\bullet}$. Hence, we have $[y_1^{p^2}\otimes x_1^p\otimes \cdots \otimes x_{n-1}^p\otimes z_{n}^{p^2}]=[y_1^{p^2}]\cdot [x_1^p] \cdots [x_{n-1}^p]\cdot [z_{n}^{p^2}]\in H^{n+1}(\gr^{n+1})$. By Remark~\ref{eta}, these classes  $[y_1^{p^2}], [x_1^p], \cdots , [x_{n-1}^p], [z_{n}^{p^2}]$ corresponds to  $\eta_{y_1^{p^2}},\eta_{x_1^{p}},\cdots,\eta_{x_{n-1}^{p}},\eta_{z_n^{p^2}}$, respectively. It follows that the class $(-1)^n [y_1^{p^2}]\cdot [x_1^p]\cdot \cdots \cdot [x_{n-1}^p]\cdot [z_{n}^{p^2}]$ corresponds to $(-1)^{n}\eta_{y_1^{p^2}}\wedge\eta_{x_1^{p}}\wedge\cdots\wedge\eta_{x_{n-1}^{p}}\wedge\eta_{z_n^{p^2}}\in\Lambda^{n+1}(n_G)$.
\end{proof}

\subsection{Torus action}

To establish $ \eta_{y_1^{p^2}}\wedge\eta_{x_1^{p}}\wedge\cdots\wedge\eta_{x_{n-1}^{p}}\wedge\eta_{z_n^{p^2}}\neq 0  \ \text{in} \   E_{n+1}^{n+1,0}$, we wish to use an approach based on torus weight as in \cite[Section~4.2]{ABM21}.
We first recall some facts about $B\mathbb{G}_m^l$ for $l\geq 0$.  We have a canonical isomorphism  
\[
\begin{aligned}
	\Hom_{\Grp}(\mathbb{G}_m^l,\mathbb{G}_m)
	&\simeq \mathbb{Z}^l,\\
	\left(
	\chi_\mu\colon
	(\lambda_1,\ldots,\lambda_l)
	\longmapsto
	\lambda_1^{a_1}\cdots\lambda_l^{a_l}
	\right)
	&\longmapsto
	\mu=(a_1,\ldots,a_l).
\end{aligned}
\]
For a $\mathbb{G}_m^l$-representation $V$, we  define $
V_\mu =V_{\wt=\mu}:
=
\left\{
v\in V \,\middle|\, t\cdot v=\chi_{-\mu}(t)v
\text{ for all } t\in \mathbb{G}_m^l
\right\},
$ for $\mu \in \mathbb{Z}^l$. Then $V$ has a decomposition $\displaystyle V=\bigoplus_{\mu\in \mathbb{Z}^l} V_{\mu}$. And for $\mathbb{G}_m^l$-representation $V,W$, we have $\displaystyle \Hom_{\mathsf{Rep}(\mathbb{G}_m^l)}(V, W)=\prod_{\mu\in \mathbb{Z}^l}\Hom(V_{\mu},W_{\mu})$; see, for example, \cite[\S I.2.11]{Jan03}.
Hence, we have an  equivalence of categories
\[
\begin{aligned}
 \QCoh(B\mathbb{G}_m^l)= \mathsf{Rep}(\mathbb{G}_m^l)\simeq \Fun(\mathbb{Z}^l, \Mod_k),\
	V
	&\longmapsto
(V_\mu)_{\mu\in\mathbb{Z}^l},
\end{aligned}
\]
Moreover, this equivalence can be enhanced to a symmetric monoidal equivalence. Note our convention  follows that in the convention in \cite[Construction 2.2.1]{Bha22}. Moreover, there is a derived version of the above equivalence;  see, for example, \cite[Construction~2.2.1]{Bha22}. More precisely, we have a symmetric monoidal equivalence. 
\[
\begin{aligned}
	\Fun\bigl(\mathbb Z^l,\mathsf D(k)\bigr)
	&\simeq
	\mathsf D(B\mathbb G_m^l),\
	F
	&\longmapsto
	\bigoplus_{\mu\in\mathbb Z^l}
	F(\mu)\otimes \mathcal O(-\mu),
\end{aligned}
\]
where $\mathcal O(\mu)$ is the line bundle on $B\mathbb G_m^l$ corresponding to one-dimensional $\mathbb G_m^l$-representation $\chi_{\mu}$. In what follows, we will call an element in $\mathbb{Z}^l$  a weight. For a weight $\mu\in \mathbb{Z}^l$ and a $\mathbb G_m^l$-representation $V$,  call $V_\mu =V_{\wt=\mu}$ weight-$\mu$ part of $V$. If an element $a\in V_\mu =V_{\wt=\mu}$, we will say $\wt(a)=\mu$. 

\begin{construction}\label{a1}
Let \(G\) be a finite type affine group scheme over \(k\), equipped with an action of \(\mathbb G_m^l\) by group scheme automorphisms. The action of \(\mathbb G_m^l\) on \(G\) induces an action of \(\mathbb G_m^l\) on \(BG\). The quotient $[BG/\mathbb G_m^l]$ is equivalent to $B(G\rtimes\mathbb G_m^l ) $ in $\Stk_{/B\mathbb{G}_m^l}$, where the base point \(B\mathbb G_m^l\to B(G\rtimes\mathbb G_m^l)\) is induced by the section \(\mathbb G_m^l\to G\rtimes\mathbb G_m^l\). By Example~\ref{action}, the filtered object \(\mathrm{Fil}_{\mathrm{Hod}}^{*}\dR_{BG/k}\) carries a natural \(\mathbb G_m^l\)-action. That is, the filtered object \(\mathrm{Fil}_{\mathrm{Hod}}^{*}\dR_{B(G\rtimes \mathbb G_m^l)/B\mathbb G_m^l}\) gives a \(\mathbb G_m^l\)-equivariant refinement of \(\mathrm{Fil}^*_{\mathrm{Hod}}\dR_{BG/k}\) in the sense that, denoting the base point \(*=\operatorname{Spec}(k)\to B\mathbb G_m^l\) by \(e_T\), we have \(e_T^*\mathrm{Fil}_{\mathrm{Hod}}^{*}\dR_{B(G\rtimes \mathbb G_m^l)/B\mathbb G_m^l}\simeq \mathrm{Fil}_{\mathrm{Hod}}^{*}\dR_{BG/k}\). Similarly, we have a $\mathbb G_m^l$-action on Hodge cohomology  of $BG$.

The filtered object
\(\mathrm{Fil}_{\mathrm{Hod}}^{*}
\dR_{B(G\rtimes \mathbb{G}_m^{l})/B\mathbb{G}_m^{l}}\)
in \(\mathsf{D}(B\mathbb{G}_m^{l})\) induces a spectral sequence in
\(\QCoh(B\mathbb G_m^{l})=\mathsf{Rep}(\mathbb G_m^{l})\)
\[
E_1^{i,j}
=
\mathcal{H}^j\!\left(
\pi_*\Lambda^i
L_{B(G\rtimes \mathbb{G}_m^{l})/B\mathbb{G}_m^{l}}
\right)
\Longrightarrow
\mathcal{H}^{i+j}\!\left(
\dR_{B(G\rtimes \mathbb{G}_m^{l})/B\mathbb{G}_m^{l}}
\right).
\]
By base change and the \(t\)-exactness of the pullback functor
\(e_T^*:\mathsf{D}(B\mathbb{G}_m^{l})\to\mathsf{D}(k)\),
after pulling back along \(e_T\), the above spectral sequence recovers
the Hodge-to-de Rham spectral sequence of \(BG\). In particular, we obtain a $\mathbb G_m^l$-equivariant refinement of the Hodge-to-de Rham spectral sequence of \(BG\).
\end{construction}

\begin{example}[{cf.\cite[Proposition 4.10]{ABM21}}]
The group scheme \(\alpha_{p^n}\) carries a natural
\(\mathbb G_m\)-action by group scheme automorphisms, 
\[
\mathbb{G}_m\curvearrowright \alpha_{p^n} \qquad \lambda\cdot x=\lambda x 
.\]
Hence, the Hodge cohomology of $B\alpha_{p^n}$ carries a \(\mathbb G_m\)-action.
For $\operatorname{char}(k) =2$, we have
	$\wt(\alpha_{t^{p^i}})=p^i$, $\wt(\eta_{t^{p^i}})=p^i$, and $\wt(\omega_t)=1$. For $\operatorname{char}(k) \neq 2$, we have
	$\wt(\alpha_{t^{p^i}})=p^i, \wt(\bar{\beta}(\alpha_{t^{p^i}}))=p^{i+1}$, $\wt(\eta_{t^{p^i}})=p^i$, and $\wt(\omega_t)=1$.
\end{example}

\begin{construction}\label{a2}

 Let us construct a $\mathbb{G}_m^{2n}$-action on $G_{n+1}$ by group scheme automorphisms . Note that $\widetilde{G}_{n+1}=\mathbb{A}^{4n-1}$ carries a natural \(\mathbb G_m^{4n-1}\)-action by coordinatewise scaling.  However, this action does not necessarily give rise to group automorphisms. Hence, we want to restrict this action to a closed subgroup scheme of $\mathbb{G}_m^{4n-1}$ such that this action gives rise to group automorphisms of  $\widetilde{G}_{n+1}$. Suppose $L\subseteq \mathbb{G}_m^{4n-1}$ is such a closed subgroup scheme. Then, for any $R\in \CRing_k$, an element $(\lambda_1,\cdots , \lambda_{4n-1})\in L(R)\subseteq \mathbb{G}_m^{4n-1}(R)$  should satisfy
 \[
 \lambda_{3n}w_1+\lambda_{3n}w_1'+\lambda_{n}^p\lambda_{2n}^py_1^p(z_1')^p
 =
 \lambda_{3n}(w_1+w_1'+y_1^p(z_1')^p)
 \]
 and, for \(2\leq k\leq n\),
 \[
 \lambda_{3n+k-1}w_k+\lambda_{3n+k-1}w_k'
 +\lambda_{2n+k-2}^p\lambda_{n+k-1}z_{k-1}^py_k'
 +\lambda_{k-1}\lambda_{2n+k-1}^px_{k-1}(z_k')^p 
 =
 \lambda_{3n+k-1}(w_k+w_k'
 +z_{k-1}^py_k'
 +x_{k-1}(z_k')^p )
 .\]
 Thus, we have $\lambda_{3n}=\lambda_{n}^p\lambda_{2n}^p$.
 And, for \(2\leq k\leq n\), we have $\lambda_{3n+k-1} =\lambda_{2n+k-2}^p\lambda_{n+k-1}=\lambda_{k-1}\lambda_{2n+k-1}^p$. Equivalently, we have $\lambda_{3n+k} =\lambda_{n+k}\lambda_{2n+k-1}^p$ and $\lambda_{k}=\lambda_{n+k}\lambda_{2n+k-1}^p\lambda_{2n+k}^{-p}$ for \(1\leq k\leq n-1\).

Hence, we define a subgroup functor
\(L_{n+1}\subseteq \mathbb G_m^{4n-1}\)
by setting, for every \(R\in\CRing_k\),
\[ L_{n+1}(R)=
 \left\{
 (\lambda_1,\ldots,\lambda_{4n-1})
 \in \mathbb G_m^{4n-1}(R)
 \;\middle|\;
 \begin{aligned}
 	\lambda_{3n}
 	&=\lambda_n^p\lambda_{2n}^p,\\
 	\lambda_{3n+k}&=\lambda_{n+k}\lambda_{2n+k-1}^p,\qquad &1\leq k\leq n-1,\\
 \lambda_{k}&=\lambda_{n+k}\lambda_{2n+k-1}^p\lambda_{2n+k}^{-p},
 	\qquad &1\leq k\leq n-1
 \end{aligned}
 \right\}.
 \]
 which defines a closed subgroup scheme of $\mathbb G_m^{4n-1}$, and
 \[
 \begin{aligned}
 	\mathbb G_m^{2n} \xrightarrow{\ \sim\ } L_{n+1},\
 	(\mu_1,\ldots,\mu_{2n})
 	\longmapsto
 	(\lambda_1,\ldots,\lambda_{4n-1}).
 \end{aligned}
 \]
 where the coordinates \(\lambda_k\) are given by
 \[
 \lambda_k=
 \begin{cases}
 	\mu_{k+1}\mu_{n+k}^{\,p}\mu_{n+k+1}^{-p},
 	& 1\leq k\leq n-1,\\[0.4em]
 	\mu_{k-n+1},
 	& n\leq k\leq 3n-1,\\[0.4em]
 	\mu_1^p\mu_{n+1}^p,
 	& k=3n,\\[0.4em]
 	\mu_{k-3n+1}\mu_{k-2n}^{\,p},
 	& 3n+1\leq k\leq 4n-1.
 \end{cases}
 \]
 Thus,  $\widetilde{G}_{n+1}$ carries a $\mathbb{G}_m^{2n}$-action by group automorphisms. One can check the action preserves closed subgroup scheme $G_{n+1}\subset \widetilde{G}_{n+1}$. Hence, the closed subgroup scheme $G_{n+1}$ inherits this $\mathbb{G}_m^{2n}$-action by group scheme automorphisms.
 
 \end{construction}
 \begin{lemma}\label{weight}With respect to the \(\mathbb G_m^{2n}\)-actions on
 	\(\mathcal O(G_{n+1})\) and \(\mathcal O(\widetilde G_{n+1})\), we have
 \[
 \begin{aligned}
 	&\wt(x_i)=e_{i+1} +pe_{n+i}-pe_{n+i+1}\\
 	&\wt(y_i)=e_{i}\\
 	&\wt(z_i)=e_{n+i}\\                            
 	&\wt(w_1)=pe_1+pe_{n+1}\\
 	&\wt(w_i)=e_i+pe_{n+i-1} \qquad ( 2\leq i\leq n )
 \end{aligned}
 \]
 \end{lemma}

\subsection{Analysis of $E_1^{l,n-l}$ and completion of the proof}
By Construction~\ref{a1} and Construction~\ref{a2}, the Hodge-to-de Rham spectral sequence of $BG_{n+1}$ admits a $\mathbb{G}_m^{2n}$-equivariant refinement, namely a spectral sequence in $\QCoh(B\mathbb{G}_m^{2n})=\mathsf{Rep}(\mathbb{G}_m^{2n})$ of the form 
\[
E_1^{i,j}
=
\mathcal{H}^j\!\left(
\pi_*\Lambda^i
L_{B(G_{n+1}\rtimes \mathbb{G}_m^{2n})/B\mathbb{G}_m^{2n}}
\right)
\Longrightarrow
\mathcal{H}^{i+j}\!\left(
\dR_{B(G_{n+1}\rtimes \mathbb{G}_m^{2n})/B\mathbb{G}_m^{2n}}
\right).
\]
Here $\pi:B(G_{n+1}\rtimes \mathbb G_m^{2n})\longrightarrow B\mathbb G_m^{2n}$ denotes the structure morphism.

\begin{notation}\label{notation-T}
	We denote by
	$s:B\mathbb G_m^{2n}
	\to B(G_{n+1}\rtimes\mathbb G_m^{2n})$
	the section of $\pi$ induced by the inclusion
	$\mathbb G_m^{2n}\hookrightarrow
	G_{n+1}\rtimes\mathbb G_m^{2n}$,
	and set
	$\mathcal L:=
	L_{B(G_{n+1}\rtimes\mathbb G_m^{2n})/B\mathbb G_m^{2n}}$.
\end{notation}
\begin{convention}\label{GM}
In what follows, we will not distinguish a \(\mathbb G_m^{2n}\)-representation from its underlying \(k\)-vector space when no confusion can arise, with the \(\mathbb G_m^{2n}\)-action understood. For example, when stating results, we will often write \(H^j(BG_{n+1},\Lambda^iL_{BG_{n+1}/k})\). When it is necessary to keep track of the equivariant structure explicitly, especially in the proofs, we will use the more precise notation \(\mathcal H^j\!\left(\pi_*\Lambda^i\mathcal L\right)\). We will follow the same convention for the other cohomology groups appearing below.
\end{convention}

By Lemma~\ref{weight}, we have $\wt(\eta_{y_1^{p^2}}\wedge\eta_{x_1^{p}}\wedge\cdots\wedge\eta_{x_{n-1}^{p}}\wedge\eta_{z_n^{p^2}})=\wt(\eta_{y_1^{p^2}})+\wt(\eta_{x_1^{p}})+\cdots+\wt(\eta_{x_{n-1}^{p}})+\wt (\eta_{z_n^{p^2}})=p^2e_1+p(e_2+\cdots+e_n)+p^2e_{n+1}$. 
\begin{notation}
	We write $\chi_{n+1}=p^2e_1+p(e_2+\cdots+e_n)+p^2e_{n+1}\in \mathbb Z^{2n}$. 
\end{notation}	
 We next need to understand weight-$\chi_{n+1}$ part of $E_1^{l,n-l}=H^{n-l}(BG_{n+1},\Lambda^l L_{BG_{n+1}/k})$. We first  record a lemma that will be used later.

\begin{lemma}\label{r}
	Let \(H\) and \(K\) be affine group schemes over \(k\), with \(K\) acting on \(H\) by group automorphisms, and let \(G=H\rtimes K\) be the corresponding semidirect product. Write
	$q:BG\longrightarrow BK$ for the morphism induced by the projection $G\to K$, and $s:BK\longrightarrow BG$	for the morphism induced by the section $K\hookrightarrow G$. Let \(V\in \QCoh(BG)\simeq \mathsf{Rep}(G)\). Suppose that the restriction of \(V\) to \(BH\) is a trivial $H$-representation. Then we have a canonical isomorphism $V\simeq q^*s^*(V)$.
\end{lemma}

\begin{proof}
	Let $\rho:G\to \operatorname{GL}(V)$ be the action. It suffices to show that the composite \[G=H\rtimes K\xrightarrow{\mathrm{pr}_2}K\xrightarrow{\rho|_K}\operatorname{GL}(V)\] agrees with $\rho$. For any $k$-algebra $R$ and $(h,k)\in G(R)$, we have $(h,k)=(h,1)(1,k)$. Hence, $\rho((h,k))=\rho((h,1))\rho((1,k))=\rho((1,k))$ by assumption. This completes the proof.
\end{proof}

By an argument similar to the proof of Lemma~\ref{split}, we obtain the following splitting.
\begin{lemma}\label{splitting}
	We have a canonical equivalence $\mathcal L\simeq \mathcal H^0(\mathcal L)\oplus \mathcal H^1(\mathcal L)[-1]$ in $\mathsf{D}(B(G_{n+1}\rtimes \mathbb G_m^{2n})).$ 
	
\end{lemma}

This splitting gives the following description of $\mathcal H^{n-l}\left(\pi_*\Lambda^l\mathcal L\right)$.

\begin{corollary}\label{subq}
	For $0\leq l\leq n$, there is a canonical isomorphism in $\QCoh(B\mathbb{G}_m^{2n})=\mathsf{Rep}(\mathbb{G}_m^{2n})$
\[
\mathcal H^{n-l}\left(\pi_*\Lambda^l\mathcal L\right)
\simeq
\bigoplus_{i=0}^l
\Lambda^i(s^*\mathcal H^0(\mathcal L))\otimes
\mathcal H^{n+i-2l}\!\left(
\pi_*\mathrm{Sym}^{l-i}\bigl(\mathcal H^1(\mathcal L)\bigr)
\right).
\]
\end{corollary}

\begin{proof}
By Lemma~\ref{splitting}, we have canonical equivalences	
\[
	\Lambda^l\mathcal L\simeq \bigoplus_{i=0}^l\Lambda^i\bigr(\mathcal H^0(\mathcal L)\bigr)\otimes\Lambda^{l-i} \bigl( \mathcal H^1(\mathcal L) [-1]\bigr)
	\simeq 
	\bigoplus_{i=0}^l
	\Lambda^i\mathcal H^0(\mathcal L)
	\otimes
	\mathrm{Sym}^{l-i}\bigl(\mathcal H^1(\mathcal L)\bigr)[i-l].
	\]
	Since the restriction of $\mathcal H^0(\mathcal L)$ to $BG_{n+1}$ is $\mathcal H^0(L_{BG_{n+1}/k})=n_{G_{n+1}}$, which is a trivial $G_{n+1}$-representation by Lemma~\ref{normal}, Lemma~\ref{r} gives a canonical isomorphism $\mathcal H^0(\mathcal L)\simeq \pi^* s^* \mathcal H^0(\mathcal L)$. Then we have
	\[
	\pi_*(\Lambda^l\mathcal L)
	\simeq
	\pi_*\!\left(
	\bigoplus_{i=0}^l
	\Lambda^i\bigl(\pi^* s^*\mathcal H^0(\mathcal L)\bigr)\otimes
	\mathrm{Sym}^{l-i}\bigl(\mathcal H^1(\mathcal L)\bigr)[i-l]
	\right)
	\simeq
	\bigoplus_{i=0}^l
	\Lambda^i\bigl(s^*\mathcal H^0(\mathcal L)\bigr)\otimes
	\pi_*\mathrm{Sym}^{l-i}\bigl(\mathcal H^1(\mathcal L)\bigr)[i-l].
	\]
	The last equivalence above  follows from the projection formula. Applying $\mathcal H^{n-l}$, we obtain the desired isomorphism.

\end{proof}

We next analyze the \(\mathbb G_m^{2n}\)-representation
$\mathcal H^{a}\!\left(
\pi_*\mathrm{Sym}^{b}\bigl(\mathcal H^1(\mathcal L)\bigr)
\right)$ for integers $a,b\geq 0$.
 
\begin{construction}
By the base change property for the cotangent complex, we see that $\mathcal{L}|_{BG_{n+1}}\simeq L_{BG_{n+1}/k}$. Hence, we obtain an isomorphism $\mathcal{H}^1(\mathcal{L})|_{BG_{n+1}}\simeq \mathcal{H}^1(L_{BG_{n+1}/k})= \omega_{G_{n+1}}$. In other words,  the \(G_{n+1}\rtimes\mathbb G_m^{2n}\)-representation
 \(\mathcal H^1(\mathcal L)\) restricts to the \(G_{n+1}\)-representation
 \(\omega_{G_{n+1}}\). Moreover, we have an isomorphism
 $
 \omega_{G_{n+1}}
 \simeq
 \omega_{\widetilde G_{n+1}}|_{BG_{n+1}},
 $
 where the \(\widetilde G_{n+1}\)-action on
 \(\omega_{\widetilde G_{n+1}}\) is induced by the conjugation action. Hence, the underlying $k$-vector space of $\omega_{G_{n+1}}$ is
 \[
 \omega_{G_{n+1}}
 =
 \bigoplus_{i=1}^{n-1}k\omega_{x_i}
 \oplus
 \bigoplus_{i=1}^{n}k\omega_{y_i}
 \oplus
 \bigoplus_{i=1}^{n}k\omega_{z_i}
 \oplus
 \bigoplus_{i=1}^{n}k\omega_{w_i},
 \]
 and a direct computation shows that the $\mathcal{O}(G_{n+1})$-comodule structure on $\omega_{G_{n+1}}$ is given by
 
\begin{equation}\label{eq:coaction}
\begin{aligned}
	\omega_{x_i}
	&\longmapsto \omega_{x_i}\otimes 1,
	&&1\leq i\leq n-1,\\
	\omega_{y_i}
	&\longmapsto \omega_{y_i}\otimes 1,
	&&1\leq i\leq n,\\
	\omega_{z_i}
	&\longmapsto \omega_{z_i}\otimes 1,
	&&1\leq i\leq n,\\
	\omega_{w_1}
	&\longmapsto \omega_{w_1}\otimes 1,\\
	\omega_{w_i}
	&\longmapsto
	\omega_{w_i}\otimes 1
	+\omega_{x_{i-1}}\otimes z_i^p
	-\omega_{y_i}\otimes z_{i-1}^p,
	&&2\leq i\leq n.
\end{aligned}
\end{equation}
 The subspace of $\omega_{G_{n+1}}$ spanned by
 $\omega_{x_i},\omega_{y_i},\omega_{z_i}$ is stable under both the $G_{n+1}$-action
 and the $\mathbb G_m^{2n}$-action, since these generators
 are $G_{n+1}$-invariant and each has a pure weight
 under the $\mathbb G_m^{2n}$-action. It is therefore stable under the
 $G_{n+1}\rtimes\mathbb G_m^{2n}$-action and defines a subrepresentation
 $\mathcal W_{n+1,1}\subset\mathcal H^1(\mathcal L)$.
Set
 $\mathcal W_{n+1,2}:=
 \mathcal H^1(\mathcal L)/\mathcal W_{n+1,1}$. Thus, we obtain a short exact sequence in $\QCoh(B(G_{n+1}\rtimes\mathbb G_m^{2n}))
 =\mathsf{Rep}(G_{n+1}\rtimes\mathbb G_m^{2n})$:
 \[
 0
 \longrightarrow
 \mathcal W_{n+1,1}
 \longrightarrow
 \mathcal H^1(\mathcal L)
 \longrightarrow
 \mathcal W_{n+1,2}
 \longrightarrow
 0.
 \]
 Moreover, the formulas~\eqref{eq:coaction} show that the restrictions of both
 $\mathcal W_{n+1,1}$ and $\mathcal W_{n+1,2}$
 to $BG_{n+1}$ are trivial $G_{n+1}$-representations.
\end{construction}

\begin{notation}
	For $i=1,2$, put $W_{n+1,i}:=s^*\mathcal W_{n+1,i}$,
	viewed as a $\mathbb G_m^{2n}$-representation, and set
	$W_{n+1}:=W_{n+1,1}\oplus W_{n+1,2}$. Its underlying $k$-vector space is
	\[
	W_{n+1}
	=
	\bigoplus_{i=1}^{n-1}k\omega_{x_i}
	\oplus
	\bigoplus_{i=1}^{n}k\omega_{y_i}
	\oplus
	\bigoplus_{i=1}^{n}k\omega_{z_i}
	\oplus
	\bigoplus_{i=1}^{n}k\omega_{w_i},
	\]
	where we also write $\omega_{w_i}$ for its image in
	$W_{n+1,2}$.
\end{notation}

\begin{lemma}\label{sq}
	For integers $a,b\geq 0$, the \(\mathbb G_m^{2n}\)-representation
	$\mathcal H^{a}\!\left(
	\pi_*\mathrm{Sym}^{b}\bigl(\mathcal H^1(\mathcal L)\bigr)
	\right)$ is a subquotient of $\mathrm{Sym}^b(W_{n+1})\otimes \mathcal H^{a}(\pi_*\mathcal{O}_{B(G_{n+1}\rtimes \mathbb{G}_m^{2n})})$ in $\QCoh(B\mathbb{G}_m^{2n})=\mathsf{Rep}(\mathbb{G}_m^{2n})$. Consequently, the \(\mathbb G_m^{2n}\)-representation $\mathcal H^{n-l}\left(\pi_*\Lambda^l\mathcal L\right)$ is a subquotient of $\displaystyle \bigoplus_{i=0}^l
	\Lambda^i(s^*\mathcal H^0(\mathcal L))\otimes
	\mathrm{Sym}^{l-i}(W_{n+1})\otimes \mathcal H^{n+i-2l}(\pi_*\mathcal{O}_{B(G_{n+1}\rtimes \mathbb{G}_m^{2n})})$ for every $0\leq l\leq n$.
\end{lemma}

\begin{proof}
	The \(G_{n+1}\rtimes \mathbb G_m^{2n}\)-representation $\mathrm{Sym}^{b}\bigl(\mathcal H^1(\mathcal L)\bigr)$ admits a $(b+1)$-step decreasing filtration
	\[0=F^{b+1}\subset F^{b}=\mathrm{Sym}^{b}(\mathcal W_{n+1,1})\subset F^{b-1}\subset \cdots \subset F^{0}=\mathrm{Sym}^{b}\bigl(\mathcal H^1(\mathcal L)\bigr),\]
	with graded pieces $\gr^i\bigl(\mathrm{Sym}^b(\mathcal H^1(\mathcal L))\bigr) \simeq \mathrm{Sym}^i(\mathcal W_{n+1,1}) \otimes \mathrm{Sym}^{b-i}(\mathcal{W}_{n+1,2})$ for $0\leq i\leq b$. Pushing forward this filtration along $\pi$, we obtain a $(b+1)$-step filtration $\pi_*F^\bullet$
	on $\pi_*\mathrm{Sym}^b(\mathcal H^1(\mathcal L))$
	with graded pieces
	$\gr^i\bigl(\pi_*\mathrm{Sym}^b(\mathcal H^1(\mathcal L))\bigr)\simeq
	\pi_*\bigl(
	\mathrm{Sym}^i(\mathcal W_{n+1,1})
	\otimes
	\mathrm{Sym}^{b-i}(\mathcal W_{n+1,2})
	\bigr)$
	for $0\leq i\leq b$. Since the restriction of each $\mathcal W_{n+1,j}$
	to $BG_{n+1}$ is a trivial $G_{n+1}$-representation, Lemma~\ref{r} gives
	$\mathcal W_{n+1,j}\simeq\pi^*s^*\mathcal W_{n+1,j}\simeq \pi^*W_{n+1,j}$
	for $j=1,2$. Hence, we have $\mathrm{Sym}^i(\mathcal{W}_{n+1,1}) \otimes \mathrm{Sym}^{b-i}(\mathcal{W}_{n+1,2})\simeq \pi^* \bigl(\mathrm{Sym}^i(W_{n+1,1}) \otimes \mathrm{Sym}^{b-i}(W_{n+1,2}) \bigr)$. The projection formula therefore yields $\pi_*\bigl(\mathrm{Sym}^i(\mathcal{W}_{n+1,1}) \otimes \mathrm{Sym}^{b-i}(\mathcal{W}_{n+1,2})\bigr)\simeq\mathrm{Sym}^i(W_{n+1,1}) \otimes \mathrm{Sym}^{b-i}(W_{n+1,2})\otimes \pi_*\mathcal{O}$.
	
	Since $\mathsf{Rep}(\mathbb G_m^{2n})$ is a semisimple abelian category,
	it follows from the above filtration that
	$\mathcal H^a\!\left(
	\pi_*\mathrm{Sym}^b(\mathcal H^1(\mathcal L))
	\right)$ is a subquotient of
	$\displaystyle \bigoplus_{i=0}^b
	\mathrm{Sym}^i(W_{n+1,1})\otimes
	\mathrm{Sym}^{b-i}(W_{n+1,2})\otimes
	\mathcal H^a(\pi_*\mathcal O)
	\simeq
	\mathrm{Sym}^b(W_{n+1})\otimes
	\mathcal H^a(\pi_*\mathcal O)$. Combining this with Lemma~\ref{subq}, we obtain the last assertion.
\end{proof}

 By Convention~\ref{GM}, we can reformulate the preceding result as follows: the \(\mathbb G_m^{2n}\)-representation $E_1^{l,n-l}= H^{n-l}(BG_{n+1},\Lambda^l L_{BG_{n+1}/k})$ is a subquotient of $ \displaystyle \bigoplus_{i=0}^l
\Lambda^i(n_{G_{n+1}})\otimes
\mathrm{Sym}^{l-i}(W_{n+1})\otimes H^{n+i-2l}(BG_{n+1},\mathcal{O})$. Here, by abuse of notation, we write $n_{G_{n+1}}$ for
the $\mathbb G_m^{2n}$-representation
$s^*\mathcal H^0(\mathcal L)$, whose underlying $k$-vector
space is canonically identified with that of $n_{G_{n+1}}$. We next analyze $H^{r}(BG_{n+1},\mathcal{O})$ as a $\mathbb G_m^{2n}$-representation.

\begin{construction}\label{equi} We can consider a closed subscheme $H_{n+1}=V(x_1,\ldots,x_{n-1},y_1,\ldots,y_n,z_1,\ldots, z_n)$ of $G_{n+1}$. One can check that the closed subscheme $H_{n+1}$ is in fact a closed subgroup scheme of $G_{n+1}$ and $H_{n+1}\simeq \alpha_{p^2}^{n}$ with coordinates $w_1,\cdots,w_n $ corresponding to the respective factors.
	Let $K_{n+1}=\alpha_p^{n-1}
	\times \alpha_{p^2}^{n}
	\times \alpha_{p^2}^{n}$ as a group scheme, with coordinates $x_1,\ldots ,x_{n-1},y_1,\ldots ,y_n,z_1, \ldots, z_n$ corresponding to the respective factors. Consider the natural projection map \(G_{n+1}\to K_{n+1}\) given by
	$(x_1,\ldots,x_{n-1},y_1,\ldots,y_n,z_1,\ldots,z_n,w_1,\ldots,w_n)
	\mapsto
	(x_1,\ldots,x_{n-1},y_1,\ldots,y_n,z_1,\ldots,z_n)$, which is a homomorphism of group schemes. It is easy to see that the kernel of this homomorphism is $H_{n+1}$. Hence, we obtain a short exact sequence of group schemes 
	\[ 1\to H_{n+1}\to G_{n+1}\to K_{n+1}\to 1.\] 
	
	It is easy to see that the $\mathbb{G}_m^{2n}$-action on $G_{n+1}$ induces  $\mathbb{G}_m^{2n}$-actions on $H_{n+1}$ and $K_{n+1}$. 
    Hence,   the above short exact sequence is $\mathbb{G}_m^{2n}$-equivariant and induces a cartesian square in $\mathsf{Stk}_{/B\mathbb{G}_m^{2n}}$:

\[
\begin{tikzcd}
{B(H_{n+1}\rtimes \mathbb{G}_m^{2n})}
\arrow[r,"f"]
\arrow[d,"h"']
\arrow[dr, phantom, "\lrcorner", very near start]
&
B(G_{n+1}\rtimes \mathbb{G}_m^{2n})
\arrow[d,"g"]
\\
B\mathbb{G}_m^{2n}
\arrow[r,"i"]
&
{B(K_{n+1}\rtimes \mathbb{G}_m^{2n})
}.
\end{tikzcd}
\]
This square can be viewed as a $\mathbb{G}_m^{2n}$-equivariant refinement of  the following cartesian square in $\mathsf{Stk}_k$
\[
\begin{tikzcd}
	{BH_{n+1}}
	\arrow[r,"f_1"]
	\arrow[d,"h_1"']
	\arrow[dr, phantom, "\lrcorner", very near start]
	&
	BG_{n+1}
	\arrow[d,"g_1"]
	\\
	*
	\arrow[r,"i_1"]
	&
	{BK_{n+1}
	}.
\end{tikzcd}
\]

\end{construction}
\begin{recall}(A review of the Leray spectral sequence; see \cite[Tag 0783]{Sta26}) Let
$X \xrightarrow{f} Y \xrightarrow{g} Z$ be a pair of composable morphisms of stacks  and $\mathcal{F}\in \mathsf{D}^{+}(X)$. Consider the Postnikov filtration of \(f_*\mathcal F\), which is exhaustive and increasing,
\[\cdots \to \tau^{\leq n-1}f_*\mathcal{F}\to \tau^{\leq n}f_*\mathcal{F}\to \cdots. \] Its graded pieces are $\gr_nf_*\mathcal{F}\simeq \mathcal H^n(f_*\mathcal{F})[-n]$. Applying $g_*$ to the increasing filtration, we obtain an increasing filtration on $(g\circ f)_*\mathcal{F}= g_*f_*\mathcal{F}$,
\[\cdots \to g_*\tau^{\leq n-1}f_*\mathcal{F}\to g_*\tau^{\leq n}f_*\mathcal{F}\to \cdots \]
with graded pieces $\gr_n\bigl((g\circ f)_*\mathcal{F}\bigr)\simeq g_* \mathcal H^n(f_*\mathcal{F})[-n]$. Note this filtration is exhaustive by left $t$-exactness of $g_*$.
By \cite[Proposition~1.2.2.7 and Proposition~1.2.2.14]{Lur17}, we obtain a convergent spectral sequence  in $\mathsf{QCoh}(Z)=\mathsf{D}(Z)^{\heartsuit}$
\[
E_2^{a,b}
=
\mathcal H^a \bigl(g_*\mathcal H^b (f_*\mathcal{F})\bigr)
\Longrightarrow
\mathcal H^{a+b}\bigl((g\circ f)_*\mathcal{F}\bigr).
\]
\end{recall}

\begin{proposition}\label{G}
We have a convergent  spectral sequence in $\QCoh(B\mathbb{G}_m^{2n})=\mathsf{Rep}(\mathbb{G}_m^{2n})$ of the form
\[E_2^{a,b}
\simeq H^a( BK_{n+1},\mathcal{O})\otimes H^b( BH_{n+1},\mathcal{O})\Longrightarrow H^{a+b}(BG_{n+1},\mathcal{O}).\]
Consequently, the $\mathbb{G}_m^{2n}$-representation $H^r(BG_{n+1},\mathcal{O})$ is a subquotient of  $\displaystyle \bigoplus_{a+b=r} H^a( BK_{n+1},\mathcal{O})\otimes H^b( BH_{n+1},\mathcal{O})$.
\end{proposition}
\begin{proof} Since $\mathsf{Rep}(\mathbb G_m^{2n})$ is a semisimple abelian category, the second assertion follows from the first one. It therefore remains to prove the first assertion.
	Consider the composable morphisms
	$
	B(G_{n+1}\rtimes \mathbb{G}_m^{2n})\xrightarrow{g}
	B(K_{n+1}\rtimes \mathbb{G}_m^{2n})\xrightarrow{\pi_K}B\mathbb{G}_m^{2n}
	$, where $\pi_K$  is the structure morphism from $B(K_{n+1}\rtimes\mathbb{G}_m^{2n})$ to $B\mathbb{G}_m^{2n}$.
	Applying the Leray spectral sequence to $g$ and $\pi_K$ with $\mathcal{F}=\mathcal{O}_{B(G_{n+1}\rtimes \mathbb{G}_m^{2n})}$, we obtain a convergent  spectral sequence in $\QCoh(B\mathbb{G}_m^{2n})=\mathsf{Rep}(\mathbb{G}_m^{2n})$
	\[ 
	E_2^{a,b} 
	=
	\mathcal H^a ((\pi_K)_*\mathcal H^b (g_*\mathcal{O})) 
	\Longrightarrow 
	\mathcal H^{a+b}(\pi_*\mathcal{O}). 
	\] 
    We next identify the \(E_2\)-page. Since \(H_{n+1}\) is contained in the center of \(G_{n+1}\), the $K_{n+1}$-action on $H^b(BH_{n+1},\mathcal O)$ induced by conjugation is trivial. In other words, the $K_{n+1}$-action on $\mathcal H^b(g_{1*}\mathcal{O})$ is trivial.
	By base change and $t$-exactness of $j^*$, we have $j^*\mathcal H^b (g_*\mathcal{O})\simeq \mathcal H^b(g_{1*}\mathcal{O})$.
	Thus $j^*\mathcal H^b (g_*\mathcal{O})$ is a trivial $K_{n+1}$-representation. It follows from Lemma~\ref{r} that there is an equivalence $\mathcal H^b (g_*\mathcal{O})\simeq \pi_K^*i^*\mathcal H^b(g_*\mathcal{O})$.
	We thus obtain equivalences
	\[ 
	(\pi_K)_*\mathcal H^b (g_*\mathcal{O})\simeq  (\pi_K)_*\pi_K^*i^*\mathcal H^b(g_*\mathcal{O})\simeq (\pi_K)_*\mathcal{O}\otimes i^*\mathcal H^b(g_*\mathcal{O}),\] 
	where the second equivalence follows from the projection formula. By base change and $t$-exactness of $i^*$, we have equivalences $i^*\mathcal H^b (g_*\mathcal{O})\simeq \mathcal H^b( i^* g_*\mathcal{O})\simeq \mathcal H^b( h_*f^*\mathcal{O})\simeq \mathcal H^b(h_*\mathcal{O})$.  Therefore, we see that
	$(\pi_K)_*\mathcal H^b (g_*\mathcal{O}) \simeq (\pi_K)_*\mathcal{O}\otimes \mathcal H^b(h_*\mathcal{O})$. Taking \(\mathcal H^a\), we obtain an identification	$E_2^{a,b}\simeq  \mathcal H^a ((\pi_K)_*\mathcal H^b (g_*\mathcal{O})) \simeq \mathcal H^a((\pi_K)_*\mathcal{O})\otimes \mathcal H^b(h_*\mathcal{O})$. 
	
\end{proof}

Suppose now $\operatorname{char}(k)=2$.
By the K\"unneth formula and Proposition~\ref{compute},
we have
$H^*(BK_{n+1},\mathcal O)\simeq
\mathrm{Sym}^*(H^1(BK_{n+1},\mathcal O))$
and
$H^*(BH_{n+1},\mathcal O)\simeq
\mathrm{Sym}^*(H^1(BH_{n+1},\mathcal O))$,
where
\begin{align*}
	H^1(BH_{n+1},\mathcal O)
	&=
	\bigoplus_{i=1}^n
	\left(k\alpha_{w_i}\oplus k\alpha_{w_i^p}\right),\\[0.5em]
	H^1(BK_{n+1},\mathcal O)
	&=
	\bigoplus_{i=1}^{n-1}k\alpha_{x_i}
	\oplus
	\bigoplus_{i=1}^n
	\left(k\alpha_{y_i}\oplus k\alpha_{y_i^p}\right)
	\oplus
	\bigoplus_{i=1}^n
	\left(k\alpha_{z_i}\oplus k\alpha_{z_i^p}\right).
\end{align*}
It follows that $H^r(BG_{n+1},\mathcal{O})$ is a subquotient of $\displaystyle \bigoplus_{i+j=r} \mathrm{Sym}^i\bigl(H^1( BK_{n+1},\mathcal{O})\bigr)\otimes \mathrm{Sym}^j\bigl(H^1( BH_{n+1},\mathcal{O})\bigr) \simeq\mathrm{Sym}^r\bigl((H^1(BH_{n+1},\mathcal{O}) \oplus H^1(BK_{n+1},\mathcal{O})\bigr)$ by Proposition~\ref{G}.
We denote
$H^1(BH_{n+1},\mathcal O)\oplus H^1(BK_{n+1},\mathcal O)$
by $V_{n+1}$, and $n_{G_{n+1}}$ by $U_{n+1}$.
By Lemma~\ref{sq}, the
$\mathbb G_m^{2n}$-representation $E_1^{l,n-l}=H^{n-l}(BG_{n+1},\Lambda^l L_{BG_{n+1}/k})$
is a subquotient of
\[
\bigoplus_{i=0}^l
\mathrm{Sym}^{n+i-2l}(V_{n+1})
\otimes
\mathrm{Sym}^{l-i}(W_{n+1})
\otimes
\Lambda^i(U_{n+1})
\]
for $0\leq l\leq n$. The following lemma describes the weight-$\chi_{n+1}$ part
of of the above direct sum for $0\leq l\leq n$.
\begin{lemma}\label{c1}
	Under the above assumption, for $0\leq l\leq n$, we have
\[
\begin{aligned}
& \bigl(\bigoplus_{0\leq i\leq l} \mathrm{Sym}^{n+i-2l}(V_{n+1})\otimes\mathrm{Sym}^{l-i}(W_{n+1})\otimes\Lambda^i(U_{n+1})\bigr)_{\wt=\chi_{n+1}}=\left(\operatorname{Sym}^{n-l}(V_{n+1}) \otimes {\Lambda}^l(U_{n+1})\right)_{\wt=\chi_{n+1}} \\
&= \begin{cases}
k\alpha_{w_1^p}\alpha_{y_2^p}\alpha_{y_3^p} \cdots \alpha_{y_n^p} , & l=0 \\
k\alpha_{w_{l+1}^p}\alpha_{y_{l+2}^p} \cdots \alpha_{y_n^p} \otimes \eta_{x_1^p} \wedge \cdots \wedge \eta_{x_{l-1}^p}\wedge \eta_{y_1^{p^2}} , & 1 \leq l \leq n-1 \\
0, & l=n
\end{cases}
\end{aligned}
\]
Here we use the convention that $\mathrm{Sym}^m(V_{n+1})=0$
for $m<0$, and that an empty product and an empty exterior
product are both equal to $1$.
\end{lemma}

\begin{proof} 
	Throughout the proof, \(p=2\), although we retain the notation \(p\) in 
	the weight formulas.

	Choose the monomial basis of 
	$ \displaystyle 
	\bigl(\bigoplus_{0\leq i\leq l} \mathrm{Sym}^{n+i-2l}(V_{n+1})\otimes\mathrm{Sym}^{l-i}(W_{n+1})\otimes\Lambda^i(U_{n+1})\bigr) 
	$ 
	induced by the chosen weight bases of
	$V_{n+1}$, $W_{n+1}$, and $U_{n+1}$. 
	Thus, a basis element is of the form  
	$ 
	u=a_1\cdots a_{n+i-2l}\otimes b_1\cdots b_{l-i} \otimes c_1\wedge\cdots\wedge c_i, 
	$ 
	where the \(a_j\)'s are basis elements of \(V_{n+1}\), the \(b_k\)'s are basis elements of \(W_{n+1}\), while the 
	\(c_m\)'s are distinct basis elements of \(U_{n+1}\). 
	
	For a weight 
	$ 
	\lambda=\sum_{i=1}^{2n}v_i(\lambda)e_i \in \mathbb{Z}^{2n}$, 
	we define 
	$ 
	\delta_1(\lambda) 
	= 
	\frac{v_1(\lambda)}{p} 
	+\sum_{i=2}^n v_i(\lambda), 
	\delta_2(\lambda) 
	= 
	\sum_{i=n+1}^{2n}v_i(\lambda). 
	$ 
	Recall that 
	$\chi_{n+1}=p^2e_1+p(e_2+\cdots+e_n)+p^2e_{n+1}. 
	$ 
	Hence, we have 
	$\delta_1(\chi_{n+1})=np$, $\delta_2(\chi_{n+1})=p^2$. 
	
	We first determine which generators can not occur in a monomial of weight 
	\(\chi_{n+1}\). Since all of the first \(n\) coordinates of the weights of the 
	basis generators are non-negative, no cancellation can occur in these 
	coordinates. Consequently, 
	$\eta_{w_1^{p^2}}$ 
	cannot occur, since its \(e_1\)-coordinate is \(p^3>p^2\). Similarly, 
	for \(2\leq i\leq n\), neither 
	$\eta_{w_i^{p^2}}$ nor $\eta_{y_i^{p^2}}$ 
	can occur, since their \(e_i\)-coordinates are \(p^2>p\). 
	
	After excluding these terms, every remaining basis generator \(c\) 
	satisfies $\delta_1(\wt(c))\leq p$. 
	For a monomial
	$u=a_1\cdots a_{n+i-2l}\otimes b_1\cdots b_{l-i}
	\otimes c_1\wedge\cdots\wedge c_i$
	of weight $\chi_{n+1}$, the total number of factors is $n+i-l$.
	Hence we have
	$np=\delta_1(\wt(u))\leq p(n+i-l)\leq np$.
	Thus $i=l$, and every factor of \(u\) must satisfy  
	$\delta_1(\wt(c))=p$. Hence, we have\[\bigl(\bigoplus_{0\leq i\leq l}\mathrm{Sym}^{n+i-2l}(V_{n+1})\otimes\mathrm{Sym}^{l-i}(W_{n+1})\otimes\Lambda^i(U_{n+1})\bigr)_{\wt=\chi_{n+1}}=\left(\operatorname{Sym}^{n-l}(V_{n+1}) \otimes {\Lambda}^l(U_{n+1})\right)_{\wt=\chi_{n+1}}.\] 
	It follows that the only possible generators in the symmetric part are 
	$ 
	\alpha_{w_1^p},\ldots,\alpha_{w_n^p}, 
	\alpha_{y_2^p},\ldots,\alpha_{y_n^p}, 
	$. 
	And the only possible generators in the exterior part are $\eta_{x_1^p},\ldots,\eta_{x_{n-1}^p}, \eta_{y_1^{p^2}}$.
	Moreover, each of the remaining generators can occur with multiplicity 
	at most one in the symmetric factor. Indeed, if \(\alpha_{w_1^p}\) occurred 
	twice, then $v_1(\wt(u))\geq 2p^2>p^2=v_1(\chi_{n+1})$.
	Similarly, if \(\alpha_{w_i^p}\) occurred twice for some \(2\leq i\leq n\), 
	then $v_i(\wt(u))\geq 2p>p=v_i(\chi_{n+1})$.
	The same argument shows that none of the elements 
	$
	\alpha_{y_2^p},\ldots,\alpha_{y_n^p} 
	$
	can occur more than once. Thus, every remaining generator appears with 
	multiplicity either \(0\) or \(1\). 
	
	We now distinguish the case $l=n$ from the case $0\leq l \leq n-1$.  If \(l=n\), there is no symmetric factor, and hence only choice of $u$ is 
	\[ \eta_{y_1^{p^2}}\wedge\eta_{x_1^p}\wedge\cdots\wedge\eta_{x_{n-1}^p} .\] 
	However, we have $
	\delta_2(\wt(\eta_{y_1^{p^2}}\wedge\eta_{x_1^p}\wedge\cdots\wedge\eta_{x_{n-1}^p}))=0\neq p^2. 
	$
	This proves that the weight-\(\chi_{n+1}\) subspace is zero when \(l=n\). 
	
	Assume from now on that \(0\leq l\leq n-1\). Among the remaining
	generators, only the elements
	\(\alpha_{w_1^p},\ldots,\alpha_{w_n^p}\)
	have non-zero \(\delta_2\)-weight, and each of them contributes exactly
	\(p^2\). Since \(\delta_2(\chi_{n+1})=p^2\), exactly one of the elements
	\(\alpha_{w_i^p}\) must occur in the symmetric factor.
	
	Assume first that \(l=0\). Since there is no exterior factor, the
	\(e_1\)-coordinate forces \(\alpha_{w_1^p}\) to occur. Comparing the
	\(e_{n+1},\ldots,e_{2n}\)-coordinates then shows that none of the factors
	\(\alpha_{w_i^p}\), \(2\leq i\leq n\), can occur: the factor
	\(\alpha_{w_1^p}\) already contributes the full \(e_{n+1}\)-coordinate of
	\(\chi_{n+1}\), whereas the \(e_{n+2},\ldots,e_{2n}\)-coordinates of \(\chi_{n+1}\)
	are zero. So this forces all the factors
	\(\alpha_{y_2^p},\ldots,\alpha_{y_n^p}\) to occur. Hence the unique
	monomial of weight \(\chi_{n+1}\) is
	\[
	\alpha_{w_1^p}\alpha_{y_2^p}\cdots\alpha_{y_n^p}.
	\]
	
	Now suppose that \(1\leq l\leq n-1\). We first claim that
	\(\alpha_{w_1^p}\) cannot occur. Indeed, if \(\alpha_{w_1^p}\) occurred,
	then \(\eta_{y_1^{p^2}}\) could not occur, since either one already
	contributes the full \(e_1\)-coordinate of \(\chi_{n+1}\). Hence, the exterior
	part would consist only of factors of the form \(\eta_{x_i^p}\). If \(i\)
	were the largest index for which \(\eta_{x_i^p}\) occurred, then the
	negative \(e_{n+i+1}\)-component contributed by \(\eta_{x_i^p}\) could not
	be cancelled. This contradicts the fact that the
	\(e_{n+2},\ldots,e_{2n}\)-coordinates of \(\chi_{n+1}\) are zero. Therefore
	\(\alpha_{w_1^p}\) does not occur, and \(\eta_{y_1^{p^2}}\) must occur.
	The remaining \(l-1\) exterior factors are of the form \(\eta_{x_i^p}\).
	
	Recall that
	$
	\wt(\eta_{x_i^p})
	=
	p e_{i+1}+p^2e_{n+i}-p^2e_{n+i+1},
	$
	while, for \(2\leq r\leq n\),
	 we have  $
	\wt(\alpha_{w_r^p})
	=
	p e_r+p^2e_{n+r-1}.
	$
	Since the target weight has \(p^2e_{n+1}\) as its only non-zero component
	among \(e_{n+1},\ldots,e_{2n}\), the last \(n\) coordinates force a
	telescoping chain
	$
	\eta_{x_1^p},\ \eta_{x_2^p},\ \ldots,\ \eta_{x_{r-2}^p},\
	\alpha_{w_r^p}.
$
	As there are exactly \(l-1\) exterior factors of the form
	\(\eta_{x_i^p}\), we must have \(r=l+1\). Thus, the exterior part is
	\[
	\eta_{x_1^p}\wedge\cdots\wedge\eta_{x_{l-1}^p}\wedge
	\eta_{y_1^{p^2}},
	\]
	and the distinguished symmetric factor is \(\alpha_{w_{l+1}^p}\).
	Finally, comparing the first \(n\) coordinates shows that the remaining
	symmetric factors must be
	\(\alpha_{y_{l+2}^p},\ldots,\alpha_{y_n^p}\). Consequently, the unique
	monomial of weight \(\chi_{n+1}\) is
	\[
	\alpha_{w_{l+1}^p}\alpha_{y_{l+2}^p}\cdots\alpha_{y_n^p}
	\otimes
	\eta_{x_1^p}\wedge\cdots\wedge\eta_{x_{l-1}^p}
	\wedge\eta_{y_1^{p^2}}.
	\]
\end{proof}

	As an immediate consequence, we obtain the following corollary.
\begin{corollary}
	When $\operatorname{char}(k)=2$, we have
	$\dim (E_1^{l,n-l})_{\wt=\chi_{n+1}}\leq 1$
	for $0\leq l\leq n-1$, while
	\mbox{$(E_1^{n,0})_{\wt=\chi_{n+1}}=0$}.
\end{corollary}

Suppose now $\operatorname{char}(k)\neq 2$.
By the K\"unneth formula and Proposition~\ref{compute},
we have
$H^*(BH_{n+1},\mathcal O)\simeq
\Lambda^*(H^1(BH_{n+1},\mathcal O))
\otimes \mathrm{Sym}^*(V'_{n+1,1})$
and
$H^*(BK_{n+1},\mathcal O)\simeq
\Lambda^*(H^1(BK_{n+1},\mathcal O))
\otimes \mathrm{Sym}^*(V'_{n+1,2})$,
where 
\[
V'_{n+1,1}
=
\bigoplus_{i=1}^n
\left(
k\bar{\beta}(\alpha_{w_i})
\oplus
k\bar{\beta}(\alpha_{w_i^p})
\right),
\]
and
\[
V'_{n+1,2}
=
\bigoplus_{i=1}^{n-1}
k\bar{\beta}(\alpha_{x_i})
\oplus
\bigoplus_{i=1}^n
\left(
k\bar{\beta}(\alpha_{y_i})
\oplus
k\bar{\beta}(\alpha_{y_i^p})
\right)
\oplus
\bigoplus_{i=1}^n
\left(
k\bar{\beta}(\alpha_{z_i})
\oplus
k\bar{\beta}(\alpha_{z_i^p})
\right).
\]
We denote
$H^1(BH_{n+1},\mathcal O)\oplus H^1(BK_{n+1},\mathcal O)$
by $V_{n+1}$,
$V'_{n+1,1}\oplus V'_{n+1,2}$ by $V'_{n+1}$,
and $n_{G_{n+1}}$ by $U_{n+1}$.
Hence, for every $r\geq0$, the
$\mathbb G_m^{2n}$-representation $H^r(BG_{n+1},\mathcal O)$
is a subquotient of the degree $r$ part of
$\Lambda^*(V_{n+1})\otimes\mathrm{Sym}^*(V'_{n+1})$,
where the grading is defined by placing
$V_{n+1}$ in degree $1$ and $V'_{n+1}$ in degree $2$. By Lemma~\ref{sq}, the
$\mathbb G_m^{2n}$-representation $E_1^{l,n-l}$
is therefore a subquotient of
\[
\bigoplus_{i=0}^l
\left(
\Lambda^*(V_{n+1})
\otimes \mathrm{Sym}^*(V'_{n+1})
\right)_{\mathrm{deg}=n+i-2l}
\otimes \mathrm{Sym}^{l-i}(W_{n+1})
\otimes \Lambda^i(U_{n+1})
\]
for $0\leq l\leq n$.
We now compute the weight-$\chi_{n+1}$ part
of the above direct sum for $0\leq l\leq n$.

\begin{lemma}\label{c2} 
	Under the above assumption, for $0\leq l\leq n$, we have
\[
\begin{aligned}
	&\left(
	\bigoplus_{i=0}^l
	\left(
	\Lambda^*(V_{n+1})
	\otimes \mathrm{Sym}^*(V_{n+1}')
	\right)_{\mathrm{deg}=n+i-2l}
	\otimes \mathrm{Sym}^{l-i}(W_{n+1})
	\otimes \Lambda^i(U_{n+1})
	\right)_{\wt=\chi_{n+1}}
	\\[0.5em]
	&\quad=
	\left(
	\Lambda^{n-l}(V_{n+1})
	\otimes \Lambda^l(U_{n+1})
	\right)_{\wt=\chi_{n+1}}
	\\[0.5em]
	&\quad=
	\begin{cases}
		k\alpha_{w_1^p}\alpha_{y_2^p}\cdots\alpha_{y_n^p},
		& l=0,\\[0.4em]
		k\alpha_{w_{l+1}^p}\alpha_{y_{l+2}^p}\cdots\alpha_{y_n^p}
		\otimes
		\eta_{x_1^p}\wedge\cdots\wedge\eta_{x_{l-1}^p}
		\wedge\eta_{y_1^{p^2}},
		& 1\leq l\leq n-1,\\[0.4em]
		0,
		& l=n.
	\end{cases}
\end{aligned}
\]
\end{lemma}

\begin{proof}
We use the same notation as in Lemma~\ref{c1}. By the same argument, the generators
$\eta_{w_i^{p^2}},\bar{\beta}(\alpha_{w_i^p})$ for  $1\leq i\leq n$ 
and $\eta_{y_i^{p^2}} ,\bar{\beta}(\alpha_{y_i^p})$ for  $2\leq i\leq n$
cannot occur in a monomial of weight $\chi_{n+1}$. After excluding these terms, every remaining basis generator \(c\)
satisfies
$
\delta_1(\wt(c))\leq p
$. Consider a basis monomial $u$ of weight $\chi_{n+1}$
in the $i$-th summand, with $a$ factors from $V_{n+1}$
and $b$ factors from $V_{n+1}'$.
Since these generators have degrees $1$ and $2$, respectively,
we have $a+2b=n+i-2l$.
Since $u$ has $l-i$ factors from $W_{n+1}$ and $i$ factors
from $U_{n+1}$, the total number of factors is
$a+b+l=n+i-l-b$.
Hence, we have
$np=\delta_1(\wt(u))\leq p(n+i-l-b)\leq np$.
Thus $i=l$ and $b=0$, so $u$ contains exactly $n$ factors,
and each factor $c$ satisfies $\delta_1(\wt(c))=p$. It follows that no generator from $V_{n+1}'$ or $W_{n+1}$ occurs. Consequently, we have
\[
\left(
\bigoplus_{i=0}^l
\left(
\Lambda^*(V_{n+1})\otimes \mathrm{Sym}^*(V_{n+1}')
\right)_{\mathrm{deg}=n+i-2l}
\otimes \mathrm{Sym}^{l-i}(W_{n+1})
\otimes \Lambda^i(U_{n+1})
\right)_{\wt=\chi_{n+1}}
=
\left(
\Lambda^{n-l}(V_{n+1})\otimes \Lambda^l(U_{n+1})
\right)_{\wt=\chi_{n+1}}.
\] The remaining argument is the same as in Lemma~\ref{c1}.
 
\end{proof}

As an immediate consequence, we obtain the following corollary.
\begin{corollary}
	When $\operatorname{char}(k)\neq 2$, we have
	$\dim (E_1^{l,n-l})_{\wt=\chi_{n+1}}\leq 1$
	for $0\leq l\leq n-1$, while
	\mbox{$(E_1^{n,0})_{\wt=\chi_{n+1}}=0$}.
\end{corollary}

The following lemma shows that the above dimension bounds are attained
in every positive characteristic and gives an explicit description of
$(E_1^{l,n-l})_{\wt=\chi_{n+1}}$
for $0\leq l\leq n$.

\begin{lemma}\label{bound}
 For  $0\leq l\leq n-1 $, we have
     $\dim (E_1^{l,n-l})_{\wt=\chi_{n+1}}= 1$ , while    $(E_1^{n,0})_{\wt=\chi_{n+1}}=0$. More precisely, we have \[(E_1^{l,n-l})_{\wt=\chi_{n+1}}=
     \begin{cases}
     	k\alpha_{w_1^p}\alpha_{y_2^p}\alpha_{y_3^p} \cdots \alpha_{y_n^p} , & l=0 \\
     	k\alpha_{w_{l+1}^p}\alpha_{y_{l+2}^p} \cdots \alpha_{y_n^p} \otimes \eta_{x_1^p} \wedge \cdots \wedge \eta_{x_{l-1}^p}\wedge \eta_{y_1^{p^2}} , & 1 \leq l \leq n-1 \\
     	0, & l=n
     \end{cases}
     \]
\end{lemma}

\begin{proof}Observe that $\alpha_{w_1^p}\alpha_{y_2^p}\alpha_{y_3^p} \cdots \alpha_{y_n^p}\in H^n(BG_{n+1},\mathcal O)=E_1^{0,n}$ and $\alpha_{w_{l+1}^p}\alpha_{y_{l+2}^p} \cdots \alpha_{y_n^p} \otimes \eta_{x_1^p} \wedge \cdots \wedge \eta_{x_{l-1}^p}\wedge \eta_{y_1^{p^2}}\in H^{n-l}(BG_{n+1},\mathcal O)\otimes \Lambda^l(n_{G_{n+1}})\simeq H^{n-l}(BG_{n+1},\Lambda^ln_{G_{n+1}})\subset H^{n-l}(BG_{n+1},\Lambda^lL_{BG_{n+1}/k})=E_1^{l,n-l}$ for $1 \leq l \leq n-1$. Hence, it suffices to show that the class $\alpha_{w_{l+1}^p}\alpha_{y_{l+2}^p} \cdots \alpha_{y_n^p}$ is nonzero in  $H^{n-l}(BG_{n+1},\mathcal O)$  for  $0\leq l\leq n-1 $.
    Consider the closed scheme $H'_{n+1}=V(z_1,\cdots,z_n)$ of $G_{n+1}$. One can check that $H'_{n+1}$ is a closed subgroup scheme of $G_{n+1}$ and $H'_{n+1}\simeq \alpha_p^{n-1}\times\alpha_{p^2}^{2n}$ with coordinates $x_1,\cdots, x_{n-1},y_1,\cdots, y_n,w_1,\cdots,w_n $ corresponding to the respective factors. By the K\"unneth formula and Proposition~\ref{compute}, we see that $\bigl(\alpha_{w_{l+1}^p}\alpha_{y_{l+2}^p}\cdots\alpha_{y_n^p}\bigr)|_{BH'_{n+1}}=\alpha_{w_{l+1}^p}\alpha_{y_{l+2}^p} \cdots \alpha_{y_n^p}\neq 0$ in $H^{n-l}(BH'_{n+1},\mathcal{O})$.

\end{proof}

\begin{proof}[Proof of Theorem~\ref{thm:main1}(2)]
 For $n=1$, we have seen in Lemma~\ref{form} the class $d_2(\alpha_2)\in E^{2,0}_2$ is represented by $\eta_{y_1^{p^2}}\wedge\eta_{z_1^{p^2}}\in \Lambda^2(n_{G_2})\simeq E^{2,0}_1$. Hence, it suffices to show that this nonzero class does not lie in the image of $d_1\colon E^{1,0}_1\to E^{2,0}_1$. Indeed, by Lemma~\ref{bound}, we have $(E^{1,0}_1)_{\wt=\chi_2}=0$. Since $d_1$ preserves weights and
 $\wt(\eta_{y_1^{p^2}}\wedge\eta_{z_1^{p^2}})=\chi_2$, we conclude that
 $d_2(\alpha_2)$ is non-zero in $E^{2,0}_2$.
 
For $n\geq 2$, we have seen in Lemma~\ref{form} that the class $d_{n+1}(\alpha_2)\in E^{n+1,0}_{n+1}$ is represented by $(-1)^{n}\eta_{y_1^{p^2}}\wedge\eta_{x_1^{p}}\wedge\cdots\wedge\eta_{x_{n-1}^{p}}\wedge\eta_{z_n^{p^2}}\in \Lambda^2(n_{G_2})\simeq E^{2,0}_1$. Moreover, we have  $\wt(\eta_{y_1^{p^2}}\wedge\eta_{x_1^{p}}\wedge\cdots\wedge\eta_{x_{n-1}^{p}}\wedge\eta_{z_n^{p^2}})=\chi_{n+1}$. We claim that $(E_2^{l,n-l})_{\mathrm{wt}=\chi_{n+1}}=0$ for $1\leq l \leq n$.
For $1\leq l\leq n-1$, the Leibniz rule gives
\[d_1(\alpha_{w_{l+1}^p}\alpha_{y_{l+2}^p} \cdots \alpha_{y_n^p} \otimes \eta_{x_1^p} \wedge \cdots \wedge \eta_{x_{l-1}^p}\wedge \eta_{y_1^{p^2}})=d_1(\alpha_{w_{l+1}^p})\alpha_{y_{l+2}^p} \cdots \alpha_{y_n^p} \otimes \eta_{x_1^p} \wedge \cdots \wedge \eta_{x_{l-1}^p}\wedge \eta_{y_1^{p^2}},\]
in $H^{n-l}(BG_{n+1},\mathcal{O})\otimes \Lambda^l(n_{G_{n+1}})\subset E_1^{l,n-l}$.  The class $d_1(\alpha_{w_{l+1}^p}) $ is represented by the image of $[w_{l+1}^p]$ under boundary operator
$\partial:H^1(\gr^0C^{\bullet})\to H^{2}(\gr^1C^{\bullet}) $ induced by the short exact sequence of  cochain complexes
\[ 0\to \gr^1C^{\bullet}\to \mathrm{Fil}^0C^{\bullet}/\mathrm{Fil}^2C^{\bullet} \to \gr^0C^{\bullet}\to 0.\] Note that $\delta(w_{l+1}^p)=-z_l^{p^2}\otimes y_{l+1}^p-x_l^p\otimes z_{l+1}^{p^2}$ and  $d(w_{l+1}^p)=0$ in the bicomplex $B^{\bullet,\bullet}$. Then we have
$\partial([w_{l+1}^p])
=[d_{\mathrm{Tot}}(w_{l+1}^p)]=-[z_l^{p^2}\otimes y_{l+1}^p +x_l^p\otimes z_{l+1}^{p^2} ]=-[z_l^{p^2}\otimes y_{l+1}^p]=-[z_l^{p^2}]\cdot[y_{l+1}^p]=[y_{l+1}^p]\cdot[z_l^{p^2}]\in H^1(\gr^1C^{\bullet})$. It follows that $d_1(\alpha_{w_{l+1}^p})= \alpha_{y_{l+1}^p}\otimes \eta_{z_l^{p^2}}$ in $H^1(BG_{n+1},\mathcal{O})\otimes n_G$. Thus, we have
\[d_1(\alpha_{w_{l+1}^p}\alpha_{y_{l+2}^p} \cdots \alpha_{y_n^p} \otimes \eta_{x_1^p} \wedge \cdots \wedge \eta_{x_{l-1}^p}\wedge \eta_{y_1^{p^2}})=(-1)^{n-l-1}\alpha_{y_{l+1}^p} \alpha_{y_{l+2}^p} \cdots \alpha_{y_n^p} \otimes\eta_{z_l^{p^2}}\wedge \eta_{x_1^p} \wedge \cdots \wedge \eta_{x_{l-1}^p}\wedge \eta_{y_1^{p^2}}.\] We claim this class is non-zero in $H^{n-l}(BG_{n+1},\mathcal O)\otimes \Lambda^{l+1}(n_{G_{n+1}})\subset H^{n-l}(BG_{n+1},\Lambda^{l+1}L_{BG_{n+1}/k})=E_1^{l+1,n-l}$. It suffices to show $\alpha_{y_{l+1}^p}\alpha_{y_{l+2}^p} \cdots \alpha_{y_n^p}\neq 0$ in $H^{n-l}(BG_{n+1},\mathcal O)$. Indeed, restricting to the closed subgroup scheme
$H'_{n+1}$ introduced in Lemma~\ref{bound}, we obtain
$\big(\alpha_{y_{l+1}^p}\alpha_{y_{l+2}^p} \cdots \alpha_{y_n^p}\big)|_{BH'_{n+1}}=\alpha_{y_{l+1}^p}\alpha_{y_{l+2}^p} \cdots \alpha_{y_n^p}\neq 0$. It is non-zero by the K\"unneth formula and Proposition~\ref{compute}. Thus, the vector space $(E_2^{l,n-l})_{\mathrm{wt}=\chi_{n+1}}=0$ for $1\leq l \leq n$.  Consequently, we obtain a chain of inclusions:
\[ 
(E_{n+1}^{n+1,0})_{\wt=\chi_{n+1}} \subseteq (E_{n}^{n+1,0})_{\wt=a_{n+1}} \subseteq (E_{n-1}^{n+1,0})_{\wt=\chi_{n+1}} \subseteq \dots \subseteq (E_{1}^{n+1,0})_{\wt=\chi_{n+1}} 
.\]
It follows that $\eta_{y_1^{p^2}}\wedge\eta_{x_1^{p}}\wedge\cdots\wedge\eta_{x_{n-1}^{p}}\wedge\eta_{z_n^{p^2}}\neq 0$  in   $E_{n+1}^{n+1,0}$. This completes the proof.

\end{proof}
\subsection{Designing such group schemes in practice}
In this subsection, we explain how to design, for each $n\geq 2$, a finite type group scheme $G$ over $k$ together with a class $\alpha_n\in H^{n-1}(BG,\mathcal O)$ that survives to the
$E_n$-page of the Hodge-to-de Rham spectral sequence of $BG$. Whenever possible, we also aim to arrange for $d_n(\alpha_n)\in E_n^{n,0}$ to be non-zero.

We first fix a smooth affine group scheme $\widetilde G$ together with a closed subgroup scheme $G$, and consider how to construct a class $\alpha_n\in H^{n-1}(BG,\mathcal O)$
that survives to the $E_n$-page of the Hodge-to-de Rham spectral sequence of $BG$. We apply Theorem~\ref{PD1} and Construction~\ref{PD2} to the closed embedding $G\hookrightarrow \widetilde{G}$ and retain the notation introduced there. We begin with the following observation.
\begin{observation}
	To construct a class in $H^{n-1}(BG,\mathcal O)$ that survives to the $E_n$-page, it suffices to construct an element  $t_{n-1}\in C^{n-1}$ such that  $d_{\mathrm{Tot}}(t_{n-1})\in \mathrm{Fil}^nC^n$. Indeed, such an element defines a class \([t_{n-1}]\in H^{n-1}(\mathrm{Fil}^0C^{\bullet}/\mathrm{Fil}^nC^{\bullet})\), whose image in $H^{n-1}(\gr^0C^{\bullet})\simeq H^{n-1}(BG,\mathcal O)$ gives the desired class $\alpha_n$. One natural approach to constructing such an element \(t_{n-1}\) is to choose elements \(s_1,\ldots,s_{n-1}\in C^{n-1}\) successively such that $s_i\in \mathrm{Fil}^{i-1}C^{n-1}$ for $1\leq i\leq n-1$ and the partial sums  $t_i=s_1+\cdots+s_i$, satisfy $d_{\mathrm{Tot}}(t_i)\in \mathrm{Fil}^{i+1}C^n$ for $1\leq i\leq n-1$.
\end{observation}
We next simplify the problem as follows.
\begin{simplification} Inspired by our study of $H[F^{4/3}]$ in Proposition~\ref{Heisenberg2}, we impose the additional condition that each $s_i$ lies in $D_{A^{\otimes n-1}}(I_{n-1})$ and is a $p$-th power. Then the de Rham differential satisfies $d(s_i)=0$, and hence $d_{\mathrm{Tot}}(t_i)=\delta(t_i)$. Consequently, the required condition becomes $\delta(t_i)\in \mathrm{Fil}_{\mathrm{PD}}^{i+1}D_{A^{\otimes n}}(I_{n})$. Using the canonical identifications
	$D_{A^{\otimes r}}(I_r)\simeq D_A(I)^{\otimes r}$, we hope, for simplicity, to find elements $s_i$ that have the form of a pure tensor
	$q_1^p\otimes\cdots\otimes q_{n-1}^p$,
	and require each $\delta(t_i)$ to be a pure tensor as well. 
\end{simplification}	
	
We next impose further conditions on $s_i$ to arrange for $d_n(\alpha_n)$ to be non-zero whenever possible.	
\begin{restrict}
	To arrange for \(d_n(\alpha_n)\in E_n^{n,0}\) to be nonzero  in the Hodge-to-de Rham spectral sequence of $BG$ as possible as we can, we need  $\delta(t_{n-1})=d_{\mathrm{Tot}}(t_{n-1})\notin \mathrm{Fil}^{n+1}_{\mathrm{PD}}D_{A^{\otimes n}}(I_{n})$. Indeed, otherwise $t_{n-1}$ determines a class in $H^0(\mathrm{Fil}^{0}C^{\bullet}/\mathrm{Fil}^{n+1}C^{\bullet})$ which lifts $\alpha_n$. Motivated by this observation, we strengthen our requirement by imposing the following condition: $\delta(t_i)\in \mathrm{Fil}^{i+1}_{\mathrm{PD}}D_{A^{\otimes n}}(I_{n})\setminus \mathrm{Fil}^{i+2}_{\mathrm{PD}}D_{A^{\otimes n}}(I_{n})$ for $1\leq i\leq n-1$. A natural way to achieve this is to require
	$\delta(t_i)$ to be a pure tensor
	$r_1\otimes\cdots\otimes r_n$,
	with exactly $i+1$ factors in
	$\mathrm{Fil}^1_{\mathrm{PD}}D_A(I)
	\setminus\mathrm{Fil}^2_{\mathrm{PD}}D_A(I)$
	and the remaining $n-i-1$ factors in
	$D_A(I)\setminus\mathrm{Fil}^1_{\mathrm{PD}}D_A(I)$, by Lemma~\ref{PD-tensor}.
	
On the other hand, to ensure that \(\alpha_n\neq 0\), we need \(s_1\notin \mathrm{Fil}^{1}_{\mathrm{PD}}D_A(I)\). It is natural to extend this requirement to
	$s_i\notin\mathrm{Fil}^i_{\mathrm{PD}}
	D_{A^{\otimes(n-1)}}(I_{n-1})$ for $1\leq i\leq n-1$.
	To achieve this, we seek $s_i$ of the form
	$q_1^p\otimes q_2^p\otimes\cdots\otimes q_{n-1}^p$,
	with exactly $i-1$ factors in
	$\mathrm{Fil}^1_{\mathrm{PD}}D_A(I)
	\setminus\mathrm{Fil}^2_{\mathrm{PD}}D_A(I)$
	and the remaining $n-i$ factors in
	$D_A(I)\setminus\mathrm{Fil}^1_{\mathrm{PD}}D_A(I)$, by Lemma~\ref{PD-tensor}.
\end{restrict}

\begin{lemma}\label{PD-tensor}
	Let $(A,I)$ be a ring-ideal pair, where $A$ is a $k$-algebra and $I$ is generated by a regular sequence. Let
	$(D_A(I),J,\gamma)$ be the divided power envelope of $(A,I)$.
	For $n\geq 1$, let
	$
	(D_{A^{\otimes n}}(I_n),J_n,\gamma)
	$
	be the divided power envelope of
	$(A^{\otimes n},I_n)$, where $I_n:=\ker\bigl(A^{\otimes n}\to (A/I)^{\otimes n}\bigr)$. For $0\leq l\leq n$, suppose that
	$a_1,\ldots,a_l\in J\setminus J^{[2]}$ and
	$a_{l+1},\ldots,a_n\in D_{A}(I)\setminus J$. Under the canonical identification $D_{A^{\otimes n}}(I_{n})\simeq D_{A}(I)^{\otimes n}$, the element
	$
	a_1\otimes\cdots\otimes a_n
	$
	lies in 
	$J_n^{[l]}\setminus J_n^{[l+1]}$.
\end{lemma}

\begin{proof} Since the first $l$ factors lie in $J$, we have $a_1\otimes\cdots\otimes a_n\in J_n^{[l]}$. It remains to show that 
	$a_1\otimes\cdots\otimes a_n\notin J_n^{[l+1]}$. Write $B=(A/I)^{\otimes n}$.
	Since $I$ is generated by a regular sequence, we have
	$\displaystyle
	J_n/J_n^{[2]}\simeq I_n/I_n^2 \simeq \bigoplus_{i=1}^n M_i 
	$, where $M_i=
	(A/I)^{\otimes(i-1)}
	\otimes I/I^2
	\otimes (A/I)^{\otimes(n-i)}$
	is the $B$-module obtained by putting $I/I^2$ in the $i$-th factor.
	Note that 
	\[ J_n^{[l]}/J_n^{[l+1]}\simeq 
	\Gamma_B^l(J_n/J_n^{[2]})
	\simeq
	\bigoplus_{i_1+\cdots+i_n=l}
	\Gamma_B^{i_1}(M_1)\otimes_B\cdots\otimes_B
	\Gamma_B^{i_n}(M_n).
	\]
	The image of
	$a_1\otimes\cdots\otimes a_n$ lies in the summand
	$
	M_1\otimes_B\cdots\otimes_B M_l\otimes_B B\otimes_B \cdots \otimes_B B
	$
	corresponding to
	$
	(i_1,\ldots,i_n)=(1,\ldots,1,0,\ldots,0).
	$
	
	There is a natural identification
	\[
	M_1\otimes_B\cdots\otimes_B M_l\otimes_B B\otimes_B \cdots \otimes_B B
	\simeq
	(I/I^2)^{\otimes l}
	\otimes (A/I)^{\otimes(n-l)} .
	\]
	Under this identification, the image of
	$a_1\otimes\cdots\otimes a_n$ corresponds to 
	$\bar a_1\otimes\cdots\otimes\bar a_l
	\otimes \bar a_{l+1}\otimes\cdots\otimes\bar a_n
	$ in $(I/I^2)^{\otimes l} \otimes (A/I)^{\otimes(n-l)}$
Here $\bar a_i$ denotes the image of $a_i$ in $J/J^{[2]}\simeq I/I^2$ for $1 \leq i \leq l$ and $\bar a_j$ denotes the image of $a_j$ in $D_A(I)/J\simeq A/I$ for $l+1 \leq j \leq n$. Since they are all non-zero by the assumption, we have	$\bar a_1\otimes\cdots\otimes\bar a_l
\otimes \bar a_{l+1}\otimes\cdots\otimes\bar a_n \neq 0
$ in $(I/I^2)^{\otimes l} \otimes (A/I)^{\otimes(n-l)}$. Consequently, we have $a_1\otimes\cdots\otimes a_n\notin J_n^{[l+1]}$.
\end{proof}

We now seek to construct a smooth affine group scheme $\widetilde G$ together with a closed subgroup scheme $G$ and the elements $s_i$ so as to arrange for $d_n(\alpha_n)$ to be non-zero.

\begin{construction}
	We now proceed to choose the elements $s_i$ stage by stage, imposing additional conditions at each stage.
	
	\medskip
	\noindent\textit{Stage $1$.} To make $\delta(s_1)$ a pure tensor of the form
	$q_1^p\otimes\cdots\otimes q_n^p$, it is natural to consider
	a pure tensor
	$s_1=e_1^p\otimes b_1^p\otimes\cdots\otimes b_{n-2}^p$
	such that $\delta(b_i)=0$ for all $i$ and
	$\delta(e_1)=-a_1\otimes c_1$.
	Indeed, the Leibniz rule for $\delta$ gives
	$\delta(s_1)
	=\delta(e_1^p)\otimes b_1^p\otimes\cdots\otimes b_{n-2}^p
	=-a_1^p\otimes c_1^p\otimes b_1^p\otimes\cdots\otimes b_{n-2}^p$. To meet the requirement that
	$\delta(s_1)\in
	\mathrm{Fil}^{2}_{\mathrm{PD}}D_{A^{\otimes n}}(I_n)
	\setminus
	\mathrm{Fil}^{3}_{\mathrm{PD}}D_{A^{\otimes n}}(I_n)$,
	we impose the conditions
	$a_1^p,c_1^p\in
	\mathrm{Fil}^{1}_{\mathrm{PD}}D_A(I)
	\setminus\mathrm{Fil}^{2}_{\mathrm{PD}}D_A(I)$
	and
	$b_1^p,\ldots,b_{n-2}^p\in
	D_A(I)\setminus\mathrm{Fil}^{1}_{\mathrm{PD}}D_A(I)$.
	To meet the requirement that
	$s_1\in
	\mathrm{Fil}^{0}_{\mathrm{PD}}D_{A^{\otimes(n-1)}}(I_{n-1})
	\setminus
	\mathrm{Fil}^{1}_{\mathrm{PD}}D_{A^{\otimes(n-1)}}(I_{n-1})$,
	we further require
	$e_1^p\in D_A(I)\setminus\mathrm{Fil}^{1}_{\mathrm{PD}}D_A(I)$. For $n=2$, this already
	completes the construction.
	
	\medskip
	\noindent\textit{Stage $2$.}
	Suppose $n\geq 3$. We need to choose $s_2\in \mathrm{Fil}^{1}D_{A^{\otimes n-1}}(I_{n-1})$ such that $\delta(s_2)=a_1^p\otimes c_1^p\otimes b_1^p \otimes \cdots \otimes b_{n-2}^p+ u$, where $u$ is some element in $\mathrm{Fil}^3_{\mathrm{PD}}D_{A^{\otimes n-1}}(I_{n-1})$. We hope to obtain $u$, up to a sign, from $\delta(s_1)$
	by replacing $c_1^p\otimes b_1^p$ with $r_1^p\otimes r_2^p$,
	where
	$r_1^p,r_2^p\in\mathrm{Fil}_{\mathrm{PD}}^1D_A(I)$.
	Write $a_2=r_1$ and $c_2=r_2$.
	One way to achieve this is to let
	$s_2=a_1^p\otimes e_2^p\otimes b_2^p\otimes\cdots\otimes b_{n-2}^p$ with $\delta(a_1)=0$ and $\delta(e_2)=-c_1\otimes b_1 - a_2\otimes c_2 $. Indeed, the Leibniz rule for $\delta$ gives
	\[
	\begin{aligned}
		\delta(s_1+s_2)
		&=-a_1^p\otimes c_1^p\otimes b_1^p\otimes\cdots\otimes b_{n-2}^p-a_1^p\otimes\delta(e_2^p)\otimes b_2^p
		\otimes\cdots\otimes b_{n-2}^p\\
		&=-a_1^p\otimes c_1^p\otimes b_1^p\otimes\cdots\otimes b_{n-2}^p+a_1^p\otimes c_1^p\otimes b_1^p\otimes\cdots\otimes b_{n-2}^p
		+a_1^p\otimes a_2^p\otimes c_2^p
		\otimes b_2^p\otimes\cdots\otimes b_{n-2}^p\\
		&=a_1^p\otimes a_2^p\otimes c_2^p
		\otimes b_2^p\otimes\cdots\otimes b_{n-2}^p.
	\end{aligned}
	\]
	To meet the requirement that $\delta(t_2)=a_1^p\otimes a_2^p\otimes c_2^p \otimes b_2^p\otimes\cdots\otimes b_{n-2}^p\in
	\mathrm{Fil}^{3}_{\mathrm{PD}}D_{A^{\otimes n}}(I_n)
	\setminus
	\mathrm{Fil}^{4}_{\mathrm{PD}}D_{A^{\otimes n}}(I_n)$, we impose the condition $a_2^p,c_2^p\in \mathrm{Fil}^{1}_{\mathrm{PD}}D_{A}(I)\setminus\mathrm{Fil}^{2}_{\mathrm{PD}}D_{A}(I)$. To meet the requirement that
	$s_2\in
	\mathrm{Fil}^{1}_{\mathrm{PD}}D_{A^{\otimes(n-1)}}(I_{n-1})
	\setminus
	\mathrm{Fil}^{2}_{\mathrm{PD}}D_{A^{\otimes(n-1)}}(I_{n-1})$,
	we further require $e_2^p\in D_{A}(I)\setminus\mathrm{Fil}^{1}_{\mathrm{PD}}D_{A}(I)$.
	 
	Continuing in this way, we arrive at the following construction.
	We seek elements
	$a_i,c_i,e_i\in D_A(I)$ for $1\leq i\leq n-1$
	and $b_j\in D_A(I)$ for $1\leq j\leq n-2$,
	subject to the following conditions:
	
\begin{enumerate}
	\item We have $\delta(a_i)=0$ for $1\leq i\leq n-2$
	and $\delta(b_j)=0$ for $1\leq j\leq n-2$.
	
	\item $\delta(e_1)=-a_1\otimes c_1$.
	
	\item For $2\leq i\leq n-1$, we have
	$\delta(e_i)=-c_{i-1}\otimes b_{i-1}-a_i\otimes c_i$.
	
	\item We have
	$a_i^p,c_i^p\in
	\mathrm{Fil}_{\mathrm{PD}}^1D_A(I)
	\setminus\mathrm{Fil}_{\mathrm{PD}}^2D_A(I)$
	and
	$b_j^p,e_i^p\in D_A(I)\setminus
	\mathrm{Fil}_{\mathrm{PD}}^1D_A(I)$
	for all relevant indices.
\end{enumerate}

Equivalently, the first three conditions can be expressed
in terms of the comultiplication as follows:
\begin{enumerate}
	\item We have
	$m^*(a_i)=a_i\otimes1+1\otimes a_i$
	for $1\leq i\leq n-2$ and
	$m^*(b_j)=b_j\otimes1+1\otimes b_j$
	for $1\leq j\leq n-2$.
	
	\item We have
	$m^*(e_1)=e_1\otimes1+1\otimes e_1+a_1\otimes c_1$.
	
	\item For $2\leq i\leq n-1$, we have
	$m^*(e_i)=e_i\otimes1+1\otimes e_i
	+c_{i-1}\otimes b_{i-1}+a_i\otimes c_i$.
\end{enumerate}
	
Then we set $s_1=e_1^p\otimes b_1^p \otimes \cdots \otimes b_{n-2}^p$, $s_i=a_1^p\otimes\cdots \otimes a_{i-1}^p\otimes e_i^p\otimes b_{i}\cdots \otimes b_{n-2}$ for $2\leq i\leq n-1$. Hence the class  $\alpha_n=[s_1]=[e_1^p\otimes b_1^p \otimes \cdots \otimes b_{n-2}^p]\in H^{n-1}(\gr^0C^{\bullet})\simeq H^{n-1}(BG,\mathcal O)$ survives to the $E_n$-page. Moreover, by the same argument as in Lemma~\ref{form}, we see that
$d_n(\alpha_n)\in E^{n,0}_n$ is represented by  $[d_{\mathrm{Tot}}(s_1)]=(-1)^{n-1}[a_1^p\otimes \otimes \cdots a_{n-1}^p\otimes c_{n-1}^p] \in H^{n}(\gr^nC^{\bullet}) $.

\end{construction}

In practice, we realize these conditions by choosing
$a_i,c_i,e_i\in A$ for $1\leq i\leq n-1$
and $b_j\in A$ for $1\leq j\leq n-2$,
requiring the same comultiplication identities to hold in
$A\otimes A$, and imposing
$a_i^p,c_i^p\in I\setminus I^2$
and $b_j^p,e_i^p\in A\setminus I$
for all relevant indices. Indeed, the canonical identifications $I/I^2\simeq \mathrm{Fil}_{\mathrm{PD}}^1D_A(I)
/\mathrm{Fil}_{\mathrm{PD}}^2D_A(I)$ and $A/I\simeq D_A(I)/\mathrm{Fil}_{\mathrm{PD}}^1D_A(I)$ ensure the required conditions.

\begin{example} A natural way to construct a closed embedding of affine group schemes
	$G\hookrightarrow\widetilde G$ together with the above data is to take
	the underlying scheme of $\widetilde G$ to be an affine space with $a_i,b_j,c_i,e_i$ as coordinates, and define
	the group law by the above comultiplication identities. For simplicity, we further require $a_{n-1}$ and all the $c_i$
	to be primitive. More precisely, let $\widetilde{G}=\mathbb{A}^{4n-5}$ as a scheme with coordinates $x_i,z_i,w_i$ for $1\leq i\leq n-1$
	and $y_j$ for $1\leq j\leq n-2$. Define a comultiplication on $\mathcal{O}(\widetilde G)$ by declaring $x_i,y_j,z_i$ to be primitive and setting
	\[
	\begin{aligned}
		\Delta(w_1)
		&=w_1\otimes1+1\otimes w_1+x_1\otimes z_1,\\
		\Delta(w_i)
		&=w_i\otimes1+1\otimes w_i
		+z_{i-1}\otimes y_{i-1}+x_i\otimes z_i
		\qquad (2\leq i\leq n-1).
	\end{aligned}
	\] One checks that these formulas define a group scheme structure
	on $\widetilde G$. For integers $l,m\geq 2$, we set $I=(x^p_i,y_i^{p^l},z_i^p,w_i^{p^m})$. It is straightforward to check $G:=\mathrm{Spec}(A/I)$ is a closed subgroup scheme of $\widetilde G$. Taking $a_i=x_i$, $c_i=z_i$, $e_i=w_i$, and $b_j=y_j$, we obtain the required data. Then the corresponding class $\alpha_n=[w_1^p\otimes y_1^p\otimes \cdots \otimes y_{n-2}^p]\in H^{n-1}(\gr^0C^{\bullet})\simeq H^{n-1}(BG,\mathcal O)$ survives to the $E_n$-page. Moreover, one can show $d_n(\alpha_n)\neq 0$ in $E_n^{n,0}$ by a torus weight argument. The group schemes in Construction~\ref{example} can be constructed in a similar way.

\end{example}

\section{Smooth projective varieties with highly non-degenerate Hodge-to-de Rham spectral sequences}
  In this section, we will construct smooth projective varieties $X_n$ of dimension $n$ over $k$ whose Hodge-to-de Rham spectral sequence has a non-zero differential on the \(E_n\)-page.

We first recall the notion of Hodge $d$-equivalences from \cite[Definition~5.1]{ABM21}.
\begin{definition}[Hodge $d$-equivalences]

    We will say that a morphism $X\to Y$ of syntomic Artin stacks over $k$ is a 
	 Hodge $d$-equivalence if, for each $s \geq 0$,
	 we have 
    $$\cofib \left(
	 R\Gamma(Y,\Lambda^s L_{Y/k})\rightarrow R\Gamma(X,\Lambda^s L_{X/k})
	 \right) \in 
	 \mathsf{D}(k)^{\geq d-s}. $$
\end{definition}

We first record the following auxiliary lemma on spectral sequences.
\begin{lemma}\label{lemma:ss-comparison}
	Let \(\varphi\colon E\to F\) be a morphism of spectral sequences, and write
	$
	\varphi_r\colon E_r^{i,j}\to F_r^{i,j}
	$
	for the induced map on the \(E_r\)-page. Suppose that, for some $r\geq 1$,
	$\varphi_r$ is an isomorphism for $i+j\leq d-1$ and is injective for
	$i+j=d$. Then the same holds for $\varphi_{r+1}$.
\end{lemma}

\begin{proof}
	
	Suppose first that $i+j\leq d-1$.  Since 
	$\varphi_r\colon E_r^{i,j}\longrightarrow F_r^{i,j}$ is an isomorphism for $i+j\leq d-1$ and is injective for
	$i+j=d$, the map on $d_r$-cycles
	\[
	\ker\left(
	d_r\colon E_r^{i,j}\to E_r^{i+r,j-r+1}
	\right)
	\longrightarrow
	\ker\left(
	d_r\colon F_r^{i,j}\to F_r^{i+r,j-r+1}
	\right)
	\]
	is an isomorphism.  Similarly, the map on $d_r$-boundaries in bidegree $(i,j)$ is an
	isomorphism, as these boundaries are the images of the differentials from $E_r^{i-r,j+r-1}$ and $F_r^{i-r,j+r-1}$, and $\varphi_r\colon E_r^{i-r,j+r-1}\to F_r^{i-r,j+r-1}$ is an isomorphism. It follows that
	$
	\varphi_{r+1}\colon E_{r+1}^{i,j}\longrightarrow F_{r+1}^{i,j}
$
	is an isomorphism for $i+j<d$.
	
	Now suppose that $i+j=d$. Since 
	$\varphi_r\colon E_r^{i,j}\longrightarrow F_r^{i,j}$  is injective for $i+j=d$, the map on $d_r$-cycles
	\[
	\ker\left(
	d_r\colon E_r^{i,j}\to E_r^{i+r,j-r+1}
	\right)
	\longrightarrow
	\ker\left(
	d_r\colon F_r^{i,j}\to F_r^{i+r,j-r+1}
	\right)
	\]
	 is injective. Similarly, the map on $d_r$-boundaries in bidegree $(i,j)$ is an
	 isomorphism, as these boundaries are the images of the differentials from $E_r^{i-r,j+r-1}$ and $F_r^{i-r,j+r-1}$ and $\varphi_r\colon E_r^{i-r,j+r-1}\to F_r^{i-r,j+r-1}$ is an isomorphism.
	 It follows that
	 $
	 \varphi_{r+1}\colon E_{r+1}^{i,j}\longrightarrow F_{r+1}^{i,j}
	 $
	 is injective for $i+j=d$.
\end{proof}

We now state a refinement of the Antieau--Bhatt--Mathew approximation theorem, which  is essentially contained in the argument of \cite[Theorem 1.2 and Proposition 5.11]{ABM21}.
\begin{theorem}[Antieau--Bhatt--Mathew]\label{ABM}
Suppose \(G\) is a finite group scheme over $k$. For any integer \(d\geq 0\), there exists a smooth
projective \(k\)-scheme \(X\) of dimension \(d\), together with a map
\(X\to BG\) such that, writing \(E_r^{i,j}(Z)\) for the \((i,j)\)-term on the \(E_r\)-page of the Hodge-to-de Rham spectral sequence of \(Z\), the induced map between the Hodge-to-de Rham spectral sequences
\[
E_r^{i,j}(BG)
\longrightarrow
E_r^{i,j}(X)
\]
 is injective for \(i+j\leq d\), for every $r\geq 1$.
\end{theorem}
\begin{proof}
By the argument in the proof of  \cite[Theorem 1.2]{ABM21},  for an integer $d\geq 0$,  there exists a finite
dimensional representation $V$ of $G$, and a $d$-dimensional smooth projective scheme $X$ with a closed embedding $f:X\hookrightarrow [\mathbb P(V)/G]$ which is a Hodge $d$-equivalence. We take the composite  map $X\hookrightarrow [\mathbb P(V)/G]\to BG$ as the desired map.

We denote $[\mathbb P(V)/G]$ by $Y$ for notational simplicity.  By the definition of a Hodge $d$-equivalence, for each $s \geq 0$,
we have  
\[\cofib(R\Gamma(Y,\Lambda^s L_{Y/k})\to R\Gamma(X,\Lambda^s L_{X/k})) \in \mathsf{D}(k)^{\geq d-s}.\]
It follows that  
$ H^j(Y,\Lambda^i L_{Y/k})\to H^j(X,\Lambda^i L_{X/k})  $ 
is an isomorphism for $i+j\leq d-1$ and  injective for $i+j= d$. Equivalently, we see that
$f^*\colon E_1^{i,j}(Y)\to E_1^{i,j}(X)$ is an isomorphism for $i+j\leq d-1$ and injective for $i+j=d$. By repeatedly applying Lemma~\ref{lemma:ss-comparison}, we see that $f^*\colon E_r^{i,j}(Y)\to E_r^{i,j}(X)$ is an isomorphism for $i+j\leq d-1$ and injective for $i+j=d$, for every $r\geq 1$.

It remains to show $ E_r^{i,j}(BG)\rightarrow E_r^{i,j}(Y)$ is injective for $i+j\leq d$ and $r\geq 1$. Let $\pi\colon Y=[\mathbb P(V)/G]\to BG$
be the structure map, and put $n=\dim_k(V)$. The stack $Y=[\mathbb P(V)/G]$ is the
projective bundle over $BG$ associated with the vector bundle corresponding to the representation \(V\). We first observe the first Chern class $c_1:=c_1(\mathcal{O}(1))\in E_1^{1,1}(Y)=H^1\bigl(Y,L_{Y/k})$ is a permanent cycle in the Hodge-to-de Rham spectral
sequence. Indeed, the first Chern class \(c_1\) admits a lift along the natural map
$
H^2(\mathrm{Fil}^1_{\mathrm{Hod}}\dR_{Y/k})
\longrightarrow
H^2(\gr^1_{\mathrm{Hod}}\dR_{Y/k})
=
H^1(Y,L_{Y/k})
$; see, for example,  \cite[Tag~0FME]{Sta26} or \cite[Lecture~14]{Cla21}. Now we claim that  $E_r^{*,*}(Y)$ is a free module over $E_r^{*,*}(BG)$ on $1,c_1,\cdots,c_1^{n-1}$, where $c_1:=c_1(\mathcal{O}(1))$. We  prove this claim by induction. For $r=1$, the claim comes from the projective bundle formula in \cite[Proposition~5.11]{ABM21}.
Suppose that the claim holds for all \(r\leq m\). We just need to show that
\[
\bigoplus_{l=0}^{n-1}
E_{m+1}^{\,i-l,j-l}(BG)
\xrightarrow{\ \sum_{l=0}^{n-1}\pi^*(-)\cdot c_1^l\ }
E_{m+1}^{\,i,j}\bigl(Y\bigr)
\]is an isomorphism. 
 Consider the following diagram:
 \[
\begin{tikzcd}[column sep=large]
\cdots \arrow[r,"d_m"] &
\displaystyle\bigoplus_{l=0}^{n-1}
E_m^{\,i-l-m,j-l+m-1}(BG)
\arrow[r,"d_m"]
\arrow[d,"\sum_l \pi^*(-)\cdot c_1^l"'] &
\displaystyle\bigoplus_{l=0}^{n-1}
E_m^{\,i-l,j-l}(BG)
\arrow[r,"d_m"]
\arrow[d,"\sum_l \pi^*(-)\cdot c_1^l"'] &
\displaystyle\bigoplus_{l=0}^{n-1}
E_m^{\,i-l+m,j-l-m+1}(BG)
\arrow[r,"d_m"]
\arrow[d,"\sum_l \pi^*(-)\cdot c_1^l"'] &
\cdots
\\
\cdots \arrow[r,"d_m"] &
E_m^{\,i-m,j+m-1}\bigl(Y\bigr)
\arrow[r,"d_m"] &
E_m^{\,i,j}\bigl(Y\bigr)
\arrow[r,"d_m"] &
E_m^{\,i+m,j-m+1}\bigl(Y\bigr)
\arrow[r,"d_m"] &
\cdots .
\end{tikzcd}
\]
Since $c_1$ is a permanent cycle, the Leibniz rule shows
that this diagram defines a map of complexes. By the induction hypothesis,
the vertical maps are isomorphisms. Thus we obtain an
isomorphism of complexes.
Taking cohomology, we obtain an isomorphism
\[
\bigoplus_{l=0}^{n-1}
E_{m+1}^{\,i-l,j-l}(BG)
\xrightarrow{\ \sum_{l=0}^{n-1}\pi^*(-)\cdot c_1^l\ }
E_{m+1}^{\,i,j}\bigl(Y\bigr).
\] Hence the claim holds for \(r=m+1\), and the induction is complete.

\end{proof}

\begin{remark}
    There is a more conceptual explanation of this theorem in \cite[Theorem 4.14]{LY26}.
\end{remark} 
 
\begin{remark}
	When $d\geq 1$, we can show that the smooth projective \(k\)-scheme $X$ appearing in Theorem~\ref{ABM} is geometrically integral. In other words, \(X\) is a smooth projective variety over \(k\). In fact, we have $H^0(X,\mathcal O)\simeq H^0([\mathbb P(V)/G],\mathcal O) \simeq H^0(BG,\mathcal O)\simeq k$. Here the first isomorphism follows from the fact that $X\to [\mathbb P(V)/G]$ is a Hodge $d$-equivalence, while the second isomorphism follows from the projective bundle formula. Hence, by \cite[Tag~0FD3]{Sta26}, we see that $X$ is geometrically integral.
\end{remark}

Applying Theorem~\ref{ABM} to $BG_{n}$, we obtain the following result.
    \begin{theorem}\label{thm:main 2}
 For $n\geq 2$, there exists a smooth projective variety $X_n$ over $k$ of dimension $n$, and there exists a class $\alpha_{n} \in H^{n-1}(X_{n},\mathcal{O}) =E_1^{0,n-1}$ in the Hodge-to-de Rham spectral sequence of $X_n$
    which satisfies $d_1(\alpha_{n}),\cdots,d_{n-1}(\alpha_{n})=0$ and $d_{n}(\alpha_{n})\neq 0$ 
in $E^{n,0}_n$.
\end{theorem}

Applying Theorem~\ref{ABM} to $B(H[F])$ and $B\alpha_p$, we obtain the following result.
\begin{proposition}\label{surface}
There is a smooth projective surface $S_1$ over $k$ with a global 1-form $\omega\in H^0(S_1,\Omega^1_{S_1/k})=E_1^{1,0}$ in the Hodge-to-de Rham spectral sequence of $S_1$ which satisfies $d_1(\omega)\neq 0$ in $H^0(S_1,\Omega^2_{S_1/k})=E^{2,0}_1$.  Moreover, there is a smooth projective surface $S_2$ over $k$ with a class $\alpha\in H^1(S_2,\mathcal{O})=E_1^{0,1}$ in the Hodge-to-de Rham spectral sequence of $S_2$ which satisfies $d_1(\alpha)\neq 0$ in $H^1(S_2,\Omega^1_{S_2/k})=E^{1,1}_1$.
\end{proposition}

By taking products of $X_n$ $(n\geq 2)$, $S_1$, or $S_2$ with several copies of an arbitrary elliptic curve, we obtain the following corollary.

\begin{corollary}\label{cor}
      Given \(n\geq 1\) and \(i,j\geq 0\) satisfying   $ j\geq n-1$ and \((i,j)\neq(0,0)\), there exists a smooth projective variety $Y$ over $k$ and a class $\beta  \in H^{j}(Y,\Omega^i_{Y/k}) =E_1^{i,j}$ in the Hodge-to-de Rham spectral sequence of $Y$
   which satisfies $d_r(\beta)=0 $ for $1\leq r<n$ and $d_{n}(\beta)\neq 0$ in $E^{n+i,j-n+1}_n$.

\end{corollary}

\begin{proof}
First, suppose $n\geq 2$. 
Let $E$ be an arbitrary elliptic curve over $k$, and $\gamma \in H^1(E,\mathcal{O}  )$ be a nonzero class, and $\omega \in H^0(E,\Omega^1_{E/k}  )$ be a nonzero differential form. It is well known that   the Hodge-to-de Rham spectral sequence of a smooth projective curve degenerates at the $E_1$-page \cite[Corollary~5.2]{Oda69}. Therefore, we have $d_r(\gamma)=0$ and $d_r(\omega)=0$ for $r\geq 1$. Let $Y=X_n\times E^{i+j-n+1}$. Applying the K\"{u}nneth formula to $Y=X_n\times E^{i+j-n+1}$, we obtain
$  E_r^{*,*}(Y)\simeq E_r^{*,*}(X_n)\otimes E_r^{*,*}(E) ^{\otimes i+j-n+1 } $  for $r\geq 1$, and the differentials are compatible with the isomorphism in the evident sense.
Therefore, we define $\beta=\alpha_n\otimes \gamma^{\otimes j-n+1}\otimes\omega^{\otimes i} $ as a class in $E_1^{i,j}(Y)$. Then it is easy to see $d_r(\beta)=d_r(\alpha_n)\otimes \gamma^{\otimes j-n+1}\otimes\omega^{\otimes i}=0$ for $1 \leq r\leq n-1$, and  $d_n(\beta)=d_n(\alpha_n)\otimes \gamma^{\otimes j-n+1}\otimes\omega^{\otimes i}\neq 0$ in $E^{n,0}_n(X_n) \otimes E^{0,1}_n(E)^{\otimes j-n+1}\otimes E^{1,0}_n(E)^{\otimes i} \subset E^{n+i,j-n+1}_n(Y).$

For the case $n=1$, applying the same argument to $S_1$ and $S_2$ , we obtain the desired result.

\end{proof}

\begin{remark}\label{CY}
	To the best of our knowledge, prior to the present work, the only example we have found in the literature among smooth projective varieties in positive characteristic for which the Hodge-to-de Rham spectral sequence is known to have a non-zero differential beyond the first page is Hirokado's Calabi-Yau threefold $\mathcal H$ in characteristic $3$; see \cite{Hir99}. In \cite[Theorem~3.6(iii)]{Eke04}, Ekedahl computed its Hodge numbers as follows.
	\[
	\begin{array}{c|cccc}
		h^{i,j}(\mathcal H) & i=0 & i=1 & i=2 & i=3\\
		\hline
		j=0 & 1 & 0 & 0 & 1\\
		j=1 & 0 & 42 & 0 & 0\\
		j=2 & 0 & 0 & 42 & 0\\
		j=3 & 1 & 0 & 0 & 1
	\end{array}
	\]
	Therefore, the only possible non-zero differentials in the Hodge-to-de Rham spectral sequence of $\mathcal H$	are $
	d_2^{0,3}\colon H^3(\mathcal H,\mathcal O)\to H^2(\mathcal H,\Omega_{\mathcal H/k}^2)
	$ and 
	$d_2^{1,1}\colon H^1(\mathcal H,\Omega_{\mathcal H/k}^1)\to H^0(\mathcal H,\Omega_{\mathcal H/k}^3)
	$. Lam  later showed \cite[Proof of Theorem~5.4]{Lam20} that the conjugate spectral sequence of $\mathcal H$ has a non-zero differential on the $E_2$-page, and hence  the Hodge-to-de Rham spectral sequence does not degenerate at the $E_1$-page. Consequently, at least one of $d_2^{0,3}$ and $d_2^{1,1}$ is non-zero. By Serre duality, these two differentials  are dual to each other up to sign. Thus, both  $d_2^{0,3}$ and $d_2^{1,1}$ are  non-zero. 
\end{remark}

\appendix

\section{Derived de Rham cohomology and derived divided power de Rham complex}
In this section, we review derived de Rham cohomology and the derived divided power de Rham complex to make the exposition self-contained.

\subsection{Derived de Rham cohomology}
\label{app:relative-kan-extension}
We first recall the construction of derived de Rham cohomology; see \cite[Chapitre~VIII, \S2]{Ill72a} and \cite[Definition~2.1]{Bha12}. Following \cite[Definition~2.1.1]{FKM26}, we organize the definitions below as follows. For an $\infty$-category $\mathsf{C}$, we write $\Ar(\mathsf{C})$ for $\Fun([1],\mathsf{C})$. When $\mathsf{C}=\CAlg^\an$, the $\infty$-category of animated rings, the arrow category $\Ar(\CAlg^\an)$ admits a family of compact projective generators
 \[\mathbb{Z}[x_1,\ldots,x_l] \to \mathbb{Z}[x_1,\ldots,x_m], \qquad 0\leq l\leq m,\]
 where the map sends \(x_i\) to \(x_i\) for \(1\leq i\leq l\); see \cite[Corollary~2.13]{Mao24}. We denote the full subcategory spanned by these objects by $\Ar(\CAlg^\an)_{\mathrm{poly}}$. Let $\CAlg^{\an}\CAlg^{\mathsf{fil}}$ be the unstraightening of the functor \[
 \CAlg^\an
 \longrightarrow
 \Cat,
 \qquad
 R
 \longmapsto
 \CAlg\bigl(\mathsf{Fil}(\mathsf{D}(R))\bigr).
 \]

\begin{construction}[Derived de Rham cohomology]\label{de Rham}
We define the Hodge-filtered derived de Rham cohomology functor to be the left Kan extension along the inclusion  $\Ar(\CAlg^\an)_{\mathrm{poly}}\hookrightarrow\Ar(\CAlg^\an)$ of the classical Hodge-filtered de Rham cohomology  functor $\mathrm{Fil}_{\mathrm{Hod}}^{*} \Omega_{-/-}^\bullet\colon
\Ar(\CAlg^\an)_{\mathrm{poly}}
\longrightarrow
\CAlg^{\an}\CAlg^{\mathsf{fil}}$. 

   \[
\begin{aligned}
\mathrm{Fil}_{\mathrm{Hod}}^{*} \dR\colon
\Ar\bigl(\CAlg^{\an}\bigr)
&\longrightarrow
\CAlg^{\an}\CAlg^{\mathsf{fil}} \\
(A\to B)
&\longmapsto
\mathrm{Fil}_{\mathrm{Hod}}^{*}\mathrm{dR}_{B/A}.
\end{aligned}
\]

\end{construction}

We list some basic properties of derived de Rham cohomology.
\begin{proposition}\label{syn}
	Let \(A \to R\) be a morphism of animated rings.
	
	\begin{enumerate}
		\item[(1)] \textup{(Base change)}
		For any morphism of animated rings \(A \to A'\), the natural map
		\[
		\mathrm{Fil}_{\mathrm{Hod}}^{*}\dR_{R/A}
		\otimes_A A'
		\longrightarrow
		\Fil_{\mathrm{Hod}}^{ *}
		\dR_{R\otimes_A A'/A'}
		\]
		is an equivalence.	
	\item[(2)]\textup{(Graded pieces)} There is a natural equivalence in $\CAlg\bigl(\Gr(\mathsf{D}(A))\bigr)$ \[\gr^*_{\mathrm{Hod}}\dR_{R/A}\simeq \bigoplus_i\Lambda^i(L_{R/A})[-i].\]
	 Here the direct sum is understood in the graded sense. 
	\item[(3)] \textup{(Descent in characteristic $p$)}
	Assume that $A$ is an $\mathbb F_p$-algebra and that
	$A\to R$ is quasisyntomic.
	\begin{enumerate}
		\item[(a)] \textup{(Quasisyntomic descent in \(R\))}
		If \(R \to R^0\) is quasisyntomic and faithfully flat, let $R^i:= (R^0)^{\otimes_R i+1}$. Then
		\[
		\mathrm{Fil}_{\mathrm{Hod}}^{*}\dR_{R/A}
		\longrightarrow
		\lim_{[i]\in \Delta}\mathrm{Fil}_{\mathrm{Hod}}^{ *}\dR_{R^i/A}
		\]
		is an equivalence.
		
		\item[(b)] \textup{(Flat descent in $A$)}
		If \(A \to A^0\) is faithfully flat, let $A^i:= (A^0)^{\otimes_A i+1}$ and $R^i:= R \otimes_A A^i$. Then
		\[
		\mathrm{Fil}_{\mathrm{Hod}}^{ *}\dR_{R/A}
		\longrightarrow
		\lim_{[i]\in \Delta} \mathrm{Fil}_{\mathrm{Hod}}^{ *}\dR_{R^i/A^i}
		\]
		is an equivalence.
	\end{enumerate}
\end{enumerate}
\end{proposition}

We now want to extend the construction to stacks. We imitate the definition of cotangent complexes in \cite[Definition~3.27]{Tan24}; see also \cite[Definition~3.27]{ABM21} for the definition over a fixed affine base.  
\begin{construction}\label{relative}

For a syntomic morphism $X\to Y$ between Artin stacks in the sense of  \cite[Construction~A.1.12]{KP24}, we can define the Hodge-filtered derived de Rham cohomology of $X\to Y$ as follows:
When $Y=\operatorname{Spec}(A)$ is an affine scheme, we define 
\[\mathrm{Fil}^*_{\mathrm{Hod}}\dR_{X/\mathrm{Spec}(A)}=\mathrm{Fil}^*_{\mathrm{Hod}} \dR_{X/A}:= \lim_{\substack{\mathrm{Spec} B\to X\\ B/A\ \mathrm{syntomic}}} \mathrm{Fil}^*_{\mathrm{Hod}}\dR_{B/A} \in \CAlg(\Fil(\mathsf{D}(A)))\]

For general $Y$, we define 
\[\mathrm{Fil}^*_{\mathrm{Hod}}\dR_{X/Y}:=\lim_{u:\mathrm{Spec}(A)\to Y} u_*\mathrm{Fil}^*_{\mathrm{Hod}}\dR_{X_A/A}  \in \CAlg(\Fil(\mathsf{D}(Y)))\] 

We denote by $\dR_{X/Y}$ the underlying object of
$\mathrm{Fil}^*_{\mathrm{Hod}}\dR_{X/Y}$ and refer to it as the de Rham cohomology of $X\to Y$.

We now explain the construction more conceptually. We define $\CAlg^{\mathrm{fil}}\to \PStk^\op$  as the unstraightening of the functor
\[
\begin{aligned}
\PStk^{\mathsf{op}}
&\longrightarrow
\mathsf{Pr}^{\mathrm{L}}\hookrightarrow \Cat \\
X
&\longmapsto
\CAlg(\mathsf{Fil}(\mathsf{D}(X))).
\end{aligned}
\]

We know that $\CAlg^{\mathsf{fil}}\times_{\PStk^{\op}} \CRing\to \CRing$ can be identified with $\CAlg^{\an}\CAlg^{\mathsf{fil}}\times_{\CAlg^\an} \CRing \to \CRing$ as cocartesian fibration over $\CRing$. Thus, we can restrict $\mathrm{Fil}_{\mathrm{Hod}}^{*}\dR$ to a functor from $\Ar(\CRing)$ to $\CAlg^{\mathsf{fil}}$. Let $\Ar(\CRing)^\syn$ be the full subcategory of $\Ar(\CRing)$ spanned  by syntomic maps of commutative rings. Similarly, we can define $\Ar(\Aff)^{\syn}$ , $\Ar(\ArtStk)^\syn$ in the evident sense. Consider the following commutative diagram in $\Cat$,

\[
\begin{tikzcd}
\Ar(\CRing)^\syn
=(\Ar(\Aff)^{\syn})^\op
  \arrow[r]
  \arrow[d]
&
\CAlg^{\mathsf{fil}}
  \arrow[d]
\\
(\Ar(\ArtStk)^\syn)^{\op}
  \arrow[r]
&
\PStk^{\op}.
\end{tikzcd}
\]

We define  $
\mathrm{Fil}_{\mathrm{Hod}}^{*}\dR: (\Ar(\ArtStk)^\syn)^{\op}
\longrightarrow
\CAlg^{\mathsf{fil}}
$ as the relative right Kan extension along the inclusion $(\Ar(\Aff)^{\syn})^\op \hookrightarrow (\Ar(\ArtStk)^\syn)^{\op}$ over $\PStk^{\op}$ of  $
\mathrm{Fil}_{\mathrm{Hod}}^{*}\dR: \Ar(\CRing)^\syn
\longrightarrow
\CAlg^{\mathsf{fil}}
$  in the sense of \cite[Tag~02Z2]{Lur26}. For our purposes, we will only use this definition in characteristic \(p\).
\end{construction}

\begin{remark}[Descent properties of $\dR$]\label{lim}
	\leavevmode\par
	
	\begin{enumerate}
		\item[(1)]\textup{(Descent in the source)}
	Let $R$ be an $\mathbb F_p$-algebra. Write $\mathrm{Syn}_R$ for the category of syntomic $R$-algebras. By Proposition~\ref{syn}, the functor
	\[
	\mathrm{Fil}^*_{\mathrm{Hod}}\dR_{-/R}
	\colon
	\mathrm{Syn}_R
	\longrightarrow
	\CAlg\bigl(\mathsf{Fil}(\mathsf{D}(R))\bigr)
	\]
	satisfies syntomic descent. Consequently, for a syntomic cover $X_0\to X$ of syntomic Artin stacks over $R$, we have a natural equivalence $\displaystyle \mathrm{Fil}^*_{\mathrm{Hod}}\dR_{X/R}\simeq \lim_{[i]\in \Delta} \mathrm{Fil}^*_{\mathrm{Hod}}\dR_{X_0^{\times_X (i+1)}/R}$. The same statement holds in the relative setting. More precisely, for a syntomic cover $X_0\to X$ of syntomic Artin stacks over $Y$, we also have $\displaystyle \mathrm{Fil}^*_{\mathrm{Hod}}\dR_{X/Y}\simeq \lim_{[i]\in \Delta} \mathrm{Fil}^*_{\mathrm{Hod}}\dR_{X_0^{\times_X (i+1)}/Y}$ as $\mathrm{Fil}^*_{\mathrm{Hod}}\dR_{-/Y}$ can be written as a limit of the functors $u_*\mathrm{Fil}^*_{\mathrm{Hod}}\dR_{-\times_Y \mathrm{Spec}(A)/A}$.
	\item[(2)]\textup{(Descent in the base)}
	Let  $X\to Y$ be syntomic morphism of Artin stacks over
	$\mathbb F_p$. Write $\CRing_{/Y}:=\CRing\times_{\ArtStk^{\op}}(\ArtStk_{/Y})^{\op}$. The functor \[F\colon \CRing_{/Y}\longrightarrow  \CAlg\bigl(\mathsf{Fil}(\mathsf{D}(Y))\bigr), \qquad (u:\mathrm{Spec}(A)\to Y)\longmapsto u_*\mathrm{Fil}^*_{\mathrm{Hod}}\dR_{X\times_Y \mathrm{Spec}(A)/A}\] satisfies flat descent. Indeed, for every morphism $\mathrm{Spec}(A)\to Y$,
	the $A$-stack $X\times_Y\mathrm{Spec}(A)$ lies in the
	colimit closure of syntomic affine schemes over $A$.
	The assertion therefore follows from Proposition~\ref{syn}(3)(b).
	Note that the functor \[G\colon (\ArtStk_{/Y})^{\op} \longrightarrow  \CAlg\bigl(\mathsf{Fil}(\mathsf{D}(Y))\bigr), \qquad (u:Z\to Y)\longmapsto u_*\mathrm{Fil}^*_{\mathrm{Hod}}\dR_{X\times_Y Z/Z}\] is the right Kan extension of $F$ along the inclusion $\CRing_{/Y}\hookrightarrow (\ArtStk_{/Y})^{\op}$. Thus, for an fppf cover $Y_0\to Y$, we have an equivalence $\displaystyle \mathrm{Fil}^*_{\mathrm{Hod}}\dR_{X/Y}\simeq \lim_{[i]\in \Delta} (f_i)_*\mathrm{Fil}^*_{\mathrm{Hod}}\dR_{X\times_Y Y_i/Y_i}$, where $Y_i:=Y_0^{\times_Y (i+1)}$ and  $f_i:Y_i\to Y$ is the natural map in the \v{C}ech nerve of $Y_0\to Y$.
\end{enumerate}
\end{remark}

\begin{remark}
	The functor defined by the following assignment
	\[
	(\Aff_{/Y})^{\op}\longrightarrow \CAlg^{\mathsf{fil}},
	\qquad
	(\mathrm{Spec}(A)\to Y)\longmapsto
	\mathrm{Fil}^*_{\mathrm{Hod}}\dR_{X_A/A}
	\]
	determines an object of
	$
	\Fun_{(\Aff_{/Y})^{\op}}
	\left(
	(\Aff_{/Y})^{\op},
	\CAlg^{\mathsf{fil}}\times_{\PStk^\op}(\Aff_{/Y})^{\op}
	\right).
	$
	Suppose that this object is a cocartesian section, i.e. it lies in
	$
	\Fun_{(\Aff_{/Y})^{\op}}^{\mathrm{cocart}}
	\left(
	(\Aff_{/Y})^{\op},
	\CAlg^{\mathsf{fil}}\times_{\PStk^\op}(\Aff_{/Y})^{\op}
	\right).
	$
	Then this cocartesian section corresponds to $\mathrm{Fil}^*_{\mathrm{Hod}}\dR_{X/Y}$ under the equivalences
	\[
	\CAlg\bigl(\mathsf{Fil}(\mathsf{D}(Y))\bigr)
	\simeq
	\lim_{\mathrm{Spec}(A)\to Y}
	\CAlg\bigl(\mathsf{Fil}(\mathsf{D}(A))\bigr)
	\simeq
	\Fun_{(\Aff_{/Y})^{\op}}^{\mathrm{cocart}}
	\left(
	(\Aff_{/Y})^{\op},
	\CAlg^{\mathsf{fil}}\times_{\PStk^\op}(\Aff_{/Y})^{\op}
	\right);
	\]
	see the following lemma. In this case, for every morphism $f\colon Y'\to Y$, there is a natural equivalence $f^*\mathrm{Fil}^*_{\mathrm{Hod}}\dR_{X/Y}\simeq \mathrm{Fil}^*_{\mathrm{Hod}}\dR_{X\times_Y Y'/Y'}$. For example, when $X\to Y$ is an affine syntomic morphism of Artin stacks, the above cocartesian condition is satisfied.
\end{remark}
The proof of the following lemma essentially comes from the argument in \cite[Lemma~3.16.]{BZFN10}. Let $F_0:\CRing \to \mathsf{Pr}^{\mathrm{L}}$ be an accessible functor satisfying fppf descent. Thus, the right Kan extension $F: \Stk^\op\to \mathsf{Pr}^{\mathrm{L}}$  of $F_0$ along the inclusion  $\CRing\hookrightarrow \Stk^\op$ commutes with limits. We denote by $\mathsf C\to\Stk^\op$ the unstraightening of the functor $F: \Stk^\op\to \mathsf{Pr}^{\mathrm{L}}$.
\begin{lemma}\label{BZFN}
	Let $(X_{i})_{i\in I}$ be a functor from $I$ to $\Stk$ with colimit $X$, where $I$ is a small $\infty$-category, or more generally an $\infty$-category admitting a small cofinal full subcategory. Let $(M_i)_{i\in I^{\op}}$ be a cocartesian section of $\mathsf{C}\times_{\Stk^\op}I^{\op} \to I^{\op}$, i.e. an object in $\Fun^{\mathrm{cocart}}_{I^{\op}}(I^{\op}, \mathsf{C}\times_{\Stk^\op}I^{\op})$. Under the equivalence $\displaystyle F(X)\simeq \lim_{i\in I^{\op}}F(X_i)\simeq \Fun^{\mathrm{cocart}}_{I^{\op}}(I^{\op}, \mathsf{C}\times_{\Stk^\op}I^{\op})$ , the cocartesian section $(M_i)_{i\in I^{\op}}$ defines an object  $M\in F(X)$. Then there is a canonical equivalence
	
	\[\lim_{i\in I^{\op}} (f_i)_*M_i\simeq M,\]
	where $f_i:X_i\to X$ is the canonical morphism induced by the colimit diagram.
	
\end{lemma} 

\begin{proof}Since $\displaystyle F(X)\simeq \lim_{i\in I^{\op}}F(X_i)$, there is an equivalence
$\displaystyle M\simeq \lim_{i\in I^{\op}}(f_i)_*(f_i)^*M\simeq \lim_{i\in I^{\op}}(f_i)_*M_i$.
\end{proof}

We next study the graded pieces of Hodge-filtered derived de Rham cohomology. For a commutative ring $R$, we denote by $\ArtStk_{R}^{\mathrm{qcqs,syn}}$ the full subcategory of $\ArtStk_{R}$ spanned by quasi-compact quasi-separated syntomic Artin stacks over $R$.
\begin{lemma}\label{Kan}
	 For every $m\geq 0$, the Hodge cohomology functor
	\[ R\Gamma(-,\Lambda^mL_{-/R})\colon \ArtStk_{R}^{\mathrm{qcqs,syn},\op}\longrightarrow \mathsf{D}(R) \qquad X\longmapsto R\Gamma(X,\Lambda^mL_{X/R}) \] is right Kan extended from $\mathrm{Syn}_R$.
\end{lemma}

\begin{proof}
	Let $F$ be the right Kan extension of the functor $\Lambda^mL_{-/R}\colon\mathrm{Syn}_R\longrightarrow \mathsf{D}(R) ,~A\longmapsto \Lambda^mL_{A/R} $ along the inclusion $\mathrm{Syn}_R \hookrightarrow \ArtStk_{R}^{\mathrm{qcqs,syn},\op}$. By the universal property of right Kan extension, we obtain a natural transformation $R\Gamma(-,\Lambda^mL_{-/R})\to F$, whose restriction to $\mathrm{Syn}_R$ is the identity. Hence, we only need to show this natural transformation is a natural equivalence. 
	
	We know that $\Lambda^mL_{-/R}$ satisfies flat descent; see \cite[Theorem~3.1.]{BMS19}. It follows that, for a fppf cover $Y_0\to Y$ in $\ArtStk_{R}^{\mathrm{qcqs,syn}}$, we have an equivalence $\displaystyle F(Y)\simeq  \lim_{[i]\in \Delta}F(Y_i)$, where $\displaystyle Y_i=Y_0^{\times_Y (i+1)}$.  Let $X$ be an $n$-Artin stack over $R$. We will prove by induction on \(n\) that the natural map $R\Gamma(X,\Lambda^mL_{X/R})\to F(X)$ is an isomorphism. When $X$ is an affine scheme i.e. a $(-1)$-Artin stack, the claim follows from definition of $F$. Suppose now that the claim holds for all $(k-1)$-Artin stacks, and let $X$ be a $k$-Artin stack, where $k\geq 0$. Then choose a smooth atlas $X_0=\mathrm{Spec}(A) \to X$. We have $\displaystyle F(X)\simeq\lim_{[i]\in \Delta}F(X_i)$, where $X_i:=X_0^{\times_X (i+1)}$ is a $(k-1)$-Artin stack for $i\geq 0$. By the induction hypothesis, we see that $R\Gamma(X_i,\Lambda^mL_{X_i/R})\simeq F(X_i)$ for $i\geq 0$.  Thus, it remains to show that $\displaystyle R\Gamma(X,\Lambda^mL_{X/R})\simeq \lim_{[i]\in \Delta}R\Gamma(X_i,\Lambda^mL_{X_i/R})$. Let $p_i:X_i\to X$ denote the augmentation map
	in the \v{C}ech nerve of $X_0\to X$. By flat descent for cotangent complexes of stacks in \cite[Proposition~1.1.5]{KP22}, we obtain an equivalence
	\[
	\Lambda^m L_{X/R}
	\simeq
	\lim_{[i]\in \Delta}
	p_{i*}\bigl(\Lambda^m L_{X_i/R}\bigr).
	\]
	Applying the functor $R\Gamma(X,-)$, we obtain an equivalence
	$R\Gamma(X,\Lambda^mL_{X/R})\simeq \lim_{[i]\in \Delta}R\Gamma(X_i,\Lambda^mL_{X_i/R})$.
\end{proof}	

\begin{proposition}\label{cotangent}
For a syntomic quasi-compact quasi-separated morphism $f:X\to Y $ of Artin stacks and every \(n\geq0\), there is a natural equivalence
    \[\gr^n(\dR_{X/Y})\simeq f_*( \Lambda^nL_{X/Y})[-n].\] Moreover, these equivalences assemble to a natural equivalence in $\CAlg(\Gr(\mathsf{D}(Y)))$
    \[\gr^*(\dR_{X/Y})\simeq \bigoplus_{i\in \mathbb N}f_*( \Lambda^iL_{X/Y})[-i].\]
 Here the direct sum is understood in the graded sense.   
    
\end{proposition}

\begin{proof} Suppose first that $Y=\mathrm{Spec}(A)$. We have 
	\[
	\begin{aligned}
		\gr^n(\dR_{X/A})
		\simeq
		\lim_{\substack{\mathrm{Spec} B\to X\\ B/A\ \mathrm{syntomic}}}
		\gr^n(\dR_{B/A}) 
		&\simeq
		\lim_{\substack{\mathrm{Spec} B\to X\\ B/A\ \mathrm{syntomic}}}
		 \Lambda^n L_{B/A}[-n] \\
		&\simeq
	 R\Gamma\bigl(X,\Lambda^nL_{X/A}\bigr)[-n].
	\end{aligned}
	\]
The last equivalence follows from Lemma~\ref{Kan}.
For a general Artin stack \(Y\), let \(u:\operatorname{Spec}(A)\to Y\) be a morphism, set \(X_A:=X\times_Y\operatorname{Spec}(A)\). Let \(g_A:X_A\to X\) and $f_A:X_A\to\mathrm{Spec}(A)$ be the natural projections. We identify $R\Gamma(X_A,-)$ with $f_{A,*}$. Then we have 
\[
\begin{aligned}
	\gr^n(\dR_{X/Y})
	&\simeq
	\lim_{u:\mathrm{Spec}(A)\to Y}
	u_*\gr^n(\dR_{X_A/A})\\
	&\simeq
	\lim_{u:\mathrm{Spec}(A)\to Y}
	u_*f_{A,*}\bigl(\Lambda^nL_{X_A/A}\bigr)[-n]\\
	&\simeq
	\lim_{u:\mathrm{Spec}(A)\to Y}
	f_*g_{A,*}g_A^*\bigl(\Lambda^nL_{X/Y}\bigr)[-n]\\
	&\simeq
	f_*\left(
	\lim_{u:\mathrm{Spec}(A)\to Y}
	g_{A,*}g_A^*\bigl(\Lambda^nL_{X/Y}\bigr)
	\right)[-n]\\
	&\simeq
	f_*\bigl(\Lambda^nL_{X/Y}\bigr)[-n].
\end{aligned}
\]
Here the third equivalence follows from $u_*f_{A,*}\simeq f_* g_{A,*}$ and the base change equivalence for cotangent complexes $g_A^*\bigl(\Lambda^nL_{X/Y}\bigr)\simeq \Lambda^nL_{X_A/A}$. The second assertion follows from a similar argument.
\end{proof}

We next study the base change property of derived de Rham cohomology. We will use  the conjugate filtration in these proofs.

	\begin{construction}[Conjugate filtration on derived de Rham cohomology]
		Let $A$ be an $\mathbb F_p$-algebra and let $B$ be an animated $A$-algebra. The derived de Rham cohomology $\dR_{B/A}$ carries a natural increasing exhaustive $\mathbb{N}$-indexed filtration $\mathrm{Fil}^{\mathrm{conj}}_*\dR_{B/A}$, called the conjugate filtration; see \cite[Proposition~2.3]{Bha12} and \cite[Construction~E.7]{BL22}. Moreover, the derived Cartier isomorphism
		gives $\gr_n^{\mathrm{conj}}\dR_{B/A}\simeq \Lambda^nL_{B^{(1)}/A}[-n]$, where $B^{(1)}:=B\otimes_{A,\mathrm{Frob}_A} A$ 	; see \cite[Proposition~3.5]{Bha12} and \cite[Remark~E.8]{BL22}. We next extend the definition to stacks by an approach similar to that of Construction~\ref{relative}.
		For a syntomic morphism $X\to\mathrm{Spec} A$, we define
		\[
		\mathrm{Fil}_n^{\mathrm{conj}}\dR_{X/A}
		:=
		\lim_{\substack{\mathrm{Spec} B\to X\\ B/A\ \mathrm{syntomic}}}
		\mathrm{Fil}_n^{\mathrm{conj}}\dR_{B/A} \in \mathsf{D}(A),
		\]
		and for a general syntomic morphism $X\to Y$, we define
		\[
		\mathrm{Fil}_n^{\mathrm{conj}}\dR_{X/Y}
		:=
		\lim_{u:\mathrm{Spec} A\to Y}
		u_*\mathrm{Fil}_n^{\mathrm{conj}}\dR_{X_A/A}\in \mathsf{D}(Y).
		\]
Therefore, for a syntomic morphism $X\to Y$ of $\mathbb{F}_p$-Artin stacks, we obtain a natural increasing filtration $\mathrm{Fil}_*^{\mathrm{conj}}\dR_{X/Y}$ on	$\dR_{X/Y}$. Moreover, if $f$ is quasi-compact and quasi-separated, the derived Cartier isomorphism and the same argument as in the proof of
Proposition~\ref{cotangent} give a natural equivalence
$\gr_n^{\mathrm{conj}}\dR_{X/Y}
\simeq
f^{(1)}_*
\bigl(\Lambda^nL_{X^{(1)}/Y}\bigr)[-n]
$, where
$X^{(1)}:=X\times_{Y,\mathrm{Frob}_Y}Y$
and $f^{(1)}:X^{(1)}\to Y$ is the natural projection.	
	\end{construction}
We first establish the exhaustiveness of the conjugate filtration on derived de Rham cohomology over an affine base.
\begin{lemma}\label{affe}
	Let  $X$ be a quasi-compact quasi-separated syntomic Artin stack over an $\mathbb F_p$-algebra $R$. Then the conjugate filtration on $\dR_{X/R}$ is exhaustive, and both $\dR_{X/R}$ and
	$\mathrm{Fil}_n^{\mathrm{conj}}\dR_{X/R}$
	belong to $\mathsf D(R)^{\geq0}$ for every $n\geq0$.

\end{lemma}	
\begin{proof}	
	 Suppose that $X$ is an $n$-Artin stack. We prove both assertions simultaneously by induction on $n$. When $X=\mathrm{Spec}(B)$ is an affine scheme, i.e. a $(-1)$-Artin stack, the conjugate filtration is always exhaustive. Since $X=\mathrm{Spec}(B)$ is syntomic over $R$, we have
	 $\gr_n^{\mathrm{conj}}\dR_{B/R}\simeq\Lambda^nL_{B^{(1)}/R}[-n]\in\mathsf D(R)^{\geq0}$
	 for every $n\geq0$. It follows that $\mathrm{Fil}_n^{\mathrm{conj}}\dR_{B/R}\in \mathsf{D}(R)^{\geq 0}$. Since the conjugate filtration is exhaustive, we have 
	 $\displaystyle \dR_{B/R}\simeq\colim_{n\in\mathbb N}\mathrm{Fil}_n^{\mathrm{conj}}\dR_{B/R}\in\mathsf D(R)^{\geq0}$. Suppose now that the claim holds for all $(k-1)$-Artin stacks, and let $X$ be a $k$-Artin stack, where $k\geq 0$. Choose a smooth atlas $X_0=\mathrm{Spec}(A) \to X$, and set $X_i:=X_0^{\times_X (i+1)}$. Then each $X_i$ is a quasi-compact quasi-separated syntomic $(k-1)$-Artin stack over $R$. By the induction hypothesis, the conjugate filtration on $\dR_{X_i/R}$ is exhaustive and $\dR_{X_i/R}\in\mathsf D(R)^{\geq0}$ for every $i\geq0$. Also note that the derived Cartier isomorphism and flat descent for the cotangent complex imply that the assignment
	 $X\longmapsto \mathrm{Fil}_n^{\mathrm{conj}}\dR_{X/R}$ satisfy sytomic descent for every $n\geq0$. Hence, we have
\[    \colim_{j\in \mathbb{N}} \mathrm{Fil}^{\mathrm{conj}}_j\dR_{X/R}  \simeq \colim_{j\in \mathbb{N}} \lim_{[i]\in \Delta}     \mathrm{Fil}^{\mathrm{conj}}_j\dR_{X_i/R} \simeq   \lim_{[i]\in \Delta} \colim_{j\in \mathbb{N}} \mathrm{Fil}^{\mathrm{conj}}_j\dR_{X_i/R} \simeq \lim_{[i]\in \Delta}\dR_{X_i/R}\simeq \dR_{X/R} ~,\]
where the second equivalence follows from the fact that cosimplicial limits commute with filtered colimits for coconnective objects in $\mathsf{D}(R)$; see, for example, \cite[Lemma~3.20]{Tan24}. Thus, the conjugate filtration $\mathrm{Fil}^{\mathrm{conj}}_*\dR_{X/R}$ is exhaustive. Since $\mathsf D(R)^{\geq0}$ is closed under limits,
the equivalences
$\dR_{X/R}\simeq\lim_{[i]\in\Delta}\dR_{X_i/R}$
and
$\mathrm{Fil}_n^{\mathrm{conj}}\dR_{X/R}
\simeq\lim_{[i]\in\Delta}
\mathrm{Fil}_n^{\mathrm{conj}}\dR_{X_i/R}$
show that both $\dR_{X/R}$ and
$\mathrm{Fil}_n^{\mathrm{conj}}\dR_{X/R}$
belong to $\mathsf D(R)^{\geq0}$ for every $n\geq0$. This completes the induction.
\end{proof}
We next prove the base change property for derived de Rham cohomology over an affine base.
\begin{proposition}\label{aff}
Let  $X$ be a quasi-compact quasi-separated syntomic Artin stack over an $\mathbb F_p$-algebra $R$, and let $R\to R'$ be a ring homomorphism of finite $\mathrm{Tor}$-amplitude. Writing $X_{R'}:=X\times_RR'$, the base change map
\[  \mathrm{Fil}^*_{\mathrm{Hod}}\dR_{X/R}\otimes_R R'\to  \mathrm{Fil}^*_{\mathrm{Hod}}\dR_{X_{R'}/R'} \]
is an equivalence.
\end{proposition}
\begin{proof}
	
		The induced map on the $n$-th graded pieces is the base change map 
		\[
		R\Gamma(X,\Lambda^nL_{X/R}[-n])\otimes_R R'
		\longrightarrow
		R\Gamma(X_{R'},\Lambda^nL_{X_{R'}/R'}[-n]).
		\] By \cite[Lemma~A.1.3]{HLP23}, we see that this map is an equivalence. It remains to show that $\dR_{X/R}\otimes_R R'\simeq \dR_{X_{R'}/R'}$.  For this, we consider the conjugate filtration. Indeed, the induced map on the $n$-th graded pieces is the base change map  $ R\Gamma(X^{(1)},\Lambda^nL_{X^{(1)}/R}[-n])\otimes_R R'\longrightarrow R\Gamma((X_{R'})^{(1)},\Lambda^n L_{(X_{R'})^{(1)}/R'}[-n])$. By \cite[Lemma~A.1.3]{HLP23}, this map is also an equivalence. Therefore, we have  $\mathrm{Fil}_n^{\mathrm{conj}}\dR_{X/R}\otimes_R R'\simeq \mathrm{Fil}_n^{\mathrm{conj}}\dR_{X_{R'}/R'}$.Since the conjugate filtrations on
		$\dR_{X/R}\otimes_R R'$ and $\dR_{X_{R'}/R'}$
		are exhaustive by Lemma~\ref{affe}, we conclude that
		$
		\dR_{X/R}\otimes_R R'
		\simeq
		\dR_{X_{R'}/R'}
		$.
	
\end{proof}

We finally extend the above results to the general case.
\begin{proposition}\label{basechange}
Let $f:X\to Y $ be a syntomic morphism between quasi-compact quasi-separated Artin stacks over $\mathbb{F}_p$. Then the conjugate filtration on $\dR_{X/Y}$ is exhaustive and $\dR_{X/Y}\in \mathsf{D}(Y)^{\geq 0}$. For a morphism $h: Z\to Y$ with finite $\mathrm{Tor}$-amplitude between Artin stacks over  $\mathbb{F}_p$, the base change map
\[h^*(\mathrm{Fil}^*_{\mathrm{Hod}}\dR_{X/Y})\to \mathrm{Fil}^*_{\mathrm{Hod}}\dR_{X\times_Y Z/Z}\]
is an equivalence. 
\end{proposition}
\begin{proof}We prove the following assertions simultaneously for every $n$-Artin stack $Y$ by induction on $n$:
	\begin{enumerate}
		\item The conjugate filtration on $\dR_{X/Y}$ is exhaustive,
		and both $\dR_{X/Y}$ and
		$\mathrm{Fil}_n^{\mathrm{conj}}\dR_{X/Y}$
		belong to $\mathsf D(Y)^{\geq0}$ for every $n\geq0$.
		\item For every finite Tor-amplitude morphism $h:Z\to Y$ with $Z$ an $n$-Artin stack, the base change map $h^*(\dR_{X/Y})\to \dR_{X\times_Y Z/Z}$ is an equivalence.
	\end{enumerate}
	
	When $Y$ is an affine scheme, i.e. a $(-1)$-Artin stack, these assertions follow from Lemma~\ref{affe} and Proposition~\ref{aff}. Suppose now that the claim holds for all $(k-1)$-Artin stacks, and let $Y$ be a $k$-Artin stack, where $k\geq 0$. Choose a smooth atlas $Y_0=\mathrm{Spec}(A) \to Y$, and set $Y_i:=Y_0^{\times_Y (i+1)}$ with the augmented map $g_i:Y_i\to Y$. Then each $Y_i$ is a quasi-compact quasi-separated $(k-1)$-Artin stack over $\mathbb F_p$.
	By the induction hypothesis, we have
	$\dR_{X\times_Y Y_i/Y_i}\in\mathsf D(Y_i)^{\geq0}$. By descent of derived de Rham cohomology in base in Remark~\ref{lim}(2), we have $\displaystyle \dR_{X/Y}\simeq \lim_{[i]\in \Delta}(g_i)_*\dR_{X\times_Y Y_i/Y_i}$, and $\displaystyle \mathrm{Fil}_*^{\mathrm{conj}}\dR_{X/Y}\simeq \lim_{[i]\in \Delta}(g_i)_*\mathrm{Fil}_*^{\mathrm{conj}}\dR_{X\times_Y Y_i/Y_i}$.
	Since each $(g_i)_*$ preserves $\mathsf D^{\geq0}$ and
	$\mathsf D(Y)^{\geq0}$ is closed under limits, these two equivalences above imply that
	$\dR_{X/Y}\in\mathsf D(Y)^{\geq0}$ and $\mathrm{Fil}_n^{\mathrm{conj}}\dR_{X/Y}\in\mathsf D(Y)^{\geq0}$ for every $n\geq 0$.
	
	 We next show exhaustiveness of the conjugate filtration on $\dR_{X/Y}$.
	Since each $Y_i$ is a quasi-compact quasi-separated $(k-1)$-Artin stack over $\mathbb F_p$ and the morphisms occurring in the \v{C}ech nerve $Y_0\to Y$ are of finite Tor-amplitude, it follows from the base change assertion in the induction hypothesis that the systems
	$\{\dR_{X\times_Y Y_i/Y_i}\}_{[i]\in\Delta}$ and
	$\{\mathrm{Fil}_n^{\mathrm{conj}}\dR_{X\times_Y Y_i/Y_i}\}_{[i]\in\Delta}$, for every $n\geq0$, are cocartesian sections of $\mathsf{D}\times_{\PStk^\op}\Delta\to \Delta$. Here $\mathsf{D}\to \PStk^\op$ is the unstraightening of the functor $X\mapsto \mathsf{D}(X)$ and $\Delta\to \PStk^\op$ corresponds to the \v{C}ech nerve of $Y_0\to Y$. Hence, by Lemma~\ref{BZFN}, we have $g_i^*(\dR_{X/Y}) \simeq \dR_{X\times_Y Y_i/Y_i}$ and
	$g_i^*(\mathrm{Fil}_n^{\mathrm{conj}}\dR_{X/Y}) \simeq \mathrm{Fil}_n^{\mathrm{conj}}\dR_{X\times_Y Y_i/Y_i}$ for every $n\geq 0$. We note that the functors
	$g_i^*:\mathsf{D}(Y)\to \mathsf{D}(Y_i)$ are jointly conservative and commute with colimits. Thus, it suffices to check exhaustiveness after pulling back to each $Y_i$, which follows from the induction hypothesis. 
	
	Let $h:Z\to Y$ be a finite Tor-amplitude morphism with $Z$ a $k$-Artin stack. By an argument using base change for cotangent complexes,
	as in the proof of Proposition~\ref{aff}, we see that the base change maps $h^*(\mathrm{Fil}_m^{\mathrm{conj}}\dR_{X/Y})\to \mathrm{Fil}_m^{\mathrm{conj}}\dR_{X\times_Y Z/Z}$ are equivalences for every $m\geq0$. Hence, exhaustiveness of the conjugate filtrations on $\dR_{X/Y}$ and $\dR_{X\times_Y Z/Z}$ implies the base change map $h^*(\dR_{X/Y})\to \dR_{X\times_Y Z/Z}$ is an equivalence. This completes the induction. Finally, by an argument similar to that in the proof of Proposition~\ref{aff}, together with the equivalence $h^*(\dR_{X/Y})\simeq \dR_{X\times_Y Z/Z}$, we see that the base change map $h^*(\mathrm{Fil}^*_{\mathrm{Hod}}\dR_{X/Y})\to \mathrm{Fil}^*_{\mathrm{Hod}}\dR_{X\times_Y Z/Z}$ is an equivalence.
\end{proof}

\begin{example}\label{action}
Let $R$ be a $\mathbb{F}_p$-algebra, and let $G$ be an affine flat finitely presented group scheme over $R$. For a quasi-compact quasi-separated syntomic Artin stack $X$ over $R$ equipped with a left $G$-action, we can consider the de Rham cohomology of $[X/G]$ over $BG$, i.e. $\dR_{[X/G]/ BG}$. Let $e:\mathrm{Spec}(R)\to BG$ denote the base point. Proposition~\ref{basechange} gives an equivalence $e^*\mathrm{Fil}^*_{\mathrm{Hod}}\dR_{[X/G]/ BG}\simeq \mathrm{Fil}^*_{\mathrm{Hod}}\dR_{X/R}$. Thus \(\mathrm{Fil}^*_{\mathrm{Hod}}\dR_{X/R}\) is naturally equipped with a $G$-action.

\end{example}	
\subsection{Derived divided power de Rham complex}

\begin{definition}(Animated pairs) An animated pair, or simply a pair when no confusion can arise, is a map of animated
 rings $A\to B$ with $\pi_0(A)\to \pi_0(B)$ surjective. We define $\mathsf{Pair}$ to be the full subcategory of $\Ar(\CAlg^{\an})$ spanned by animated pairs. Similarly, we define $\mathsf{Pair}_R$ in the evident way for a commutative ring $R$.

\end{definition}

\begin{remark}(Compact projective generators of \(\mathsf{Pair}_R\))
The category $\mathsf{Pair}_R$ admits a family of compact projective generators of the form 
 \[
 R[x_1,\ldots,x_m]\twoheadrightarrow R[x_1,\ldots,x_l],
 \qquad 0\leq l\leq m,
 \]
 where the map sends \(x_i\) to \(x_i\) for \(1\leq i\leq l\), and sends \(x_i\) to \(0\) for \(l+1\leq i\leq m\). We denote the full subcategory spanned by these objects by $\mathsf{Pair}_{R,\mathsf{poly}}$.
\end{remark}

\begin{remark}\label{Pair1}
    Let  $\mathsf{Pair}_R^{\mathrm{cl}}$ be the full subcategory of $\mathsf{Pair}_R$ spanned by pairs \(A\to B\) such that $A,B$ are classical $R$-algebras. It is easy to see $\mathsf{Pair}_R^{\mathrm{cl}}$ is equivalent to the category of ring-ideal pairs $(A, I)$, that is, a classical $R$-algebra $A$ along with an ideal $I$ of $A$. We will not distinguish between these two notions. Accordingly, we will often write a pair as $A\to A/I$,  even in the animated setting.
\end{remark}

\begin{construction}(Derived divided power de Rham complex with Hodge filtration) 
	\label{derived-pd-derham}
	Let $R$ be a commutative ring. For  $A\to A/I\in \mathsf{Pair}_R^{\mathrm{cl}}$, we denote its divided power envelope by $(D_A(I),J,\gamma)$. Consider the divided power de Rham complex of $D_A(I)$ over $R$; compare \cite[Tag 07HQ]{Sta26}
	\begin{equation}
		\Omega^{\mathrm{pd},\bullet}_{D_A(I)/R}: = \left( D_A(I) \xrightarrow{d} D_A(I) \otimes_A \Omega^1_{A/R} \xrightarrow{d} D_A(I) \otimes_A \Omega^2_{A/R} \xrightarrow{d} \dots \right).
	\end{equation}
This complex is  equipped with a decreasing multiplicative filtration defined by:
	\begin{equation}
		\mathrm{Fil}_{\mathrm{PD}}^{ r} \Omega^{\mathrm{pd},\bullet}_{D_A(I)/R}: = \left( \mathrm{Fil}_{\mathrm{PD}}^{ r} D_A(I) \rightarrow \mathrm{Fil}_{\mathrm{PD}}^{ r-1} D_A(I) \otimes_A \Omega^1_{A/R} \rightarrow \mathrm{Fil}_{\mathrm{PD}}^{ r-2} D_A(I) \otimes_A \Omega^2_{A/R} \rightarrow \dots \right).
	\end{equation}
	And the decreasing multiplicative filtration $\mathrm{Fil}_{\mathrm{PD}}^{*} D_A(I)$ on $ D_A(I)$  is defined by
	$\mathrm{Fil}_{\mathrm{PD}}^{ n} D_A(I)= J^{[n]}$,
	where we use the convention $\mathrm{Fil}_{\mathrm{PD}}^{ r}=D_A(I)$ for $r\leq 0$. Hence, we obtain a functor
\begin{equation}
		\mathrm{Fil}_{\mathrm{PD}}^*
		\Omega^{\mathrm{pd},\bullet}_{D_{-}(-)/R}
		:
		\mathsf{Pair}_{R}^{\mathrm{cl}}
		\longrightarrow
		\CAlg\bigl(\Fil(\mathsf D(R))\bigr).
\end{equation}
We define Hodge-filtered derived divided power de Rham complex functor 
\[
\begin{aligned}
\mathsf{Pair}_R&\longrightarrow \CAlg\bigl(\Fil(\mathsf{D}(R))\bigr)\\
( A\to A/I)&\longmapsto \mathrm{Fil}_{\mathrm{PD}}^*\dR^{\mathrm{pd}}_{LD_A(I)/R}
\end{aligned}
\]
as the left Kan extension  of the functor $\mathrm{Fil}_{\mathrm{PD}}^*\Omega^{\mathrm{pd},\bullet}_{D_{-}(-)/R}:\mathsf{Pair}_{R,\mathsf{poly}} \to \CAlg\bigl(\Fil(\mathsf{D}(R))\bigr)$ along the inclusion $\mathsf{Pair}_{R,\mathsf{poly}}\hookrightarrow \mathsf{Pair}_R$.
 
\end{construction}

\begin{lemma}\label{lemma A.7}
    Let $R$ be a commutative ring. For $A\to A/I\in \mathsf{Pair}_{R,\mathsf{poly}}$, there is a natural map of filtered commutative differential graded \(R\)-algebras  \[ \mathrm{Fil}_{\mathrm{PD}}^*\Omega^{\mathrm{pd},\bullet}_{D_A(I)/R} \to \mathrm{Fil}_{\mathrm{Hod}}^* \Omega^{\bullet}_{(A/I) /R},\] which is a filtered quasi-isomorphism. 
\end{lemma} 
\begin{proof}
	For \(A\to A/I\in \mathsf{Pair}_{R,\mathsf{poly}}\) and each $i\geq 0$, the natural maps 
   $D_A(I)\otimes_A\Omega^i_{A/R}  \to A/I\otimes_A \Omega^i_{A/R} \to  \Omega^{i}_{(A/I) /R} $ assemble to a map of cochain complexes
    $ \Omega^{\mathrm{pd}, \bullet}_{D_A(I)/R} \to  \Omega^{\bullet}_{(A/I) /R}$. It is straightforward to check that this map is multiplicative and respects the filtrations. Hence, we obtain a natural map of filtered commutative differential graded \(R\)-algebras  $ \mathrm{Fil}_{\mathrm{PD}}^*\Omega^{\mathrm{pd},\bullet}_{D_A(I)/R} \to \mathrm{Fil}_{\mathrm{Hod}}^* \Omega^{\bullet}_{(A/I) /R}$. Taking the \(i\)-th graded pieces on both sides, we obtain 
    \[
    \begin{aligned}
    	\gr^i_{\mathrm{PD}}
    	\Omega^{\mathrm{pd},\bullet}_{D_A(I)/R}
    	&\simeq
    	\Bigl[
    	\Gamma^i_{A/I}(I/I^2)
    	\longrightarrow
    	\Omega^1_{A/R}\otimes_A \Gamma^{i-1}_{A/I}(I/I^2)
    	\longrightarrow
    	\cdots
    	\longrightarrow
    	\Omega^{i-1}_{A/R}\otimes_A (I/I^2)
    	\longrightarrow
    	\Omega^i_{A/R}\otimes_A A/I
    	\Bigr]
    	\\
    	&\simeq
    	\Bigl[
    	\Gamma^i_{A/I}(I/I^2)
    	\longrightarrow
    	\bigl(\Omega^1_{A/R}\otimes_A A/I\bigr)\otimes_{A/I}
    	\Gamma^{i-1}_{A/I}(I/I^2)
    	\longrightarrow
    	\cdots
    	\longrightarrow
    	\Omega^i_{A/R}\otimes_A A/I
    	\Bigr]
    	\\
    	&\longrightarrow
    	\Omega^i_{(A/I)/R}[-i]
    	\simeq
    	\gr^i_{\mathrm{Hod}}\Omega^\bullet_{(A/I)/R}.
    \end{aligned}
    \]
    The third map is a quasi-isomorphism by \cite[Chapitre~I, \S~4.3.2]{Ill71}, since $0\to I/I^2 \to \Omega^1_{A/R}\otimes_A A/I \to\Omega^1_{(A/I)/R}\to 0$ is a short exact sequence. It therefore remains to prove that the map on underlying cochain complexes $ \Omega^{\mathrm{pd},\bullet}_{D_A(I)/R} \to  \Omega^{\bullet}_{(A/I) /R}$ is a quasi-isomorphism. Both functors carry finite coproducts to tensor products, and the comparison map is compatible with finite coproducts. Since \(\mathsf{Pair}_{R,\mathrm{poly}}\) is generated under finite coproducts by the pairs \(R[x]\to R[x]\) and \(R[x]\to R\), it suffices to verify that the comparison maps are quasi-isomorphisms for these two pairs. For the pair
$R[x]\longrightarrow R[x]$, we have $\Omega^{\mathrm{pd}, \bullet}_{D_A(I)/R}= \Omega^{\bullet}_{R[x] /R}$, and the comparison map is the identity, hence a quasi-isomorphism. For the pair $R[x]\longrightarrow R$,  we have $D_A(I)=R\langle x\rangle$, and hence
$
\Omega^{\mathrm{pd},\bullet}_{D_A(I)/R}
=
\Omega^{\mathrm{pd},\bullet}_{R\langle x\rangle/R}
=
\left[
R\langle x\rangle
\xrightarrow{d}
R\langle x\rangle\,dx
\right]
$. The comparison map is 
$[R\langle x \rangle
\xrightarrow{d}
R\langle x \rangle\, dx]\to R$, which is given in degree \(0\) by the natural augmentation \(R\langle x\rangle\to R\) and in degree \(1\) by the zero map. Hence, it suffices to show that 
\[
0 \longrightarrow R
\longrightarrow R\langle x \rangle
\xrightarrow{d}
R\langle x \rangle\, dx
\longrightarrow 0
\]
is a short exact sequence. Indeed, the \(R\)-module \(R\langle x\rangle\) is free with basis $ \{x^{[q]}\}_{q\geq 0}$,
and $ d(x^{[q]})=x^{[q-1]}dx$ for $q\geq 1$. Thus $d$ is surjective and $\ker(d)=R$. This completes the proof.

\end{proof}
\begin{theorem}\label{P1}
	For \(A\to A/I\in \mathsf{Pair}_R\), there is a natural equivalence
	\[ \mathrm{Fil}_{\mathrm{PD}}^*\dR^{\mathrm{pd}}_{LD_A(I)/R} \simeq \mathrm{Fil}_{\mathrm{Hod}}^* \dR_{(A/I) /R}.\]
\end{theorem}

\begin{proof}
	By   Lemma~\ref{lemma A.7}, for every pair $ A\to A/I\in \mathsf{Pair}_{R,\mathsf{poly}}$, there is a natural equivalence $ \mathrm{Fil}_{\mathrm{PD}}^*\Omega^{\mathrm{pd},\bullet}_{D_A(I)/R} \simeq \mathrm{Fil}_{\mathrm{Hod}}^* \Omega^{\bullet}_{(A/I) /R}$ in $\CAlg\bigl(\Fil(\mathsf{D}(R))\bigr)$.  Taking the left Kan extension of both sides along the inclusion $\mathsf{Pair}_{R,\mathsf{poly}} \hookrightarrow
	\mathsf{Pair}_R$, we obtain the desired natural equivalence. 
\end{proof}

We next come to the case of characteristic $p$, and let $R$ be a (classical) \(\mathbb F_p\)-algebra.
\begin{definition}\label{Pair2}
    We define $\mathsf{Pair}_R^{\mathsf{sm,reg}}$ as the full subcategory of $\mathsf{Pair}^\mathrm{cl}_R$ 
spanned by $  A\to A/I$ such that $A$ is smooth over $R$ and $I$ is generated by a regular sequence in $A$.
\end{definition}
We record the following result, which essentially follows from the comparison between crystalline and derived de Rham cohomology  in \cite[Theorem~3.27]{Bha12}. See also \cite[Theorem~7.16]{Mat22} and \cite[Lemma~4.32 and Proposition~4.35]{Mao24}.
\begin{theorem}\label{P2}
	For $A\to A/I\in \mathsf{Pair}_R^{\mathsf{sm,reg}}$, there is a natural equivalence $\mathrm{Fil}_{\mathrm{PD}}^{*}\dR^{\mathrm{pd}}_{LD_A(I)/R}\simeq \mathrm{Fil}_{\mathrm{PD}}^{*}\Omega^{\mathrm{pd},\bullet}_{D_A(I)/R}$ in $\CAlg\bigl(\Fil(\mathsf{D}(R))\bigr)$.
    Consequently, there is a natural equivalence \[ \mathrm{Fil}^*_{\mathrm{Hod}}\dR_{(A/I)/R}\simeq \mathrm{Fil}_{\mathrm{PD}}^{*}\Omega^{\mathrm{pd},\bullet}_{D_A(I)/R}\] for every $A\to A/I\in\mathsf{Pair}_R^{\mathsf{sm,reg}}$.
\end{theorem}
For an $R$-algebra $A$, we write $\mathrm{Frob}_A\colon A\to A$ for the absolute Frobenius of $A$. Let
$
A^{(1)}:=A\otimes_{R,\mathrm{Frob}_R}R
$. We denote by
$
\iota_A\colon A\longrightarrow A^{(1)}
$
the natural map given by \(a\mapsto a\otimes1\), and by
$
\mathrm{Frob}_{A/R}\colon A^{(1)}\longrightarrow A
$
the relative Frobenius of \(A\) over \(R\).
\begin{proof}By Theorem~\ref{P1}, the second assertion follows from the first. It therefore remains to prove that, for $(A,I)\in \mathsf{Pair}_R^{\mathsf{sm,reg}}$, there is a natural equivalence $\mathrm{Fil}_{\mathrm{PD}}^{*}\dR^{\mathrm{pd}}_{LD_A(I)/R}\simeq \mathrm{Fil}_{\mathrm{PD}}^{*}\Omega^{\mathrm{pd},\bullet}_{D_A(I)/R}$ in $\CAlg\bigl(\Fil(\mathsf{D}(R))\bigr)$. Hence, we need to show   $(A,I)\mapsto \mathrm{Fil}_{\mathrm{PD}}^{*}\Omega^{\mathrm{pd},\bullet}_{D_A(I)/R}$ as a functor $\mathsf{Pair}_R^{\mathsf{sm,reg}}\to \Fil\bigl(\mathsf{D}(R)\bigr)$ is a left Kan extension from its restriction on $\mathsf{Pair}_{R,\mathsf{poly}}$. Note that there is a canonical equivalence  
	\[
	\begin{aligned}
		\gr^i_{\mathrm{PD}}
		\Omega^{\mathrm{pd},\bullet}_{D_A(I)/R}
		&\simeq
		\Bigl[
		\Gamma^i_{A/I}(I/I^2)
		\longrightarrow
		\Omega^1_{A/R}\otimes_A \Gamma^{i-1}_{A/I}(I/I^2)
		\longrightarrow
		\cdots
		\longrightarrow
		\Omega^{i-1}_{A/R}\otimes_A (I/I^2)
		\longrightarrow
		\Omega^i_{A/R}\otimes_A A/I
		\Bigr]
		\\
		&\simeq
		\Bigl[
		\Gamma^i_{A/I}(I/I^2)
		\longrightarrow
		\bigl(\Omega^1_{A/R}\otimes_A A/I\bigr)\otimes_{A/I}
		\Gamma^{i-1}_{A/I}(I/I^2)
		\longrightarrow
		\cdots
		\longrightarrow
		\Omega^i_{A/R}\otimes_A A/I
		\Bigr].
	\end{aligned}
	\]
	The complex on the right is a shift of the total degree $i$ part of the Koszul complex
	associated with the map
	$
	d\colon I/I^2\longrightarrow \Omega^1_{A/R}\otimes_A A/I
	$
	in the sense of \cite[\S I.4.3.1]{Ill71}.
By \cite[Corollary~1.2.7]{KS04}, we therefore obtain an equivalence	$\gr^i_{\mathrm{PD}}
\Omega^{\mathrm{pd},\bullet}_{D_A(I)/R}
\simeq \Lambda^i(L_{(A/I)/R})[-i]$. It is a left Kan extension from its restriction on $\mathsf{Pair}_{R,\mathsf{poly}}$. Thus, it suffices to show that $(A,I)\mapsto \mathrm{Fil}_{\mathrm{PD}}^{*}\Omega^{\mathrm{pd},\bullet}_{D_A(I)/R}$ as a functor $\mathsf{Pair}_R^{\mathsf{sm,reg}}\to \mathsf{D}(R)$ is a left Kan extension from its restriction on $\mathsf{Pair}_{R,\mathsf{poly}}$. In fact, there is a natural increasing, exhaustive, and multiplicative filtration on $D_A(I)$, denoted by \(\mathrm{Fil}^{\mathrm{conj}}_* D_A(I)\) and called the conjugate filtration; see for example, \cite[Lemma~3.42]{Bha12} and \cite[Definition~8.10]{BMS19}. Its graded pieces are given by $\gr^{\mathrm{conj}}_i D_A(I)=\mathrm{Frob}_A^*\bigl(
\Gamma^i_{A/I}(I/I^2)\bigr)$. Moreover, we can see  this filtration induces a filtration on $\Omega^{\mathrm{pd},\bullet}_{D_A(I)/R}$, see for example\cite[Lemma 3.44]{Bha12}, the complex
$\Omega^{\mathrm{pd},\bullet}_{D_A(I)/R}$ admits increasing, exhaustive, multiplicative filtration, namely conjugate filtration,
   $\mathrm{Fil}_*^{\mathrm{conj}}\Omega^{\mathrm{pd},\bullet}_{D_A(I)/R}$, which is defined by
   \[\mathrm{Fil}_n^{\mathrm{conj}}\Omega^{\mathrm{pd},\bullet}_{D_A(I)/R}:=\mathrm{Fil}_n^{\mathrm{conj}}D_A(I)\to \mathrm{Fil}_{n}^{\mathrm{conj}}D_A(I)\otimes_A \Omega^1_{A/R}\to \mathrm{Fil}_{n}^{\mathrm{conj}}D_A(I)\otimes_A \Omega^2_{A/R} \to \cdots .\]
   Then we can see that   \begin{equation}
   	\gr_i^{\mathrm{conj}}\Omega^{\mathrm{pd},\bullet}_{D_A(I)/R}\simeq\mathrm{Frob}_A^*\Gamma^i_{A/I}\bigl(I/I^2\bigr)
   	\longrightarrow
   	\mathrm{Frob}_A^*\Gamma^i_{A/I}\bigl(I/I^2\bigr)
   	\otimes_{A}\Omega^1_{A/R}
   	\longrightarrow
   	\mathrm{Frob}_A^*\Gamma^i_{A/I}\bigl(I/I^2\bigr)
   	\otimes_{A}\Omega^2_{A/R}
   	\longrightarrow \cdots	
   \end{equation} 
   is the de Rham complex over \(A\) of the
  \(A\)-module with $R$-linear connection
  $
  \mathrm{Frob}_A^*\Gamma^i_{A/I}\bigl(I/I^2\bigr).
  $
  It thus suffices to show that this de Rham complex as a functor from such pairs \((A,I)\) to \(\mathsf{D}(R)\), is left Kan
  extended from $\mathsf{Pair}_{R,\mathsf{poly}}$. By \cite[Lemma~3.44]{Bha12}, the connection on
  $\mathrm{Frob}_A^*\Gamma^i_{A/I}(I/I^2)$
  is the Frobenius descent connection on the pullback of the
  $A^{(1)}$-module $\iota_A^*\Gamma^i_{A/I}(I/I^2)$
  along the relative Frobenius $\mathrm{Frob}_{A/R}$. The Cartier isomorphism therefore shows that the $j$-th cohomology of
  $\gr_i^{\mathrm{conj}}\Omega^{\mathrm{pd},\bullet}_{D_A(I)/R}$
  is given by
  \[
  \iota_A^*\Gamma^i_{A/I}\bigl(I/I^2\bigr)
  \otimes_{A^{(1)}}
  \Omega^j_{A^{(1)}/R}.
  \]
It follows that $(A\to A/I)\longmapsto\gr_i^{\mathrm{conj}}\Omega^{\mathrm{pd},\bullet}_{D_A(I)/R}$ is left Kan extended from $\mathsf{Pair}_{R,\mathsf{poly}}$. This completes the proof.

\end{proof}

We now clarify a subtle sign issue in this comparison, beginning
with our convention for two-term complexes.
\begin{convention}
	For a map $d:M\to N$, we write $[M\xrightarrow{d}N]_{[a,a+1]}$
	for the two-term complex with $M$ in cohomological degree $a$
	and $N$ in cohomological degree $a+1$.
\end{convention}

\begin{remark}\label{sign1}
 Note that the canonical equivalence $L_{(A/I)/R}\simeq [I/I^2\xrightarrow{d} \Omega^1_{A/R}\otimes_A A/I]_{[-1,0]}$ gives $L_{(A/I)/R}[-1]\simeq [I/I^2\xrightarrow{d} \Omega^1_{A/R}\otimes_A A/I]_{[-1,0]}[-1]= \left[I/I^2\xrightarrow{-d} \Omega^1_{A/R}\otimes_A A/I\right]_{[0,1]}$; see \cite[Tag~011G]{Sta26} for our convention on shifts of cochain complexes. On the other hand, taking $\gr^1$ of the equivalence $ \mathrm{Fil}^*_{\mathrm{Hod}}\dR_{(A/I)/R}\simeq \mathrm{Fil}_{\mathrm{PD}}^{*}\Omega^{\mathrm{pd},\bullet}_{D_A(I)/R}$ for  $(A,I)\in\mathsf{Pair}_R^{\mathsf{sm,reg}}$, we obtain an equivalence $ L_{(A/I)/R}[-1]\simeq [I/I^2\xrightarrow{d} \Omega^1_{A/R}\otimes_A A/I]_{[0,1]}$. These two equivalences are compatible in the sense that the following diagram commutes
\[
\begin{tikzcd}[column sep=large,row sep=large]
	\left[
	I/I^2\xrightarrow{-d}
	\Omega^1_{A/R}\otimes_A A/I
	\right]_{[0,1]}
	\ar[rr,"{(-\mathrm{id},\mathrm{id})}"]
	\ar[dr,"\simeq"{above}]
	&&
	\left[
	I/I^2\xrightarrow{d}
	\Omega^1_{A/R}\otimes_A A/I
	\right]_{[0,1]}
	\ar[dl,"\simeq"{above}]
	\\
	&
	L_{(A/I)/R}[-1].
	&
\end{tikzcd}
\]
Here the $(-\mathrm{id},\mathrm{id})$ is the map of cochain complexes defined by $-\mathrm{id}$ on $I/I^2$ and $\mathrm{id}$ on 	$\Omega^1_{A/R}\otimes_A A/I$. Indeed, by the construction via left Kan extension, it suffices to establish the commutativity of the diagram for $A\to A/I\in \mathsf{Pair}_{R,\mathsf{poly}}$. In this case, the commutativity can be checked directly.
\end{remark}
\section{A Sign issue in totalization}
Let $R$ be a commutative ring and let $M^{\bullet}\colon \Delta\to \mathsf{Ch}(R)$ be a cosimplicial cochain complex. We forget this functor as $ \colon \Delta \to \mathsf{D}(R)$. Since $\displaystyle \lim_{[i]\in \Delta}\colon \Fun(\Delta,\mathsf{D}(R))\to \mathsf{D}(R))$ is an exact functor, we obtain a natural equivalence $\displaystyle (\lim_{[i]\in \Delta} M^i)[-1]\simeq \lim_{[i]\in \Delta} (M^i[-1])$ in $\mathsf{D}(R)$. By Proposition \ref{Tot}, the two objects \((\lim_{[i]\in \Delta} M^i)[-1]\) and \(\lim_{[i]\in \Delta} (M^i[-1])\)  can be represented by $\operatorname{Tot} (M^{\bullet})[-1], \operatorname{Tot} (M^{\bullet}[-1])$, respectively.

\begin{lemma}[{cf.~\cite[Tag~0FLG]{Sta26}}]\label{sign2} The equivalence $\displaystyle (\lim_{[i]\in \Delta} M^i)[-1]\simeq \lim_{[i]\in \Delta} (M^i[-1])$ is represented by an isomorphism of cochain complexes $\operatorname{Tot} (M^\bullet)[-1]\to \operatorname{Tot} (M^\bullet[-1])$, which acts on $M^{i,j}$ by multiplication by $(-1)^i$. Here, we write  $M^{i,j}$ for the degree-$i$ term of the cochain complex $M^i$.
\end{lemma}
\begin{proof}For cochain complex $C$, we can consider a commutative diagram in $\mathsf{Ch}(R)$
	
		\begin{equation}
		\begin{tikzcd}
			C \arrow[r] \arrow[d] & 0 \arrow[d] \\
			\operatorname{Cone}(\mathrm{id}_C) \arrow[r] & C[1],
		\end{tikzcd}
		\end{equation}
	which is a fiber sequence in $\mathsf{D}(R)$.
Applying this construction degreewise to cosimplicial cochain complex $M^{\bullet}\colon \Delta\to \mathsf{Ch}(R)$, we obtain a commutative diagram in $\Fun(\Delta,\mathsf{Ch}(R))$	
	\[
	\begin{tikzcd}
		M^{\bullet} \arrow[r] \arrow[d] & 0 \arrow[d] \\
		\operatorname{Cone}(\mathrm{id}_M) \arrow[r] & M^{\bullet}[1],
	\end{tikzcd}
	\]
	which is a fiber sequence in $\Fun(\Delta,\mathsf{D}(R))$.
We apply totaliztion to this diagram, we obtain a commutative diagram in $\mathsf{Ch}(R)$.
	\[
	\begin{tikzcd}
	\operatorname{Tot}	(M^{\bullet}) \arrow[r] \arrow[d] & 0 \arrow[d] \\
		\operatorname{Tot}(\operatorname{Cone}(\mathrm{id}_{M^{\bullet}})) \arrow[r] & \operatorname{Tot}(M^{\bullet}[1]),
	\end{tikzcd}
	\]
	which is a fiber sequence in $\mathsf{D}(R)$.
	
	On the other hand, we have a commutative diagram in $\mathsf{D}(R)$ by taking $C=\operatorname{Tot}	(M^{\bullet})$ in diagram 12.
	\[
	\begin{tikzcd}
		\operatorname{Tot}	(M^{\bullet}) \arrow[r] \arrow[d] & 0 \arrow[d] \\
		\operatorname{Cone}(\mathrm{id}_{\operatorname{Tot}	(M^{\bullet})}) \arrow[r] & \operatorname{Tot}(M^{\bullet})[1],
	\end{tikzcd}
	\]
	which is a fiber sequence in $\mathsf{D}(R)$.
	
	Note that there is a natural isomorphism
	$
	\alpha\colon
	\operatorname{Tot}\bigl(\operatorname{Cone}(\mathrm{id}_{M^\bullet})\bigr)
	\xrightarrow{\sim}
	\operatorname{Cone}\bigl(\mathrm{id}_{\operatorname{Tot}(M^\bullet)}\bigr).
	$
	Indeed, we have
	$
	\operatorname{Tot}\bigl(\operatorname{Cone}(\mathrm{id}_{M^\bullet})\bigr)^n
	=
	\prod_{i+j=n}
	\left(M^{i,j}\oplus M^{i,j+1}\right),
	$
	and
\[
\begin{aligned}
	\operatorname{Cone}\bigl(
	\mathrm{id}_{\operatorname{Tot}(M^\bullet)}
	\bigr)^n
	&=
	\operatorname{Tot}(M^\bullet)^n
	\oplus
	\operatorname{Tot}(M^\bullet)^{n+1}
	\\
	&=
	\left(\prod_{i+j=n} M^{i,j}\right)
	\oplus
	\left(\prod_{i+j=n+1} M^{i,j}\right)
	=
	\prod_{i+j=n}
	\left(M^{i,j}\oplus M^{i,j+1}\right).
\end{aligned}
\]
	Under the above identification of the underlying graded \(R\)-modules, we define \(\alpha\) on 
	$
	M^{i,j}\oplus M^{i,j+1}
	$
	by
	$
	\alpha(x,y)=(x,(-1)^i y)$.
	Hence, we obtain a commutative diagram 
	\[
	\begin{tikzcd}[column sep=large,row sep=large]
		\operatorname{Tot}	(M^{\bullet}) \ar[r] \ar[d]
		& 0 \ar[d] \ar[ddr,bend left=15] \\
		\operatorname{Cone}(\mathrm{id}_{\operatorname{Tot}(M^{\bullet})}) \ar[r] \ar[drr,bend right=15,"\beta"]
		& \operatorname{Tot}(M^{\bullet})[1] \ar[dr,"f"] \\
		& & \operatorname{Tot}(M^{\bullet}[1]),
	\end{tikzcd}
	\]
	where $f$ acts on on $M^{i,j}$ as multiplication by $(-1)^i$ and $\beta$ defined by the composition $\operatorname{Cone}(\mathrm{id}_{\operatorname{Tot}(M^{\bullet})})\xrightarrow{\alpha^{-1}}\operatorname{Tot}(\operatorname{Cone}(\mathrm{id}_{M^{\bullet}}))\to  \operatorname{Tot}(M^{\bullet}[-1])$. Thus, the map $f\colon \operatorname{Tot}(M^{\bullet})[1]\to \operatorname{Tot}(M^{\bullet}[1])$ represents $\displaystyle (\lim_{[i]\in \Delta} M^i)[1]\simeq \lim_{[i]\in \Delta} (M^i[1])$. Finally, applying this result to the cosimplicial cochain complex $M^{\bullet}[-1]$, then shifting by \([-1]\) and taking the inverse, we obtain the desired isomorphism
	$
	\operatorname{Tot}(M^\bullet)[-1]
	\xrightarrow{\sim}
	\operatorname{Tot}(M^\bullet[-1]),
	$
	which acts on \(M^{i,j}\) as multiplication by \((-1)^i\).
\end{proof}

\printbibliography

\end{document}